\documentclass[11pt,a4paper,twoside]{amsart}

\usepackage{a4wide}
\usepackage[utf8]{inputenc}
\usepackage{amsmath}
\usepackage{amsthm}
\usepackage{mathrsfs}
\usepackage{wasysym}
\usepackage{enumitem}
\usepackage{cite}
\usepackage[footnotesize]{caption}
\usepackage{color}   
\usepackage{wrapfig}
\usepackage{tgschola,eulervm}
\usepackage{comment}
\usepackage{amsfonts}
\usepackage{amssymb}
\usepackage{graphicx}
\usepackage[usenames,dvipsnames]{xcolor}

\newcommand{\heikodetail}[1]{}
\newcommand{\hide}[1]{}

\DeclareMathOperator{\diam}{diam}

\DeclareMathOperator{\dist}{dist}
\DeclareMathOperator{\simp}{\bigtriangleup}
\DeclareMathOperator{\graph}{graph}
\DeclareMathOperator{\lip}{Lip}
\DeclareMathOperator{\bilip}{biLip}
\DeclareMathOperator{\spt}{spt}
\DeclareMathOperator{\Hom}{Hom}

\DeclareMathOperator{\trace}{trace}
\DeclareMathOperator{\im}{im}

\DeclareMathOperator{\without}{\sim}

\DeclareMathOperator{\narrow}{\llcorner}
\newcommand\ang{\mathop{\mbox{$<\!\!\!)$}}\nolimits}

\newcommand{\R}{\mathbb{R}}
\newcommand{\N}{\mathbb{N}}

\renewcommand{\S}{\mathbb{S}}

\newcommand{\D}{\mathbb{D}}
\newcommand{\Gr}{\mathcal{G}}

\newcommand\Id{{{\rm Id}}}

\newcommand{\E}{\mathcal{E}}
\newcommand{\TP}{\mathcal{T}}
\newcommand{\F}{\mathcal{F}}

\newcommand{\DC}{\mathcal{K}}
\newcommand{\Rtp}{R_{\mathrm{tp}}}

\newcommand{\RM}{\mathscr{M}}
\newcommand{\HM}{\mathcal{H}} 
 
\newcommand{\HD}{d_{\mathcal{H}}}

\newcommand{\Kugel}{\mathbb{B}}
\newcommand{\CKugel}{\overline{\mathbb{B}}}

\newcommand{\Ball}[2]{\mathbb{B}(#1,#2)}
\newcommand{\CBall}[2]{\overline{\mathbb{B}}(#1,#2)}

\newcommand{\Sphere}{\mathbb{S}}

\newcommand{\eps}{\varepsilon}

\newcommand{\Cmn}{\mathscr{C}^{1,\alpha}_{m,n}}
\newcommand{\Lmn}{\mathscr{C}^{0,1}_{m,n}}

\newcommand{\Amn}{\mathcal{A}_{m,n}^{\E}}  
\newcommand{\tAmn}{\tilde{\mathcal{A}}_{m,n}^{\E}} 

\definecolor{mygreen}{RGB}{30,50,0}

\newtheorem{thm}{Theorem}[section]
\newtheorem*{thm*}{Theorem}
\newtheorem{prop}[thm]{Proposition}
\newtheorem{lem}[thm]{Lemma}
\newtheorem{cor}[thm]{Corollary}

\newtheorem{introthm}{Theorem}
\newtheorem{introcor}{Corollary}
\newtheorem*{introthm*}{Regularity Theorem}

\theoremstyle{definition}
\newtheorem{defin}[thm]{Definition}
\newtheorem{introdefin}{Definition}
\newtheorem{rem}[thm]{Remark}
\newtheorem*{rem*}{Remark}

\newcommand{\Fo}{\,\,\,\text{for }\,\,}
\newcommand{\Foa}{\,\,\,\text{for all }\,\,}

\newcommand{\AND}{\,\,\,\text{and }\,\,}

\newcommand{\As}{\,\,\,\text{as }\,\,}

\author{S{\l}awomir Kolasi\'{n}ski}
\address[S.~Kolasi{\'n}ski]{
  \newline{}
  Max Planck Institute for Gravitational Physics
  (Albert Einstein Institute)
  \newline{} 
  Am~M{\"u}hlenberg~1,
  D-14476 Golm, Germany}
\email{s.kolasinski@mimuw.edu.pl}

\author{Pawe{\l} Strzelecki}
\address[P.~Strzelecki]{
  \newline{}
  Institute of~Mathematics,
  University of~Warsaw
  \newline{}
  Banacha~2,
  02-097 Warsaw, Poland}
\email{p.strzelecki@mimuw.edu.pl}

\author{Heiko von~der~Mosel}
\address[H.~von~der~Mosel]{
  \newline{}
  Institut f{\"u}r Mathematik,
  RWTH Aachen University
  \newline{}
  Templergraben~55,
  D-52062 Aachen, Germany}
\email{heiko@instmath.rwth-aachen.de}


\date{\today}

\title[Compactness and finiteness of isotopy types
]
{Compactness and isotopy finiteness for submanifolds 
  \\ with uniformly bounded geometric curvature
  energies}

\begin{document}
\frenchspacing

\begin{abstract}  In this paper, we establish compactness 
for various geometric curvature energies including integral Menger curvature,
and tangent-point repulsive potentials, defined a priori on the class of
compact, embedded  $m$-dimensional Lipschitz submanifolds in $\R^n$.  
It turns out that due to a smoothing effect any sequence of 
submanifolds with uniformly bounded energy contains a subsequence 
converging in $C^1$ to a limit submanifold.

This result has two applications. The first one is an 
isotopy finiteness theorem: there are only finitely many
isotopy types of such submanifolds below a given energy value, and we 
provide explicit bounds on the number of isotopy types in terms
of the respective energy. The second one is the lower  semicontinuity 
-- with respect to Hausdorff-convergence of submanifolds --
of all geometric curvature energies under consideration, which can be used 
to minimise each of these energies within prescribed isotopy classes.  
	
\end{abstract}

\maketitle

\section{Introduction}   

\subsection{Introduction and main results}\label{sec:1.1}
In this paper, we prove compactness and isotopy finiteness for several
functionals $\E : \Lmn\to [0,\infty]$ --- we refer to them as \emph{geometric
  curvature energies} --- defined on the class $\Lmn$ of all compact,
$m$-dimensional embedded Lipschitz submanifolds
of $\R^n$.  To reach this goal, we use previously established uniform
$C^{1,\alpha}$-a priori estimates on local graph representations to not only
prove compactness, but also to gain sufficient geometric rigidity, such that two
submanifolds of finite energy that have small Hausdorff-distance are ambiently
isotopic.  As a consequence of this two-fold regularisation of these energies,
we obtain isotopy finiteness: each sub-level set
\begin{gather} 
    \label{Amn} 
    \Amn(E,d):=\left\{\Sigma\in \Lmn : \quad
        \E(\Sigma)\le E, \; \diam \Sigma\le d\right\} 
\end{gather} 
contains only finitely many manifolds up to diffeomorphism but also up to
\emph{isotopy}.  We also give a
crude yet explicit estimate of the number of these isotopy classes. In addition,
we prove lower semicontinuity of all geometric curvature energies with respect
to Hausdorff-convergence, which can be combined with compactness to
minimise each energy in a fixed isotopy class.

The compactness and finiteness theorems for abstract (smooth) Riemannian
manifolds, in different guises and under several sets of assumptions, date back
at least to J.~Cheeger's paper \cite{Che70}. In particular,
\cite[Thm.~3.1]{Che70} states that for $n\not=4$ and any given constant
$C<\infty$ there are only finitely many diffeomorphism types of Riemannian
manifolds $M$ such that
\begin{gather*}
    \big\|\, |K_M|\, \big\|_{L^\infty}^{1/2} \, \cdot \, \text{Vol}\,(M)^{1/n}
    + \frac{\diam M}{\text{Vol}\,(M)^{1/2}} < C\, .
\end{gather*}
The left-hand side is bounded if, for example, the sectional curvature satisfies
$|K_M|\le 1$, the diameter of $M$ is at most $d$ and the volume -- at least $v$;
the lower bound on $\text{Vol}\, M$ can be replaced by a lower bound on the
injectivity radius. Later on, M.~Gromov, see \cite{Gro78} and \cite[Thm.
8.28]{Gro81}, generalised Cheeger's work and introduced the powerful concept of
Gromov--Hausdorff convergence, enabling the study of collapse of sequences of
manifolds with bounded curvature, where in the absence of bounds on the
injectivity radius singularities can appear in the limit. For a proof of
Gromov's compactness theorem with an improvement on the regularity of the
limiting metric, we refer to S.~Peters~\cite{Pet87}. Anderson and Cheeger
\cite{AC92} prove that the space of all Riemannian manifolds with uniform lower
bounds on the Ricci curvature $\text{Ric}_M$ and the injectivity radius, and
uniform upper bounds on the volume, is compact in the $C^\alpha$ topology for
any $\alpha<1$ (meaning $C^\alpha$ convergence of the Riemannian metrics). The
same authors in \cite{AC91} obtain a finiteness theorem for $m$-dimensional
Riemannian manifolds with uniform upper bounds for the diameter and
$|\text{Ric}_M|$, uniform lower bounds for the volume, and uniform bounds for
the scalar curvature in $L^{m/2}$. Newer developments include the papers by 
A.~Petrunin and W.~Tuschmann \cite{TuschmannP-1999},
W.~Tuschmann~\cite{Tuschmann-2002}, and V.~Kapovich, 
A.~Petrunin and W.~Tuschmann
\cite{TuschmannKP-2005}. In particular, Tuschmann \cite{Tuschmann-2002} proves
that the class $\mathcal{M}(n,C,D)$ of simply connected closed $n$-dimensional
Riemannian manifolds with sectional curvature $|K|\le C$ and diameter ${}\le D$
contains finitely many diffeomorphism types provided $n\le 6$ (surprisingly,
this result fails in each dimension $n\ge 7$).

As Cheeger writes in his survey \cite[p. 235]{Che10}, in a passage commenting on
one of the versions of his own finiteness theorem, \emph{Intuitively, the idea
  is that these manifolds can be constructed from a definite numbers of standard
  pieces.} The same comment applies in the present paper, with one notable
difference: all the submanifolds we deal with are embedded in the same $\R^n$,
but their Riemannian metrics $g$ induced by this embedding are, typically,
\emph{only} of class $C^\alpha$ and not better, so that there is no way to
define the classic curvature tensor of $g$. Instead of that, we pick up a family
of geometric `energies' that can be defined without relying on the $C^2$ (or
$C^{1,1}$) regularity of the underlying manifold; a~bound on each of these
energies, combined with a bound on the diameter, yields a bound on the number of
ambient isotopy types (which is stronger than bounding the number of
diffeomorphism types). Each of these energies can, in fact, be defined also for
non-smooth sets, more general than Lipschitz submanifolds. It also can be
minimised in a~given isotopy class.

To state our results precisely,  let us introduce the appropriate 
definitions first. For an
$(m+2)$-tuple $(x_0,x_1,\ldots,x_{m+1})$ of points of $\R^n$, we denote
the $(m+1)$-dimensional simplex with vertices at the $x_i$'s by
$\simp(x_0,x_1,\ldots,x_{m+1})$. The \emph{discrete Menger curvature\/} of
$(x_0,x_1,\ldots,x_{m+1})$ is defined by
\begin{gather}
	\label{def:DC}
    \DC(x_0,\ldots,x_{m+1}) = \frac{\HM^{m+1}(\simp(x_0,\ldots,x_{m+1}))}
    {\diam(\{x_0,\ldots,x_{m+1}\})^{m+2}} \, .
\end{gather}  
For $m=1$, $n=3$ we have 
\begin{gather*}
    \DC(x_0,x_1,x_2) = 
    \frac{\text{Area}(\simp(x_0,x_1,x_2) )}
    {\max(|x_0-x_1|, |x_1-x_2|,|x_2-x_0|)^3}
    \le  \frac 1{4R(x_0,x_1,x_2)},
\end{gather*}
where $R(x_0,x_1,x_2)$ stands for the circumradius\footnote{The triple integral
  over the inverse squared circumradius also known as {\it total Menger curvature}
  was an essential tool in G. David's proof \cite{david_1998}
  of the famous Vitushkin conjecture
  on characterising the one dimensional compact subsets of the complex plane
  that are removable for bounded analytic functions; see, e.g.,  
  \cite{tolsa-book_2014}. The most obvious generalisation to inverse powers of
  circumsphere radii of simplices turns out to be too singular for our purposes
  here; see the discussion in \cite[Appendix B]{SvdM11a}.} of the triangle
$\simp(x_0,x_1,x_2)$.

For a Lipschitz manifold $\Sigma\in \Lmn$, a number $l\in \{ 1,\ldots,m+2\} $, and
$p>0$ we set
\begin{gather}
    \label{def:epl}
    \E_p^l(\Sigma) 
    = \int_{\Sigma^l} \sup_{x_l,\ldots,x_{m+1} \in \Sigma}
    \DC(x_0,\ldots,x_{m+1})^p\ d\HM^{ml}_{x_0,\ldots,x_{l-1}}\, .
\end{gather} 
The integration in \eqref{def:epl} is performed over the product
$\Sigma^l=\Sigma\times\ldots\times\Sigma$ of $l$ copies of $\Sigma$,
with respect to the $m$-dimensional Hausdorff measure $\HM^m$ on each
copy, i.e., with respect to $\HM^{ml}$ on $\Sigma^l$. Before that, one
takes the supremum of $\DC(x_0,\ldots, x_{m+1})$ with respect to all
variables $x_i$ with indices $i\ge l$. (For $l=m+2$, no supremum is
being taken). Please note that formally the integrand is 
undefined on the diagonal of the product. 
However, we tacitly omit this issue: the choice of values of the 
integrand on the diagonal does not affect 
the value of the integral in \eqref{def:epl}, as the diagonal is 
of measure zero in the product. 

In particular, the functional $\E^{m+2}_p$ is called 
the \emph{integral Menger curvature} of $\Sigma$.

Besides all the $\E_p^l$ energies, we consider also two other
functionals that are defined via averaging the inverse powers of the
radii of spheres that are tangent to $\Sigma$ at one point and pass
through another point of $\Sigma$. Namely, we write
\begin{gather}
	\label{def:Rtp}
    \Rtp(x,y) = \frac{|x-y|^2}{2 \dist(y, x + T_x\Sigma)}         
\end{gather} 
to denote the radius of the smallest sphere which passes through $y\in
\Sigma$ and is tangent to the $m$-dimensional affine plane
$x+T_x\Sigma$. (Note that for a Lipschitz manifold $\Sigma\in \Lmn$ the
tangent plane $T_x\Sigma\in G(n,m)$ is indeed well-defined for 
$\HM^m$-almost all $x\in\Sigma$ due to the
classic Rademacher theorem.) Set 
\begin{align} 
	\label{def:TP}
	\TP_p(\Sigma) 
    & = \int_{\Sigma} \int_{\Sigma} \Rtp(x,y)^{-p}\ d\HM^m_x\ d\HM^m_y ,\\
	\label{def:TPG}
    \TP_p^G(\Sigma)   
    & = \int_{\Sigma} \sup_{y \in \Sigma}
    \big{(}\Rtp(x,y)^{-p}\big{)}\ d\HM^m_x\, . 
\end{align}
(Again, in \eqref{def:TP} it does not matter how $1/\Rtp$ is defined 
on the diagonal in $\Sigma\times\Sigma$.)

For each of the energies $\E \in \{ \E_p^l, \TP_p, \TP_p^G \}$, we
write $p_0(\E)$ to denote the scaling invariant exponent for~$\E$.
Since $\DC$ and $1/\Rtp$ scale like the inverse of length, it is easy
to see that $p_0(\E)$ equals the product of $m$ and the number of
integrals of $\Sigma$ in the definition of $\E$. Thus,
\begin{gather}
    \label{def:p0}
    p_0(\E_p^l)=ml, \qquad p_0(\TP_p)=2m,   \qquad p_0(\TP_p^G)=m.
\end{gather}  
\begin{introthm}[\textbf{finiteness of isotopy types}]
    \label{thm:finiteness}
    Let $E,d > 0$ be some numbers. Assume that $\E\in \{\E_p^l,\TP_p,\TP_p^G\}$
    and $p>p_0(\E)$. There are at most $K = K(E,d,m,n,l,p)$ different
    (ambient) isotopy types in~$\Amn(E,d)$.
\end{introthm}

Recall that two topological embedded submanifolds $\Sigma_1,\Sigma_2$ of $\R^n$
are \emph{ambient isotopic} if and only if there exists a continuous map $H :
\R^n \times [0,1] \to \R^n$ such that
\begin{equation}
	\label{isodef}
	\begin{split}
        H_t & =  H(\cdot,t) 
        \quad\mbox{is an embedding for each $t\in [0,1]$} \,,
        \\
        H(x,1) &= x \quad\mbox{for all $x \in \R^n$,} \qquad
        \text{and} \qquad
        H(\Sigma_2 \times \{ 0 \}) = \Sigma_1 \,.
    \end{split}
\end{equation}

\noindent
(Note that the inclusion mapping $f_2:\Sigma_2\to\R^n$ yields one embedding of
$\Sigma_2$ in $\R^n$, and $f_1:=H_0\circ f_2:\Sigma_2\to\R^n$ yields another
one, so that \eqref{isodef} agrees with the definition of \emph{ambient isotopy
  $\tilde{H}:\R^n\times [0,1]\to\R^n\times [0,1]$, $\tilde{H}(x,t):=(H(x,t),t)$
  between the embeddings $f_2$ and $f_1$\/} as in Burde--Zieschang
\cite[p.2]{burde-zieschang-book}.) However, because of the smoothing properties
of all geometric curvature energies described in more detail in Section
\ref{sec:1.2} all Lipschitz submanifolds with finite energy are actually of
class $C^1$, so that Theorem \ref{thm:finiteness} is stronger than
diffeomorphism finiteness since it even bounds the $C^1$-isotopy types.

\begin{rem}
    We do not have an optimal estimate for the number $K = K(E,d,m,n,l,p)$ in
    the above theorem. However, a crude check of constants involved in the
    argument yields
    \begin{equation}
        \label{est:K-finiteness}
        \log \log K \le c(m,n,l,p) \Big( |\log d| + \log \big(E^{1/p}+1\big)+  1 \Big)
    \end{equation}
    with a constant $c(m,n,l,p)$ that blows up for $p\searrow p_0(\E)$. (See
    Section~\ref{sec:finiteness},~Remark~\ref{rem:explicit}). Thus, as expected,
    for fixed dimensions $m$ and $n$, $K$ blows up for $E\to\infty$ (with $\E$
    and $d$ fixed), and for $p\searrow p_0(\E)$ (with $E$ and $d$ fixed).
\end{rem}

It is also worth noting that no lower bounds for the volume (or lower bounds for
the injectivity radius) are needed in our work. Intuitively, the reason is that
the onset of thin tubes or narrow tentacles is being penalised by each of the
energies we consider. The same penalisation effect takes care of a quantitative
embeddedness: if two roughly parallel sheets of an embedded manifold $\Sigma$
are too close to each other preventing $\Sigma$ to be described locally as
\emph{one} graph, then there are lots of roughly regular very small simplices
with vertices on $\Sigma$ (and of small tangent spheres passing through a second
point of $\Sigma$), causing the integrands $\DC$ and $1/\Rtp$ in \eqref{def:DC}
and \eqref{def:Rtp} to be very large on a set of positive measure. We will
describe the energies' quantitative control on local graph patches more
precisely in Section \ref{sec:1.2}.

Our next result states that for any geometric curvature energy
all sublevel sets are sequentially closed and compact with respect to Hausdorff convergence,
and that all these energies are sequentially lower semicontinuous.

\begin{introthm}[\textbf{lower semicontinuity and compactness}]
    \label{thm:lsc}
    For $E,d \in (0,\infty)$ and for a geometric curvature energy $\E~\in~\{ \E_p^l, 
    \TP_p, \TP_p^G \}$, where $p
    > p_0(\E)$ and $l \in \{1,2,\ldots, m+2\}$, the following holds. 
    \begin{enumerate}
    \item[\rm (i)]
        If $\Sigma_j\in\Amn(E,d)$ for all $j\in\N$ and if the $\Sigma_j$ converge
        to a compact set $\Sigma\subset\R^n$ with respect to the Hausdorff-metric as $j\to\infty$,
        then $\Sigma\in\Amn(E,d)$, and moreover,
        \begin{gather*}
            \E(\Sigma) \le \liminf_{j\to \infty} \E(\Sigma_j).
        \end{gather*}
    \item[\rm (ii)]
        For any sequence $(\Sigma_j)_j\subset\Amn(E,d)$  
   with $0\in \Sigma_j$ for all $j$ there is a submanifold $\Sigma\in\Amn(E,d)$ and a subsequence 
        $(\Sigma_{j_k})_k\subset (\Sigma_j)_j$ such that $\HD(\Sigma_{j_k},\Sigma)\to 0$ as $k\to\infty.$
    \end{enumerate}
\end{introthm}

With the energies' quantitative control over local graph patches described in
more detail in Section \ref{sec:1.2} one actually obtains $C^1$-compactness of
these graph patches, which considerably improves the Hausdorff-convergence to
$C^1$-convergence in both parts of Theorem \ref{thm:lsc}. Since, roughly
speaking, isotopy types stabilise under $C^1$-convergence we can use Theorem
\ref{thm:lsc} to deduce the following existence result by means of the direct
method in the calculus of variations.

\begin{introcor}[\textbf{existence of minimisers in isotopy classes}]
    \label{thm:minimizer}
    Let $\E$, $p$, $E$ and $d$ be as in Theorem~\ref{thm:lsc}. 
    For each reference
    manifold $M_0 \in \Amn(E,d)$ there exists $\Sigma_0 \in \Amn(E,d)$ such that
    \begin{gather*}
        \E(\Sigma_0) = \inf \left\{
            \E(\Sigma) : \Sigma \in \Amn(E,d) 
            \ \text{and}\ \Sigma \text{ is ambient isotopic to } M_0
        \right\} \,.
    \end{gather*}
\end{introcor}  

\begin{rem*}
    It is easy to see, via simple covering arguments, that
    Theorems~\ref{thm:minimizer} and \ref{thm:lsc} hold under another set
    of assumptions, with the diameter bounds replaced by volume bounds,
    i.e. with classes $\Amn(E,d)$ replaced by 
    \begin{gather*} 
        \tAmn (E,H):= \{\Sigma\in \Lmn \colon 
        \quad \E(\Sigma)\le E, \; \HM^m(\Sigma)\le H\}\, .
    \end{gather*}
(For $\Sigma$'s with $\E(\Sigma)\le E$, the diameter bounds are equivalent to volume bounds).
\end{rem*}             

\smallskip

There are numerous papers in the literature dealing with compactness
and finiteness results for \emph{immersed} manifolds, starting with J.
Langer \cite{Lan85} who considers immersed surfaces in $\R^3$ of class
$W^{2,p}$ for $p>2$, with $L^p$ bounds on the second fundamental form;
for a generalization to immersed hypersurfaces in $\R^n$ see \cite{Del01}.
G.~Smith \cite{Smi07,Smi12} considers compactness of immersed complete
submanifolds of class $C^{k,\alpha}$ with $k\ge 2$, $\alpha\in (0,1)$,
assuming uniform bounds on the second fundamental forms (and their
derivatives). Recently, P.~Breuning \cite{Bre12} has studied
compactness for a wide class of $(r,\lambda)$--immersions, i.e., $C^1$
immersions that can be represented as $\lambda$-Lipschitz graphs at a
uniform length scale~$r$.

Our work differs from all these papers in that we deal only with \emph{embedded}
objects.  The upper bounds on any of the geometric curvature energies we
consider do guarantee that the limit of a convergent sequence of submanifolds --
even if the convergence, a priori, takes place only in the Hausdorff distance --
is again an {embedded} submanifold; this is due to the penalisation effects
mentioned before. It is easy to see that under the assumptions of \cite{Lan85},
\cite{Smi07} or \cite{Bre12} a sequence of embedded submanifolds might converge
to a limit which is only immersed, not embedded.

The only comparable result
that we are aware of is due to O.~Durumeric:~\cite[Thm.~2]{Dur02} ascertains
that there are only finitely many diffeomorphism and isotopy classes of
connected $C^{1,1}$ manifolds with lower bounds on a very specific functional,
namely on the normal injectivity radius $r_{ni}$, combined roughly speaking with
bounds on volume and diameter.
Our estimate \eqref{est:K-finiteness} on the number of isotopy types is similar
in spirit to Durumeric's \cite[Sec.~5]{Dur02} where $E^{1/p}$ (which controls
the \emph{bending} of $\Sigma$ in a single coordinate chart,
cf. Section~\ref{sec:1.2} below) is replaced by the inverse of the injectivity
radius. (In \cite{Nab95}, A. Nabutovsky studies the intriguing `energy
landscape' of $\E(\Sigma)=\textrm{vol}^{1/n}(\Sigma)/r_{ni}(\Sigma)$ on the set
of $C^{1,1}$ topological hyperspheres $\Sigma=\Sigma^n\subset \R^{n+1}$; in
particular his energy $\E$ has infinitely many distinct local minima.)

Let us note that for \emph{curves} in $\R^3$, J. O'Hara, see
\cite[pp. 1340--1343]{OWR-2013},
mentions a few results that use the same analytic mechanism of proof that we
deal with: sequences of knots $f\colon \S^1\to\R^3$ that are uniformly
bilipschitz and remain bounded in a fixed $C^{1,\alpha}$ space, are precompact
in $C^1$; moreover, the knot class has to be preserved in the limit. A bound for
the bilipschitz constant of $f$ and for its $C^{1,\alpha}$-norm translates into
a bound on the number of possible knot types parameterised by $f$. For $m=1$,
all our energies are valid knot energies of that type: upon fixing the length,
upper bounds on the energy yield bilipschitz and $C^{1,\alpha}$ bounds on the
arc-length parameterisations of knots, resulting in bounds on the number of knot
classes, on the average crossing number, stick number etc.; see \cite{SSvdM13}
for more details. The results of the present paper open several questions in
higher dimensional geometric knot theory, concerning, e.g., the possible
relations between bounds on the energies of `knot conformations' and bounds on
higher dimensional knot invariants.

\subsection{The strategy of proofs and more general results 
  for $\boldsymbol{C^{1,\alpha}}$-submanifolds}\label{sec:1.2}
Let us start by explaining why we assume integrability 
above scale-invariance for
each of the geometric curvature energies.
A simple scaling argument shows that if $\Sigma\subset
\R^n=\R^{n-1}\times \R$ is a cone over an $(m-1)$-dimensional smooth
manifold $\Sigma_0\subset \R^{n-1}$, with vertex $v=(0,\ldots,0,1)$ and
$\E \in \{ \E_p^l, \TP_p, \TP_p^G \}$, then $\E(\Sigma)=\infty$
whenever $p\ge p_0(\E)$. In fact, for such $p$ the sequence of energies
of disjoint pieces of $\Sigma$,
\begin{gather*}
    \epsilon_j
    :=\E\big{(}\Sigma \cap 
    \{x\in \R^n : \quad 2^{-j-1}<|x-v|\le 2^{-j}\}\big{)}, 
    \qquad j=0,1,2,\ldots,
\end{gather*}
is nondecreasing (by scaling!), and we have $\E(\Sigma)\ge
\epsilon_0+\epsilon_1+\cdots$.

Intuitively, each fold or cusp of $\Sigma$ should introduce even `more'
small simplices (or small tangent spheres that contain another point of
$\Sigma$), and thereby should lead to an increase of energy. This
strongly suggests that for any $p>p_0(\E)$ the functional $\E$ should have
nice smoothing properties. This is indeed true; if $\E(\Sigma)$ is
finite for some $p$ above the critical exponent $p_0(\E)$, then
$\Sigma$ is an embedded manifold of class $C^{1,\alpha}$. Moreover,
$\Sigma$ can be assembled from a finite number of \emph{standard graph
  patches} -- corresponding to the \emph{standard pieces\/} in Cheeger's words
quoted in Section \ref{sec:1.1} above 
-- with the size and the graph norms (controlling how much
$\Sigma$ can bend at length scales determined by the energy) explicitly
controlled in terms of $\E(\Sigma)$. This is the reason why we can 
obtain compactness and finiteness results, along with semicontinuity 
of all these energies. 

Here is a precise description of what we mean by \emph{standard graph patches}.

For $\alpha\in (0,1]$, let $\Cmn$ denote the set of all compact,
$C^{1,\alpha}$-smooth, $m$-dimensional embedded submanifolds of~$\R^n$.

\begin{introdefin}[\textbf{$\boldsymbol{C^{1,\alpha}}$-graph patches}]
    \label{class-definition}
    For $R>0, L>0$, $d>0$, and $\alpha\in (0,1]$ we define $\Cmn(R,L,d)$ to be
    the class of those submanifolds $\Sigma\in \Cmn$ that satisfy the following
    three conditions:
    \begin{enumerate}[label=(\roman*), ref=(\roman*)]
    \item \label{cd:i:diam-bounds} \textit{diameter bounds:} 
        $\Sigma\subset \CKugel^n(0,d) ; $ 
    \item \label{cd:i:graph-patches}
        \textit{size of graph patches:} for each point 
        $x\in \Sigma$ there exists a function $f_x : 
        T_x\Sigma \to T_x\Sigma^{\perp}$ of class
        $C^{1,\alpha}$ such that $\Sigma \cap \Ball{x}{R} = 
        (x + \graph(f_x)) \cap \Ball{x}{R}$,
        $f_x(0)=0$, and $Df_x(0)=0$;
    \item \label{cd:i:ctrl-bending}
        \textit{controlled bending:} 
        for each $x \in \Sigma$, we have 
        $\|Df_x(\xi)-Df_x(\eta)\| \le L|\xi-\eta|^\alpha$ 
        for all $\xi,\eta\in T_x\Sigma$, and $\lip(f_x) \le 1$.
    \end{enumerate}
\end{introdefin}

We state below -- in a version that is adapted for our needs in this paper -- a
general regularity result which has been proved in our earlier works, see
\cite{stvdm-JGA} for the case of $\TP_p$, \cite{KStvdM-GAFA} for $\TP_p^G$, and
\cite{Kol12} for all the $\E_p^l$-energies, $l=1,\ldots,m+2$. (The case of
$\E_p^{m+2}$ for $m=2$, $n=3$ dates back to \cite{SvdM11a}; for curves in
$\R^n$, see also \cite{SSvdM10}, \cite{SSvdM09} and \cite{SvdM07}).

\begin{introthm*} 
    Fix $\E \in \{ \E_p^l, \TP_p, \TP_p^G \}$ and $p > p_0(\E)$. Assume that a
    Lipschitz manifold $\Sigma \in \Lmn$ satisfies $\E(\Sigma) \le E <
    \infty$. If $\Sigma\subset \Kugel^n(0,d)$, 
    then $\Sigma \in \Cmn(R,L,d)$ for the exponent $\alpha = 1-p_0(\E)/p\in
    (0,1)$, with $R$ and $L$ depending only on $m,n,l, p,p_0$ and $E$. In fact, one
    can take
    \begin{gather} 
        \label{RLfromEp} R=c_1(m,n,l,p) E^{-1/(p-p_0(\E))}\, ,
        \qquad L = c_2(m,n,l,p) E^{1/p} 
    \end{gather} 
    for some constants $c_1$ and $c_2$ depending only on $m,n,l, $ and $p$. 
\end{introthm*}

Let us note here that according to Blatt and Kolasi\'{n}ski \cite{BK12}, for an
$m$-dimensional embedded $C^1$-submanifold $\Sigma\subset \R^n$ and
$p_1=m(l-1)<p$ (note that $p_0(\E_p^l)=ml>p_1$) the condition
$\E^l_p(\Sigma)<\infty $ is \emph{equivalent} to $\Sigma$ being locally a graph
of a~function in the fractional Sobolev space $W^{1+s,p}$ with $s=1-m(l-1)/p\in
(0,1)$. In combination with the Sobolev imbedding, this implies that the
exponent $\alpha$ in the Regularity Theorem (which is one of the key technical
tools for the present paper) is \emph{best possible}.  Moreover, there are
$m$-dimensional graphs in $\R^n$ with finite curvature energy $\E^l_p$ for which
the graph function is \emph{nowhere} twice differentiable, so we definitely
cannot work with classic curvatures in our setting!

Nevertheless, the Regularity Theorem paves the way to our results on
compactness, finiteness and semicontinuity of geometric curvature energies with
respect to sequences $\Sigma_j\subset\Amn(E,d)$, but also in the more general
subclass of $\Cmn$ introduced in Definition \ref{class-definition}.  The key
idea is the following: due to the regularity estimates, energy and diameter
bounds present in the definition of $\Amn(E,d)$ force all the $\Sigma_j$ to be
in the same, \emph{fixed}, class $\Cmn(R,L,d)$ up to translations.
Then, the \emph{controlled bending} condition satisfied by all the $\Sigma_j$
enters the scene: after a technical preparation involving some \emph{graph
  tilting}, it enables applications of the Arzel\`{a}--Ascoli compactness
theorem in all graph patches. Thus, in fact much more can be said about the
convergence of $\Sigma_j$, at least along a subsequence. Let us make this more
precise.

\begin{introdefin}[\textbf{$\boldsymbol{C^{1,\alpha}}$-convergence 
      of graph patches}]
    \label{def:C1a-conv}
    A sequence $(\Sigma_j)_{j\in\N} \subset \Cmn$ is said 
    to \emph{converge in~$\Cmn$} to
    the set $\Sigma_0 \subset \R^n$ if
    \begin{enumerate}[label=(\roman*), ref=(\roman*)]
    \item
        \label{C1a:i:HD-conv}
        $\HD(\Sigma_j,\Sigma_0) \xrightarrow{j \to \infty} 0$;
    \item
        \label{C1a:i:smoothness} 
        $\Sigma_0$ is a $C^{1,\alpha}$-smooth embedded submanifold of $\R^n$;
    \item
        \label{C1a:i:param-conv}
        there is an index $j_0 \in \N$ and a radius $\rho > 0$ such that for
        each $x \in \Sigma_0$ and for each $j \in \N$ with $j \ge j_0$ or $j=0$
        there exists a function $f_{x,j} \in
        C^{1,\alpha}(T_{x}\Sigma_0,T_{x}\Sigma^{\perp}_0)$ such that
        \begin{gather*}
            \Sigma_j \cap \Ball{x}{\rho} 
            = (x + \graph(f_{x,j})) \cap \Ball{x}{\rho}
        \end{gather*} 
        and
        \begin{gather*}
            \| f_{x,j} - f_{x,0} \|_{C^{1,\alpha'}(T_x\Sigma_0,T_x\Sigma_0^\perp)} 
            \xrightarrow{j \to \infty} 0 
            \qquad\mbox{for each $ \alpha' \in (0,\alpha)$.}
        \end{gather*}
    \end{enumerate}
\end{introdefin}

As suggested above and mentioned in Section \ref{sec:1.1}, the following
stronger compactness result on $\Cmn(R,L,d)$ holds, from which part (ii) of
Theorem \ref{thm:lsc} follows directly, but which is also essential to prove the
other results stated in Section \ref{sec:1.1}.

\begin{introthm}[\textbf{compactness}]
    \label{thm:compactness}
    Let $R,L,d \in (0,\infty)$ and $\alpha \in (0,1]$. 
    Any sequence of  sub\-mani\-folds
    $(\Sigma_j)_{j\in\N} \subset \Cmn(R,L,d)$ contains a~subsequence which
    converges in~$\Cmn$ to some submanifold $\Sigma_0 \in \Cmn(R,L,d)$.
\end{introthm}

The convergence in $\Cmn$ is strong enough to make 
all the $\Sigma_j$ with $j$ large enough 
ambient isotopic to the limiting manifold $\Sigma_0$. 
It also turns out that there is a diffeomorphism of the 
ambient space $J_j: \R^n \to \R^n$ which is close to the identity 
in the bilipschitz sense and maps $\Sigma_j$
to $\Sigma_0$.

\begin{introthm}[\textbf{isotopy and diffeomorphism of ambient space}] 
	\label{thm:diff}
    Let $R,L,d \in (0,\infty)$, $\alpha \in (0,1]$ and let
    $(\Sigma_j)_{j\in\N} \subset \Cmn(R,L,d)$ be a~sequence of submanifolds
    which converges in~$\Cmn$ to~$\Sigma_0 \in \Cmn(R,L,d)$. Then there
    exists $j_0 \in \N$ such that for each $j \ge j_0$ the manifolds
    $\Sigma_j$ and $\Sigma_0$ are ambient isotopic. Moreover, for each $j
    \ge j_0$ there exists a diffeomorphism of the ambient space $J_j : \R^n
    \to \R^n$ such that
    \begin{gather*} 
        J_j(\Sigma_j) = \Sigma_0 \quad \text{and} \quad
        \bilip(J_j) \le 1 + C_J \HD(\Sigma_0,\Sigma_j)^{\alpha/2} \,,
    \end{gather*} 
    where $C_J = C_J(R,L,\alpha,m,n)$.
\end{introthm} 

Here, $\bilip(f)$ denotes the bilipschitz constant of an injective map
$f : X \to f(X) \subset Y$ between two metric spaces $(X,d_X)$ and
$(Y,d_Y)$, i.e. whenever $A \subset X$,
\begin{gather*}
    \bilip(f,A) := \max\{ \lip(f,A), \lip(f^{-1}, f(A)) \} \,,
    \quad
    \bilip(f) := \bilip(f,X) \,.
\end{gather*}
As usual,
\begin{gather*}
    \lip(f,A) = \sup_{x,y \in A,\ x \ne y} \frac{d_Y(f(x),f(y))}{d_X(x,y)}
    \quad \text{and} \quad
    \lip(f) = \lip(f,X)
\end{gather*}
denotes the Lipschitz constant of $f : X \to Y$.

We actually establish an upper bound $\rho$ on 
the  Hausdorff-distance $\HD(\Sigma_1,\Sigma_2)$ 
depending only on the parameters $R,L,\alpha,m,$ and $n$, 
such that  if $\HD(\Sigma_1,\Sigma_2)\le \rho$, 
then $\Sigma_1$ and $\Sigma_2$ are ambiently isotopic, and such 
that a global bilipschitz diffeomorphism $J$ on $\R^n$ 
with $J(\Sigma_2)=\Sigma_1$ exists 
(see Corollary \ref{cor:isotopy} and Lemma \ref{lem:amb-diff}).
This uniform bound leads not only to the proof of Theorem \ref{thm:diff}
but allows us also in the end to give a quantitative
estimate on the number of isotopy types in Theorem \ref{thm:finiteness}.

\medskip

The rest of the paper is organised as follows. In Section 2, we gather simple
preliminary material. In Section~3, after a technical preparation devoted to
graph tilting for functions of class $C^{1,\alpha}$, we prove
Theorem~\ref{thm:compactness}. In Section~4, we construct the isotopies between
the submanifolds $\Sigma_j$ which converge in $\Cmn$, employing a
$C^1$-version of the tubular neighbourhood theorem; parts of this material seem
to be `folklore' but we give the details for the sake of completeness.  This leads
to the proof of Theorem \ref{thm:diff}.
Section~5 contains the proof of  semicontinuity and compactness, Theorem \ref{thm:lsc},
and of Corollary \ref{thm:minimizer},
and the final
Section~6 lays out the explicit estimate for the number of isotopy types
(stated in Theorem~\ref{thm:finiteness}). The whole exposition is more or less
self-contained.

\begin{rem*}
    The letter $C$ will denote a constant whose value may change even in a
    single string of estimates. Subscripted constants (e.g. $C_{l}$,
    $C_{\text{{\rm ang}}}$ etc.)  have global meaning and their value is fixed.  We
    write $C = C(\alpha,\beta,\gamma)$ when $C$ depends \emph{only} on $\alpha$,
    $\beta$ and $\gamma$.
\end{rem*}

\section{Preliminaries}

Most of the notation in the paper is standard. In particular, we use the usual
$\| {\cdot} \|_{C^{1,\alpha}}$ norms, and
\begin{displaymath}
    \HD(E,F) := \sup\{ \dist(y,F) : y \in E \} + \sup\{ \dist(z,E) : z \in F \} 
\end{displaymath}
denotes the Hausdorff distance of sets in $\R^n$.

For a measure $\mu$, we write $f_*\mu$ to denote its push-forward,
and $\spt(\mu)$ to denote its support, cf. Federer \cite[Chapter~2]{Fed69} or
Matilla~\cite[Chapter~1]{Mat95} for definitions.

\subsection{The Grassmannian.}
Throughout the paper, $G(n,m)$ stands for the Grassmannian of all
$m$-dimensional linear subspaces\footnote{Formally, $G(n,m)$ is defined as the
  homogeneous space
  \begin{displaymath}
      G(n,m) := O(n) / (O(m) \times O(n-m)) \,,
  \end{displaymath}
  where $O(n)$ is the orthogonal group; see e.g. A.  Hatcher's
  book~\cite[Section 4.2, Examples 4.53--4.55]{hatcher}.  Thus $G(n,m)$ becomes
  a topological space with the quotient topology.  We work with the angular
  metric, cf.  \eqref{def:angle}.}  of $\R^n$.

For an $m$-plane $U\in G(n,m)$ we let $U^\perp$ be the orthogonal
complement of $U$. The symbol $U_\natural$ denotes the orthogonal
projection of $\R^n$ onto $U$. For $U,V\in G(n,m)$ we set
\begin{equation}
	\label{def:angle}
	\ang(U,V):= \|U_\natural-V_\natural\|.
\end{equation}
This is a metric, and $G(n,m)$ endowed with this metric is compact.

\begin{rem}
    \label{rem:angle-norm}
    Using \cite[8.9(3)]{All72} we have for $U,V \in G(n,m)$
    \begin{gather*}
        \ang(U,V) = \| U_{\natural} - V_{\natural} \| 
        = \| U_{\natural}^{\perp} - V_{\natural}^{\perp} \|
        = \| U_{\natural}^{\perp} \circ V_{\natural} \|
        = \| U_{\natural} \circ V_{\natural}^{\perp} \|
        = \| V_{\natural}^{\perp} \circ U_{\natural} \|
        = \| V_{\natural} \circ U_{\natural}^{\perp} \| \,.
    \end{gather*}
    In particular
    \begin{gather*}
        \ang(U,V) 
        = \sup_{e \in U \cap \Sphere} |V_{\natural}^{\perp} e|
        = \sup_{e \in V \cap \Sphere} |U_{\natural}^{\perp} e| \,.
    \end{gather*}
    Here and later $\Sphere = \{ x \in \R^n : |x| = 1 \}$ denotes the unit
    sphere in~$\R^n$. 
\end{rem}

\begin{lem}
    \label{lem:angle-isomorphism}
    Assume $U,V \in G(n,m)$. If $\ang(U,V) < 1$, then
    \begin{itemize}
    \item $U_{\natural}|_V : V \to U$ is a linear isomorphism,
    \item $U^{\perp} \cap V = \{0\}$,
    \item setting $L = (U_{\natural}|_V)^{-1} : U \to V$ we have
        \begin{gather*}
            \| L \| = (1 - \ang(U,V)^2)^{-1/2} \,.
        \end{gather*}
    \end{itemize}
\end{lem}

\begin{proof}
    If $\ang(U,V) < 1$, then, by Remark~\ref{rem:angle-norm}, for each $v \in V$, $v
    \ne 0$
    \begin{gather*}
        |U_{\natural} v|^2 = |v|^2 (1 - |U_{\natural}^{\perp}(v / |v|)|^2 
        > 0 \,.
    \end{gather*}
    Hence, $\ker U_{\natural}|_V = \{0\}$ and, since $\dim U = \dim V$,
    $U_{\natural}|_V$ is a linear isomorphism. In particular $\ker U_{\natural}
    \cap V = U^{\perp} \cap V = \{0\}$. Observe that, by Remark~\ref{rem:angle-norm},
    \begin{gather*}
        \inf_{e \in V \cap \Sphere} |U_{\natural} e|^2 
        = 1 - \sup_{e \in V \cap \Sphere} |U_{\natural}^{\perp} e|^2
        = 1 - \ang(U,V)^2 \,.
    \end{gather*}
    Set $L = (U_{\natural}|_V)^{-1}$ and compute
    \begin{multline*}
        \| L \| 
        = \sup_{u \in U, u \ne 0} |L u| |u|^{-1}
        = \sup_{v \in V, v \ne 0} |L U_{\natural} v| |U_{\natural} v|^{-1}
        \\
        = \sup_{v \in V, v \ne 0} |v| |U_{\natural} v|^{-1}
        = \left( \inf_{v \in V, v \ne 0} |U_{\natural}(v/|v|)| \right)^{-1}
        = (1 - \ang(U,V)^2)^{-1/2} \,.
        \qedhere
    \end{multline*}
\end{proof}

\begin{rem}
    \label{rem:oblique-proj}
    Let $X \in G(n,m)$ and $Y \in G(n,n-m)$ be such that $\ang(X^{\perp},Y) <
    1$. Then,  by Lemma~\ref{lem:angle-isomorphism},  $X \cap Y = \{ 0 \}$ and
    we can define the \emph{oblique projection} $P : \R^n \to X$ along $Y$,
    i.e., a linear map such that
    \begin{gather}\label{obliqueproj}
        P \circ P = P \,,
        \quad
        \ker P = Y
        \quad \text{and} \quad
        \im P = X \,.
    \end{gather}
    Note that $P$ can also be characterised by the requirement
    \begin{gather}\label{obliqueproj2}
        \{ P v \} = (v + Y) \cap X \,.
    \end{gather}
\end{rem}

\begin{figure}[!ht]
    \begin{center}
        \includegraphics*[totalheight=6.2cm]{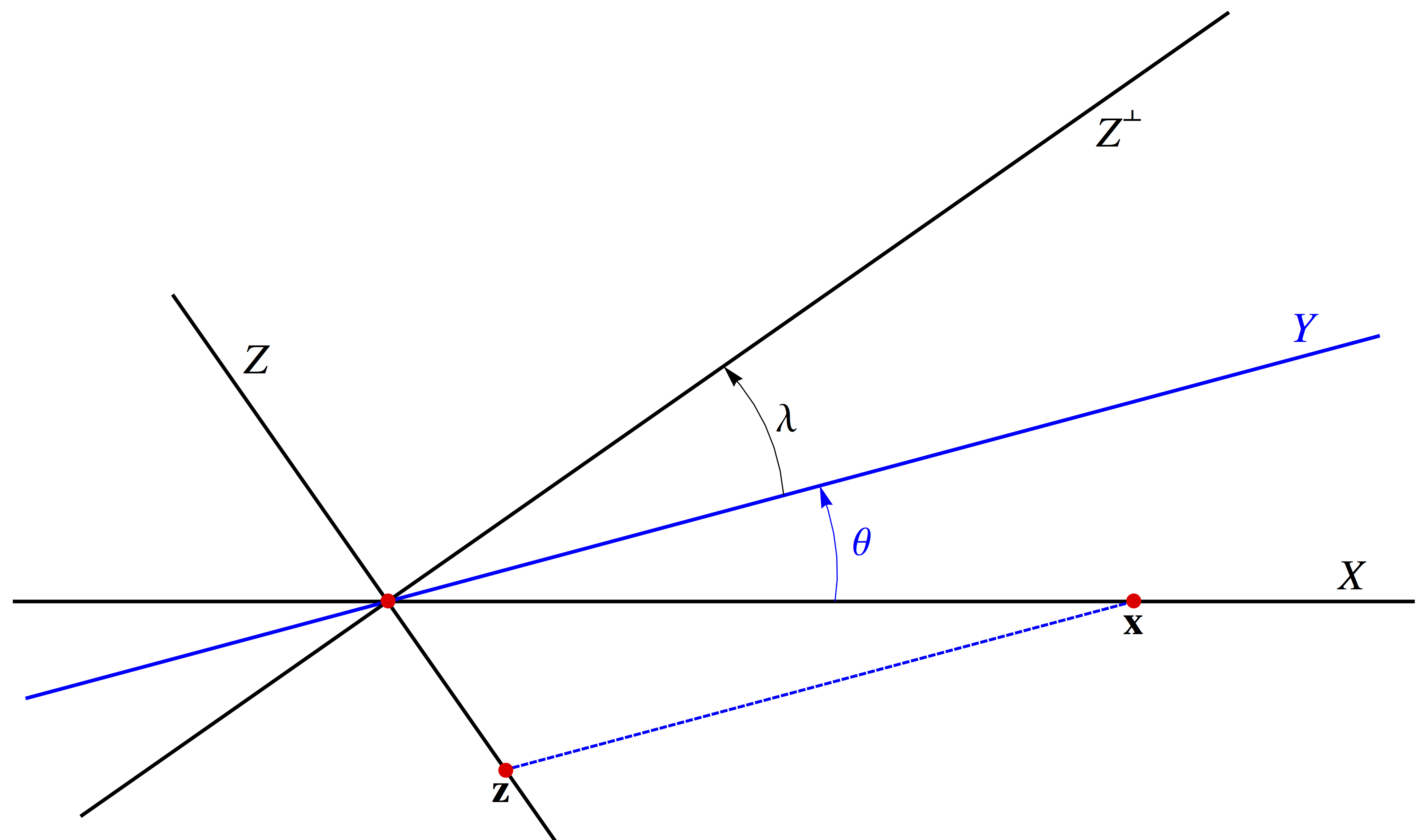}
    \end{center}

    \caption{\label{fig:prop2.3} The situation in Proposition~\ref{prop:tales1}:
      if $\mathbf{x}$, $X,Z$ are fixed, $\theta\to 0$, and
      $\mathbf{z}-\mathbf{x}\in Y$, then $\mathbf{z}\to 0$ and
      $|\mathbf{z}|\lesssim \theta$.}
\end{figure}

\begin{prop}
    \label{prop:tales1}
    Let $\theta \in [0,1]$, $\lambda \in [0,1)$ and $k \in \{ 1, \ldots,
    n-1\}$. Let $X,Y \in G(n,k)$ and $Z \in G(n,n-k)$ be such that 
    $\ang(X,Y) \le \theta$ and $\ang(Y,Z^{\perp}) \le \lambda$. 
    For any $\mathbf{x} \in X$ and $\mathbf{z} \in Z$ 
    with $\mathbf{z}-\mathbf{x} \in Y$ one has the estimate
    \begin{gather*}
        |\mathbf{z}| \le \frac{\theta}{1 - \lambda}|\mathbf{x}| \,.
    \end{gather*}
\end{prop}

\begin{proof}
    Since $Z^\perp_\natural \mathbf{u}= X^\perp_\natural \mathbf{x} =0$, two
    applications of the triangle inequality lead to
    \begin{gather*}
        |\mathbf{z}| \le |Y_{\natural}\mathbf{z}| 
        + |Y^{\perp}_{\natural}\mathbf{z}|
        \le |(Y_{\natural} - U^{\perp}_{\natural})\mathbf{z}| 
        + |Y^{\perp}_{\natural}(\mathbf{z}-\mathbf{x})| 
        + |(Y^{\perp}_{\natural} - X^{\perp}_{\natural})\mathbf{x}|
        \le \lambda |\mathbf{z}| + \theta |\mathbf{x}| \,.
        \qedhere
    \end{gather*}
\end{proof}

\subsection{An elementary topological result}  In a few proofs, we need to rely on the following standard
topological result. For the sake of completeness, we present a proof 
using degree mod 2. (There are of course other proofs, relying on 
the non-existence of the retraction of a ball onto its boundary or,
equivalently, on Brouwer's fixed point theorem.)

\begin{prop}
    \label{prop:deg-one}
    Let $\rho > 0$, $\sigma \in (0,1)$ and 
    $F \in C^0(\CKugel^m(0,\rho), \R^m)$ be such that
    \begin{gather*}
        \quad |F(x) - x| \le \sigma \rho 
        \qquad\mbox{for all }x \in \CKugel^m(0,\rho)\, .
    \end{gather*}
    Then for each $y\in \Kugel^m(0,(1-\sigma) \rho) $ 
    there exists $x \in \CKugel^m(0,\rho) $ such that $F(x) = y$.
\end{prop}

\begin{proof}
    Fix $y\in \Kugel^m(0,(1-\sigma) \rho)$. 
    Assume that $y \notin F(\CKugel^m(0,\rho))$. Then,
    \begin{gather*}
        G : \partial\Kugel^m(0,\rho) \to \partial\CKugel^m(0,\rho)\,,
        \qquad 
        G(z) = \frac{F(z)-y}{|F(z)-y|}\rho 
        \quad\textnormal{for $z\in\partial\CKugel^m(0,\rho)$,}
    \end{gather*}                 
    is well defined and continuous. 
    Since $|F(x) - x| \le \sigma \rho$ for all $x \in \CKugel^m(0,\rho)$, 
    \begin{gather*}
        |-tz + t(F(z)-y)|\le t|F(z)-z|+t|y|
        < t\sigma \rho + t(1-\sigma)\rho < \rho  =|z|
    \end{gather*} 
    for all $z\in \partial\CKugel^m(0,\rho) $ and $t\in [0,1)$; 
    hence 
    $
    (1-t)z+t(F(z)-y)\not= 0$ for all $z\in \partial\Kugel^m(0,\rho) $ 
    and for all $t\in [0,1].$
    Thus, the map
    \begin{gather*}
        H : [0,1] \times \partial\CKugel^m(0,\rho) 
        \to \partial\CKugel^m(0,\rho)\, ,
        \qquad
        H(t,z) = \frac{(1-t)z + t(F(z)-y)}{|(1-t)z + t(F(z)-y)|}\rho 
    \end{gather*}   
    yields a well defined homotopy of $G$ and the identity map 
    on the sphere
    $\partial\CKugel^m(0,\rho) $. Hence $G$ has mod 2 degree $1$. 
    On the other hand, one can
    extend $G$ to the continuous mapping
    \begin{gather*}
        \tilde{G} : \CKugel^m(0,\rho) \to \partial\CKugel^m(0,\rho) \,,
        \qquad
        \tilde{G}(z) = \frac{F(z)-y}{|F(z)-y|}\rho \,.
    \end{gather*}
    Thus $G$ has mod 2 degree $0$, a contradiction. 
    For the relevant results
    on the mod 2 degree one may, e.g., consult 
    \cite[pp. 124,125]{Hir94}. 
\end{proof}

\section{Compactness}    

Each manifold $\Sigma\in \Cmn(R,L,d)$ is assembled from standard graph patches
that have controlled bending at length scales $\lesssim R$. Thus, intuitively,
if two such manifolds $\Sigma_1,\Sigma_2\in \Cmn(R,L,d)$ are sufficiently close
in Hausdorff distance, their tangent planes at points $x\in \Sigma_1$, $y\in
\Sigma_2$ with $|x-y|\lesssim \HD(\Sigma_1,\Sigma_2)$ must be close, too, for
otherwise the Hausdorff distance of the manifolds would be too large. Before
giving the precise quantitative statement, let us mention two simple
consequences of Definition \ref{class-definition} valid for each
$\Sigma\in\Cmn(R,L,d)$: For any $r\in (0,R]$ one finds
$x+v+f_x(v)\in\Kugel^n(x,\sqrt{2}r)$ for all $x\in\Sigma,$ $v\in
T_x\Sigma\cap\Kugel^n(0,r)$, since $|f_x(v)|=|f_x(v)-f_x(0)|\le |v|<r$ so that
$$
|x-(x+v+f_x(v))|^2=|v|^2+|f_x(v)|^2<2r^2.
$$
Secondly, one can improve the estimate for $|f_x(v)|$ for such $x,v,$ and
$r\in (0,R]$ as follows.
\begin{equation}
    \label{improved-f_x-estimate}
    |f_x(v)| = \Big|\int_0^1\frac{d}{dt} f_x(tv)\,dt\Big|
    = \Big|\int_0^1 (Df_x(tv)-Df_x(0))v\,dt\Big|
    \le L|v|^{1+\alpha} <  L r^{1+\alpha}.
\end{equation}

\begin{lem}[\textbf{proximity of tangent planes}]
    \label{lem:hd-angle}
    Let $R,L,d >0$, $\alpha \in (0,1]$, $A\ge 1,$  and $\Sigma_1,\Sigma_2 \in
    \Cmn(R,L,d)$ such that
    \begin{gather*}
        \HD(\Sigma_1,\Sigma_2) 
        < \min\big{\{} 2^{-6} A^{-2} R^2 ,\, L^{-2/\alpha},\, 1 \big{\}},
    \end{gather*}
    and let $x \in \Sigma_1$ and $y \in \Sigma_2$ be such that $|x - y| \le A
    \HD(\Sigma_1,\Sigma_2)$. Then there exists a~constant $C_{\text{{\rm ang}}} =
    C_{\text{{\rm ang}}}(L,A)$
    such that
    \begin{gather}\label{angle-estimate}
        \ang(T_x\Sigma_1,T_y\Sigma_2) 
        \le C_{\text{{\rm ang}}} \HD(\Sigma_1,\Sigma_2)^{\alpha/2} \,.
    \end{gather}
    In fact, one can take $C_{\text{{\rm ang}}}(L,A)=L\big(1+(4A)^{2}\big)+2A$. 
\end{lem}

\begin{proof}
    For $\HD(\Sigma_1,\Sigma_2)=0$ we have $x=y$ and $\Sigma_1=\Sigma_2$ as
    $C^1$-manifolds, so that $T_x\Sigma_1=T_y\Sigma_2$; hence both sides of
    \eqref{angle-estimate} are zero. 
    So, let us assume that $\HD(\Sigma_1,\Sigma_2)>0.
    $
    The following arguments hold for all 
    $u \in T_x\Sigma_1$ with 
    \begin{gather}
        \label{u-condition}
        0< |u| = \HD(\Sigma_1,\Sigma_2)^{1/2} 
        < \min \big{\{} 2^{-3} A^{-1} R ,\, L^{-1/\alpha},\, 1 \big{\}} \,.
    \end{gather}
    Since $\Sigma_1 \in \Cmn(R,L,d)$ we find  
    $$
    p:= x+u+f_x(u) \in \Sigma_1\cap\Ball{x}{R}
    $$ 
    by our remark preceding this lemma, since $|u|<R/(2^3A)<R/\sqrt{2}$. 
    We thus infer
    $$
    (T_x\Sigma_1)_{\natural}(p-x)=(T_x\Sigma_1)_{\natural}(u+f_x(u))=u,
    $$
    and 
    \begin{equation}\label{16}
        \dist(p-x,T_x\Sigma_1)  =  \dist(u+f_x(u),T_x\Sigma_1)=|f_x(u)|
        \overset{\eqref{improved-f_x-estimate}}{\le}  L|u|^{1+\alpha}
        \overset{\eqref{u-condition}}{<}L|u|\frac 1L = |u|.
    \end{equation} 
    In~particular, 
    $$
    |p-x|^2 = |u|^2+|f_x(u)|^2=|u|^2+\dist^2(p-x,T_x\Sigma_1)< 2|u|^2.
    $$
    Next choose a~point $q \in \Sigma_2$ such that $|p-q| \le
    \HD(\Sigma_1,\Sigma_2)$ and set $v := (T_y\Sigma_2)_{\natural}(q-y)$. Then
    one finds
    \begin{align*}
        |v| &\le \|(T_y\Sigma_2)_{\natural}\| |q-y|  
        \, = |q-y| 
        \\
        & \le |q-p|+|p-x|+|x-y|
        \le \HD(\Sigma_1,\Sigma_2)+\sqrt{2}|u|+A\HD(\Sigma_1, \Sigma_2)
        \\
        & \le |u|\Big[ (1+A)|u|+  \sqrt{2}\Big]
        \le A|u|\Big[2|u|+\sqrt{2}\Big]
        \overset{\eqref{u-condition}}{<} 4A|u|
        \overset{\eqref{u-condition}}{<}\frac{1}{\sqrt{2}}R,
    \end{align*}
    so that $y+v+f_y(v)\in\Kugel^n(y,R)$ and $q-y=v+f_y(v)\in\graph(f_y)$, by
    virtue of our remark preceding this lemma.  Employing the identities
    $\dist(p-x,T_x\Sigma_1)=|f_x(u)|=|p-(x+u)|$ and
    $\dist(q-y,T_y\Sigma_2)=|f_y(v)|=|q - (y+v)|$ we can write
    \begin{eqnarray*}
        \dist(u,T_y\Sigma_2) 
        \le |u-v|
        & \le & |u+(x-p)| + |p-q| + |q - (y+v)| + |y-x|  \\
        & \le & |f_x(u)|+\HD(\Sigma_1,\Sigma_2)+|f_y(v)|+A|u|^2\\
        & \overset{\eqref{improved-f_x-estimate}}{\le} & 
        L\big{(} |u|^{1+\alpha} +
        |v|^{1+\alpha} \big{)} +(1+A)\HD(\Sigma_1,\Sigma_2) \\
        & \overset{\eqref{u-condition}}{<} &
        \big[L\big(1+(4A)^{2}\big)+2A\big]   
        |u|^{1+\alpha}\, =:\, C_{\text{{\rm ang}}}(L,A)|u|^{1+\alpha}.
    \end{eqnarray*}
    Since $ \dist(u,T_y\Sigma_2)=|(T_y\Sigma_2)^\perp_\natural
    (u)|=|u|\big|(T_y\Sigma_2)^\perp_\natural
    \big(\tfrac{u}{|u|}\big)\big|=|u|\dist\big(\tfrac{u}{|u|},T_y\Sigma_2\big) $
    we arrive at $ \dist\big(\tfrac{u}{|u|},T_y\Sigma_2\big) \le
    C_{\text{{\rm ang}}}(L,A)|u|^{\alpha}=C_{\text{{\rm ang}}}(L,A)\HD(\Sigma_1,\Sigma_2)^{\alpha/2}.  $
    Since the requirement \eqref{u-condition} on $u \in T_x\Sigma_1$ does not
    depend on the direction $e:=u/|u|\in\S^{n-1}=\S$ we obtain by
    Remark~\ref{rem:angle-norm}
    \begin{gather*}
        \ang(T_x\Sigma_1,T_y\Sigma_2) 
        = \sup_{e \in T_x\Sigma_1 \cap \Sphere}
        |(T_y\Sigma_2)_{\natural}^{\perp} e|
        = \sup_{e \in T_x\Sigma_1 \cap \Sphere} \dist (e,T_y\Sigma_2)
        \le C_{\text{{\rm ang}}} \HD(\Sigma_1,\Sigma_2)^{\alpha/2} \,.
        \qedhere
    \end{gather*}
\end{proof}

To prove Theorem~\ref{thm:compactness}, one applies the Arzel\`{a}--Ascoli
theorem to graph patches of the sequence $\Sigma_j$. To make this possible, it
is necessary to tilt all the graphs (of a subsequence of the $\Sigma_j$'s
intersected with a fixed ball of radius $\approx R$) so that they are all
defined over \emph{the same} plane. Here is a technical lemma that we shall use.

\begin{lem}[\textbf{graph tilting}]
    \label{lem:tilted-graph}
    Let $V \in G(n,m)$, $\alpha \in (0,1]$, $\vartheta \in (0,\frac 1{100})$, $L
    > 0$, $v \in V$, $r \in (0,\infty]$
    and $f \in C^{1,\alpha}(V,V^{\perp})$ is such that $\lip(f) \le 1$ and
    \begin{gather*}
        \| Df(x) - Df(y) \| \le L |x-y|^{\alpha}
        \quad \text{for $x,y \in  V \cap \Ball{v}{r}$}  \, .
    \end{gather*}
    Then the following holds. For each $U \in G(n,m)$ with $\ang(U,V) \le
    \vartheta$  there exists
    a~function $g \in C^{1,\alpha}(U,U^{\perp})$ such that for
    $\omega := U_\natural(v + f(v))$,
    \begin{gather*}
        \graph(f) = \graph(g)
        \\
        \quad \text{and} \quad
        \| Dg(\xi) - Dg(\eta) \| 
        \le L_g(L,\vartheta,\alpha)
        |\xi-\eta|^{\alpha} 
        \quad \text{for $\xi,\eta \in  U \cap \Ball{\omega}{\tfrac{1}{1+3\vartheta}r}$}  \,,
    \end{gather*}
    where $L_g(L,\vartheta,\alpha) := L\, (1+12\vartheta)(1+3\vartheta)^\alpha
    /(1-4\vartheta).$ Moreover, $\lip(g) \le (1 + 2 \vartheta)/(1-2\vartheta)$,
    and $g(0)=0$ if $f(0)=0.$ If $f(0)=0$ and $Df(0)=0$ then
    \begin{equation}\label{Dg0-est}
        \|Dg(0)\|^2\le\frac{\vartheta^2}{1-\vartheta^2}.
    \end{equation}
\end{lem}

\begin{rem*}
    By taking $r=\infty$ we mean $\Kugel(v,r)=\R^n$.
\end{rem*}

\begin{proof}
    If $U=V$ simply set $g:=f$, and we are done. So assume $\ang(U,V)\in
    (0,\vartheta]$ in the following.

    \medskip\noindent
    \emph{Step 1: defining  $g$.}\,\,  Set $\Sigma := \graph(f)$ and for 
    $p_1,p_2
    \in \Sigma$ define  $x_1 := V_{\natural}(p_1)$, $x_2 := V_{\natural}(p_2)$, 
    $z_1
    := U_{\natural}(p_1)$, and $z_2 := U_{\natural}(p_2)$.
    Then $p_i=x_i+f(x_i)$ for $i=1,2,$ and since $\lip(f) \le 1$ and
    $\ang(U,V)\le\vartheta < \frac 1{100}$ we have
    \begin{align}
        \label{eq:zx-comp}
        |(x_2 - x_1) - (z_2 - z_1)| 
        &= |(V_{\natural} - U_{\natural})(p_2 - p_1)|
        \le \|V_{\natural} - U_{\natural}\| |p_2 - p_1|
        \le\vartheta |p_2 - p_1| \\
        &= \vartheta \big|(x_2 - x_1) + (f(x_2) - f(x_1))\big|
        \le 2 \vartheta |x_2 - x_1|< \tfrac{1}{50}|x_2-x_1|,  \notag
    \end{align}   
    where only the very last inequality is restricted to the case
    $x_1\not=x_2.$
    If $z_1=U_\natural(p_1)=U_\natural(p_2)=z_2$ then \eqref{eq:zx-comp}
    implies $0\le (1-2\vartheta)|x_2-x_1|\le 0$, hence
    $x_1=V_\natural(p_1)= V_\natural(p_2)=x_2$, so that
    $$
    p_1=x_1+f(x_1)=x_2+f(x_2)=p_2.
    $$
    In other words, if $p_1\not= p_2$ then $U_\natural(p_1)\not=
    U_\natural(p_2)$, or $U_\natural|_\Sigma:\Sigma\to U$ is injective.

    Setting  $q: = U_\natural (0+f(0)) \in U$,
    \begin{gather*}
        \phi_1 : V \ni x \longmapsto x+f(x) \in \Sigma,
        \qquad
        \phi_2 : \Sigma\ni p\longmapsto U_\natural (p) -  q,
    \end{gather*}
    we find that $\phi:=\phi_2\circ\phi_1:V\to U$ is injective, since both 
    $\phi_1$ and $\phi_2$ are injective. Moreover, $\phi$ is continuous,
    and 
    $
    \phi(0)=\phi_2(\phi_1(0))=\phi_2(0+f(0))=
    0.
    $ 
    Letting $x_1:=0$ in \eqref{eq:zx-comp}, and setting $x:=x_2$, 
    $z:=z_2=\phi(x_2)=\phi(x)$ we infer from \eqref{eq:zx-comp}
    \begin{equation}\label{phi-deviation}
        |x-\phi(x)|\le 2\vartheta |x|\qquad\mbox{for all}\quad x
        =V_\natural(x+f(x))\in V_\natural(\graph(f))=V.
    \end{equation}
    Notice that the restricted projection $V_\natural|_U:U\to V$ is bijective,
    since $\ang(U,V)\le\vartheta<1$. Indeed, since $\dim U=\dim V$ it suffices
    to check that $V_\natural|_U$ is injective. But $V_\natural(u_1)=V_\natural
    (u_2)$ for $u_1,u_2\in U$ with $u_1\not=u_2$ would imply 
    that $0\not= u_1-u_2\in U$ would be contained in the kernel of $V_\natural$, 
    so that we would arrive at the contradictory inequality
    $$
    \|V_\natural-U_\natural\|\ge
    \Big|V_\natural\Big(\frac{u_1-u_2}{|u_1-u_2|}\Big)
    - U_\natural\Big(\frac{u_1-u_2}{|u_1-u_2|}\Big)\Big|
    = \Big|\frac{u_1-u_2}{|u_1-u_2|}\Big|=1
    > \|V_\natural-U_\natural\|>0.
    $$
    To show that $\phi$ is also surjective (hence bijective) consider a
    linear isometry $I_V:V\to\R^m$, define $F\in C^0(\R^m,\R^m)$ to be
    $F:=I_V\circ V_\natural|_U\circ\phi\circ I_V^{-1}$, and estimate
    for $\xi\in\R^m$
    \begin{eqnarray*}
        |\xi-F(\xi)| & = & 
        |I_V\circ V_\natural\big(I_V^{-1}(\xi)\big)
        -I_V\circ V_\natural\circ\phi \big(I_V^{-1}(\xi)\big)|
        =  |V_\natural\big(I_V^{-1}(\xi)\big)
        -V_\natural\circ\phi\big(I_V^{-1}(\xi)\big)|\\
        & \le & |I_V^{-1}(\xi)-\phi\big(I_V^{-1}(\xi)\big)|
        \overset{\eqref{phi-deviation}}{\le}
        2\vartheta|I_V^{-1}(\xi)| = 2\vartheta|\xi|.
    \end{eqnarray*}
    Thus, $F$ satisfies the assumptions of Proposition~\ref{prop:deg-one} for
    each  $\rho>0$ and $\sigma:=2\vartheta<1,$ 
    which implies that $F:\R^m\to\R^m$
    is surjective.  Therefore $V_\natural|_U\circ\phi:V\to V$ is surjective, 
    and finally also $\phi:V\to U$ is surjective, hence bijective.

    We are now in the position to define $g := U_\natural^\perp\circ\phi_1
    \circ\phi^{-1} \circ \tau_{-q}  : U\to U^\perp$,  where $\tau_q(x) =
    x+q$ for $x \in \R^n$ is the usual translation. Since $\phi_1$ is of
    class $C^{1,\alpha}$ one finds $\phi\in C^{1,\alpha}$, and so $\phi^{-1}\in
    C^{1,\alpha}$, and hence $g\in C^{1,\alpha}(U,U^\perp)$; see, e.g.,
    \cite[Section~2.2]{BHS-Sard} for a brief self-contained argument showing
    that the class $C^{1,\alpha}$ is closed under composition and inversion.
    Moreover, if $p \in \Sigma$, $V_{\natural}p = x$ and $U_{\natural}p =
    z$, then $z = \phi(x) +q$ and $x + f(x) = p = z + g(z)$; hence, $\graph(f) =
    \Sigma = \graph(g)$.   In particular, if $f(0)=0$ then $0=0+f(0)=
    z+g(z)$ for $z=U_\natural(0)=0$; hence $g(0)=0.$  Since $x\mapsto x+f(x)$ 
    parameterises $\Sigma=\graph(f)$, 
    one has in the
    point $p=x+f(x)\in\Sigma$ the $m$-dimensional tangent plane $
    T_p\Sigma=(\Id + Df(x))(T_xV)=(\Id+Df(x))(V)$, 
    and likewise, $T_p\Sigma=(\Id+Dg(z))(U)$
    if $p=z+g(z).$ Thus, if $f(0)=0$ and $Df(0)=0$, then we find
    $$
    T_0\Sigma=(\Id+Df(0))(V)
    =V=(\Id+Dg(0))(U)\AND \ang(U,(\Id+Dg(0)(U))=\ang(U,V)\le\vartheta,
    $$
    so that we can apply \cite[8.9(5)]{All72} 
    with $S\equiv S_2:=U$ and $S_1:=(\Id+Dg(0))(U),$ 
    $\eta_2:=0$, $\eta_1:=Dg(0)$, to obtain \eqref{Dg0-est}.

    \medskip\noindent
    \emph{Step 2: Lipschitz continuity of $g$ and oscillation of
      $Dg$.}\,\, By definition,  $z + g(z) = x + f(x)$ for $z =
    \phi(x) + q \in U$, $x \in V$, 
    so that by \eqref{eq:zx-comp} and by the assumption
    $\lip(f) \le 1$,
    \begin{eqnarray*}
        \hspace{-0.3cm}   |g(z_2)-g(z_1)| 
        &=& |(z_2 + g(z_2) - (z_1 + g(z_1)) - (z_2 - z_1)|
        \\
        &=& |(x_2 + f(x_2) - (x_1 + f(x_1)) - (\phi(x_2) - \phi(x_1) |
        \\
        &\le& |f(x_2) - f(x_1)| + |(x_2 - \phi(x_2)) - (x_1 - \phi(x_1))|
        \\
        &\le & |x_2 - x_1| + 
        |U_{\natural} (f(x_1) - f(x_2)) + U_{\natural}^{\perp}(x_2 - x_1)|
        \\
        &\overset{\textnormal{Rem. \ref{rem:angle-norm}}}{\le}& 
        |x_2 - x_1| + \ang(U,V) ( |f(x_1) - f(x_2)| + |x_2 - x_1|)
        \overset{\eqref{eq:zx-comp}}{\le}
        \frac{1 + 2\vartheta}{1 - 2\vartheta}|z_2-z_1| \,. 
    \end{eqnarray*}
    With $
    T_p\Sigma=(\Id+Df(x))(V)$ for $p:=x+f(x)\in\Sigma$ we obtain  for any $v\in
    V, $ $v\not=0$,
    $$
    \Big|V^\perp_\natural\Big(\frac{v+Df(x)v}{|v+Df(x)v|}\Big)\Big|^2
    =\frac{|Df(x)v|^2}{|v|^2+|Df(x)v|^2}
    \le\frac{\|Df\|^2_\infty|v|^2}{|v|^2+\|Df\|_\infty^2|v|^2}
    \le\frac 12,
    $$
    since $\|Df\|_\infty=\lip(f)\le 1,$ and by the fact  that for $c>0$ the 
    function $\xi\mapsto\xi/(c+\xi)$ is non-decreasing on $[0,\infty).$
    Thus, according to Remark \ref{rem:angle-norm},
    \begin{equation}\label{angle-quitesmall}
        \ang(T_p\Sigma^{\perp},V^\perp)
        = \ang(T_p\Sigma,V)\le\frac{1}{\sqrt{2}}<1,
    \end{equation}
    which implies
    \begin{equation}\label{angle-small}
        \ang(T_p\Sigma^\perp,U^\perp)=\ang(T_p\Sigma,U)\le
        \ang(T_p\Sigma,V)+\ang(V,U)\le\frac{1}{\sqrt{2}}+\vartheta <1.
    \end{equation}
    Consequently, the oblique projections $F_p:\R^n\to X:=T_p\Sigma$ along
    $V^\perp$ with $\textnormal{ker\,}F_p=V^\perp$, and
    $G_p:\R^n\to X$ along $U^\perp$ with $\textnormal{ker\,}G_p=U^\perp$
    are well-defined, and satisfy
    $
    \|F_p\|=\max_{e\in\S}|F_p(e)|\le\sqrt{2}<2,
    $
    and  $\|G_p\|<2$, which can be seen as follows. Assume without loss of
    generality $\|F_p\|>0$. Since 
    $$
    |F_p(e)|^2=|F_p(V_\natural(e)) +F_p(V^\perp_\natural(e))|^2
    = |F_p(V_\natural(e))|^2<\|F_p\|^2\Big(|V_\natural(e)|^2    
    + |V^\perp_\natural(e)|^2\Big)=\|F_p\|,
    $$
    if $V^\perp_\natural(e)\not=0.$ But $\S=\S^{n-1}$ is compact, 
    so that there exists
    $e^*\in\S$ (not necessarily unique) with 
    $\|F_p\|^2=|F_p(e^*)|^2$, which
    then necessarily means that $V^\perp_\natural(e^*)=0,$ i.e., 
    $e^*\in\S\cap V.$
    For any such $e^*$ we can write
    $$
    |F_p(e^*)|^2=|e^*+F_p(e^*)-e^*|^2
    \overset{\eqref{obliqueproj2}}{=}
    1+|F_p(e^*)-e^*|^2=1+|V^\perp_\natural(F_p(e^*))|^2,
    $$
    since $F_p(e^*)-e^*\in V^\perp$ and $e^*\in V$; see \eqref{obliqueproj2}.
    Now, with 
    $$
    |V^\perp_\natural(F_p(e^*))|\le\ang(V,X)|F_p(e^*)|
    \overset{\eqref{angle-quitesmall}}{\le}\frac{1}{\sqrt{2}}|F_p(e^*)|
    $$
    one finds $|F_p(e^*)|^2\le 1+\frac 12 |F_p(e^*)|^2,$ which immediately gives
    $\|F_p\|=|F_p(e^*)|\le\sqrt{2}.$ 
    A similar argument for $G_p$ using \eqref{angle-small} instead of
    \eqref{angle-quitesmall} leads to
    $\|G_p\|^2 \le 1+((1/\sqrt{2})+\vartheta)^2\|G_p\|^2$, and hence
    $$
    \|G_p\| \le \frac{1}{\sqrt{1 - ((1/\sqrt{2}) + \vartheta)^2}}
    < \frac{1}{\sqrt{1 - ((1/\sqrt{2})  + (1/100))^2}} <  2.
    $$

    For $z_1,z_2\in U$ and $p_i:=z_i+g(z_i)$, $i=1,2$, let
    $x_1,x_2\in V$ be those unique points with $p_i=x_i+f(x_i)$ for $i=1,2.$
    With $T_{p_i}\Sigma=(\Id+Df(x_i))(V)=(\Id+Dg(z_i))(U)$, 
    and $Df(x_i)(V)\subset V^\perp$, $Dg(z_i)(U)\subset U^\perp$ 
    for $i=1,2$, one obtains for $v\in V$ and $u\in U$
    $$
    v+Df(x_i)v \,,\, u+Dg(z_i)u\in T_{p_i}\Sigma\quad\Fo i=1,2,
    $$
    which implies $v+Df(x_i)v=F_{p_i}(v)$ and 
    $u+Dg(z_i)u=G_{p_i}(u)$ for $i=1,2.$
    Thus, it suffices to estimate
    $$
    \|Dg(z_1)-Dg(z_2)\|=\sup_{e\in U\cap\S}|(G_{p_1}-G_{p_2})(e)|.
    $$
    For any given unit vector $e \in U \cap \Sphere$, we set
    \begin{gather*}
        a_e = F_{p_1}e - F_{p_2}e \in V^{\perp} \,, \quad
        b_e = G_{p_1}e - G_{p_2}e \in U^{\perp} \,, \quad
        c_e = F_{p_1}e - G_{p_1}e \in T_{p_1}\Sigma \,,
    \end{gather*}
    \begin{gather*}
        \text{and} \quad
        \bar{a}_e= F_{p_1}(G_{p_1}e) - F_{p_2}(G_{p_1}e) 
        = G_{p_1}e - F_{p_2}(G_{p_1}e) \in V^{\perp} \,.
    \end{gather*}
    Since $F_{p_1}e \in e + V^{\perp}$, we have $e + V^{\perp} = F_{p_1}e +
    V^{\perp}$, which means that $F_{p_2}(F_{p_1}e) = F_{p_2}e$. 
    In~consequence we may write
    \begin{align}
        \label{eq:bar-a-e}
        |\bar{a}_e|
        &\le |(F_{p_1} - F_{p_2}) (G_{p_1}e - F_{p_1}e)| 
        + |F_{p_1}(F_{p_1}e) - F_{p_2}(F_{p_1}e)| \\
        &\le \|F_{p_1} - F_{p_2}\| |c_e| + |(F_{p_1} - F_{p_2})e| \,. \notag
    \end{align}
	\begin{figure}[!t]
	    \begin{center}
            \includegraphics*[totalheight=9.25cm]{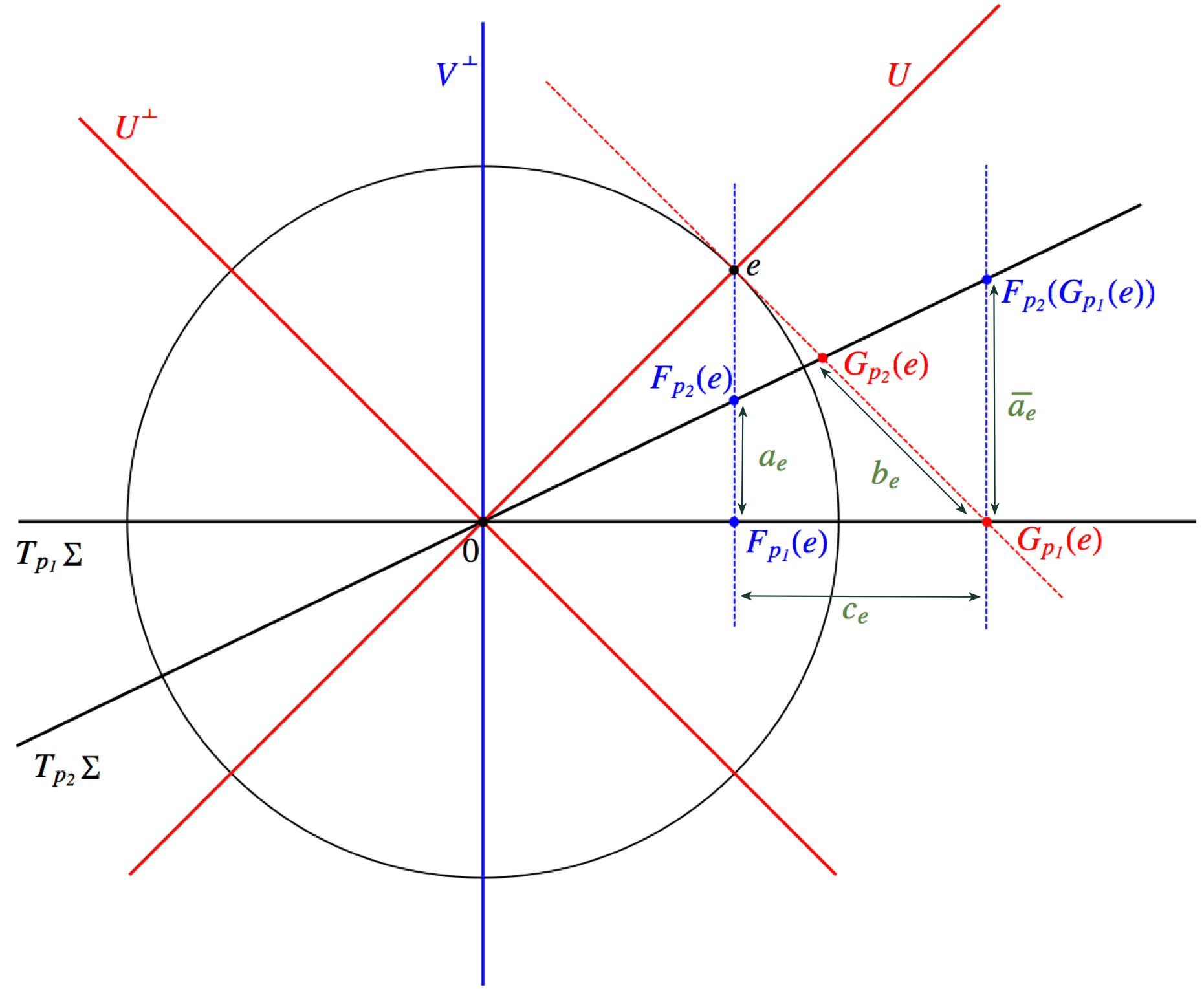}
	    \end{center}

	    \caption{\label{fig:tilting} A unit vector $e \in U \cap \Sphere$ 
        and its corresponding $a_e,b_e,c_e$ and $ \bar{a}_e$.}

	\end{figure}
    Recall that $\ang(U^{\perp},V^\perp)\le\vartheta$ by assumption, and by
    \eqref{angle-small} for $p:=p_1\in\Sigma$ that 
    $$
    \ang(U^\perp,T_{p_1}\Sigma^\perp)=\ang(U,T_{p_1}\Sigma)\le
    \frac{1}{\sqrt{2}}+\vartheta<\frac 34,
    $$
    so that we can apply Proposition \ref{prop:tales1} to the subspaces
    $X:=V^\perp$, $Y:=U^\perp,$ $Z:=T_{p_1}\Sigma$, and to the points
    $x:=F_{p_1}(e)-e\in V^\perp$, and $z:=c_e\in T_{p_1}\Sigma$ (hence
    $z-x=-(G_{p_1}(e)-e)\in U^\perp$) to arrive at
    \begin{equation}\label{eq:c-e}
        |c_e|
        \le 4\vartheta|F_{p_1}(e)-e|
        \le 4\vartheta\big(\|F_{p_1}\|+|e|\big)<12
        \vartheta.
    \end{equation}
    Combining~\eqref{eq:c-e} and~\eqref{eq:bar-a-e} we get
    \begin{gather*}
        |\bar{a}_e| < \|F_{p_1} - F_{p_2}\| (1 + 12 \vartheta) \,.
    \end{gather*}
    Observe that $\bar{a}_e - b_e = 
    G_{p_2}(e) - F_{p_2}\big(G_{p_1}(e)\big) \in
    T_{p_2}\Sigma$. Applying Proposition~\ref{prop:tales1} to 
    $z:=b_e-\bar{a}_e\in T_{p_2}\Sigma=:Z,$ $x:=b_e\in U^\perp=:X$ (with
    $z-x=-\bar{a}_e\in V^\perp=:Y$, so that
    $\ang(X,Y)=\ang(U^\perp,V^\perp)\le\vartheta=:\theta,$
    and
    $\ang(Y,Z^\perp)=\ang(V,T_{p_2}\Sigma)\le 1/\sqrt{2}=:\lambda<3/4$ by 
    \eqref{angle-quitesmall}) yields $|b_e-\bar{a}_e|<4\vartheta|b_e|,$
    and in consequence,
    $
    |G_{p_1}(e)-G_{p_2}(e)|=|b_e|\le |b_e-\bar{a}_e|+|\bar{a}_e|\le
    4\vartheta |b_e|+|\bar{a}_e|,
    $
    i.e., 
    $$
    |b_e|
    \le \tfrac{1}{1-4\vartheta}|\bar{a}_e|
    \le \tfrac{1+12\vartheta}{1-4\vartheta}
    \|F_{p_1}-F_{p_2}\|.
    $$
    Since $e \in U\cap\S^{n-1}$ was arbitrary, we  conclude
    \begin{gather}
        \label{eq:Dg-Holder}
        \|Dg(z_1) - Dg(z_2)\| 
        \le \tfrac{1 + 12 \vartheta}{1 - 4 \vartheta}\|F_{p_1}-F_{p_2}\| =
        \tfrac{1 + 12 \vartheta}{1 - 4 \vartheta}\|Df(x_1)-Df(x_2)\| \,. 
    \end{gather}
    At this point we know already that $\Sigma=\graph(f)=\graph(g)$ and that
    \begin{equation}\label{lipg}
        \lip(g)\le\frac{1+2\vartheta}{1-2\vartheta}.
    \end{equation}
    Exchanging~$U$ with~$V$ and~$f$ with~$g$, 
    and using \eqref{lipg} (instead of $\lip(f)\le 1$)
    in the derivation of~\eqref{eq:zx-comp} one shows
    \begin{gather}
        \label{eq:x2x1-z2z1-comp}
        |(x_2 - x_1) - (z_2 - z_1)| 
        \le \vartheta \left( 1 + \tfrac{1 + 2\theta}{1-2\theta} \right) |z_2 - z_1| 
        \le 3 \vartheta |z_2 - z_1| \,,
    \end{gather}
    whenever $x_i = V_\natural(p_i)$, $z_i = U_\natural(p_i)$ and $p_i = z_i +
    g(z_i) \in \Sigma$ for $i=1,2$. Set $s = \frac{1}{1+3\vartheta}r$. If $v\in
    U$, $x = V_\natural(v + g(v))$, $z_i \in U \cap \Ball{v}{s}$ and $x_i =
    V_\natural(z_i + g(z_i))$ for $i = 1,2$, then $|x_i - x| \le (1 + 3
    \vartheta)|z_i - v| < (1 + 3 \vartheta) s = r$
    by~\eqref{eq:x2x1-z2z1-comp}. Hence, employing~\eqref{eq:Dg-Holder}, 
    \eqref{eq:x2x1-z2z1-comp}, and our assumption on the
    oscillation of $Df$ on $V\cap\Kugel(v,r)$
    one~obtains
    \begin{multline*}
        \|Dg(z_1) - Dg(z_2)\|
        \overset{\eqref{eq:Dg-Holder}}{\le} \tfrac{1 + 12 \vartheta}{1 - 4 \vartheta}\|Df(x_1) - Df(x_2)\| 
        \\
        \le \tfrac{1 + 12 \vartheta}{1 - 4 \vartheta} L |x_1 - x_2|^{\alpha}
        \overset{\eqref{eq:x2x1-z2z1-comp}}{\le}
        \tfrac{(1 + 12 \vartheta)(1+3\vartheta)^\alpha}{1 - 4 \vartheta} 
        L |z_1 - z_2|^{\alpha} \,,
    \end{multline*}
    for all $z_1,z_2 \in U \cap \Ball{v}{\frac{1}{1+3\vartheta}r}$.
\end{proof}

\begin{proof}[Proof of Theorem~\ref{thm:compactness}]
    Applying Blaschke's selection theorem (cf.~\cite{Pri40}) to the
    sequence~$(\Sigma_j)_{j}$ contained in the class $\Cmn(R,L,d)$ we obtain
    a~subsequence (still denoted by $(\Sigma_j)_j$) that converges in
    the~Hausdorff metric to a~compact set~$\Sigma_0 \subset
    \CBall{0}{d}\subset\R^n$.
    Up to possibly choosing another subsequence we can assume that
    $\HD(\Sigma_0, \Sigma_{j+1}) \le \frac 12 \HD(\Sigma_0, \Sigma_j)$ for each
    $j \in \N \without \{0\}$. Then, by induction,
    \begin{gather*}
        \HD(\Sigma_0, \Sigma_{j+k}) \le 2^{-k} \HD(\Sigma_0, \Sigma_{j})
    \end{gather*}
    and in consequence, by the triangle inequality,
    \begin{gather}
        \label{eq:j-j+k}
        \HD(\Sigma_j, \Sigma_{j+k}) \ge (1 - 2^{-k}) \HD(\Sigma_0, \Sigma_{j})
    \end{gather}
    for $j,k \in \N \without \{0\}$.

    We will prove in the first step that $\Sigma_0\in\Cmn(R,L,d)$, that is,
    $\Sigma_0$ satisfies Definition~\ref{class-definition}, and then in the
    second step that the sequence $(\Sigma_j)_{j \in \N}$ converges to
    $\Sigma_0$ in $\Cmn$, i.e., converges in the sense of
    Definition~\ref{def:C1a-conv}.  Note that condition~\ref{cd:i:diam-bounds}
    of Definition~\ref{class-definition} for~$\Sigma_0$ and
    condition~\ref{C1a:i:HD-conv} of Definition~\ref{def:C1a-conv}
    for~$(\Sigma_j)_{j \in \N}$ are automatically satisfied.

    \emph{Step 1: $\Sigma_0\in\Cmn(R,L,d)$.}\,\, 
    The convergence of the $\Sigma_j$ in the Hausdorff metric implies that there
    is an index $j_0 \in \N$ such that
    \begin{gather}
        \label{eq:hd-bound}
        \HD(\Sigma_j,\Sigma_l) 
        < \min \left\{ 2^{-10} R^2 \,,\, L^{-2/\alpha} \,,\, 1 \right\}
        \quad \forall j,l \ge j_0 \,.
    \end{gather}
    Fix a point in the limit set $x_0 \in \Sigma_0$. Choose points $x_j \in
    \Sigma_j$ which realise the distance from $x_0$, i.e.
    \begin{gather}
        \label{eq:xj-dist}
        |x_0 - x_j| = \dist(x_0, \Sigma_j) \le \HD(\Sigma_0, \Sigma_j)
        \quad \text{for $j \in \N$.}
    \end{gather}
    Then $x_j \to x_0$ as $j \to \infty$ and
    \begin{equation}    
        \label{eq:xj-xj+k}    
        \begin{split}    
            |x_j - x_{j+k}| \le |x_j - x_0| + |x_0 - x_{j+k}|
            &\overset{\eqref{eq:xj-dist}}  {\le} 2 \HD(\Sigma_0, \Sigma_j) \\
            &\overset{\eqref{eq:j-j+k}} {\le} 4 \HD(\Sigma_j, \Sigma_{j+k})
            \quad\Foa j,k\ge 1.
        \end{split}
    \end{equation}
    Recalling~\eqref{eq:hd-bound}, we may apply Lemma~\ref{lem:hd-angle}
    with $A:=4$ to deduce, for the above $x_j \in \Sigma_j$ and $x_l \in
    \Sigma_l$ with $j,l \ge j_0$, the angle estimate
    \begin{equation}\label{angleconv}
        \ang(T_{x_j}\Sigma_j,T_{x_l}\Sigma_l)
        \le C_{\text{{\rm ang}}}(L,4)\HD(\Sigma_j,\Sigma_l)^{\alpha/2}\,.
    \end{equation} 
    Consequently, $(T_{x_j}\Sigma_j)_j$ is a Cauchy sequence in $G(n,m)$ that
    converges to some $T\in G(n,m),$ i.e.,
    \begin{equation}\label{conv-tangentplanes}
        \vartheta_j:= \ang(T_{x_j}\Sigma_j, T) \to 0 \As j \to \infty \,.
    \end{equation}

    Recall that by Definition \ref{class-definition} we find for each $j\in\N$ a
    function $f_j\in C^{1,\alpha}(T_{x_j}\Sigma_j,T_{x_j}\Sigma_j^\perp)$
    with $f_j(0)=0$, $Df_j(0)=0$, such that
    \begin{equation}\label{graphcont}
        \Sigma_j\cap\Kugel(x_j,R)=(x_j+\graph(f_j))\cap\Kugel(x_j,R),
    \end{equation}
    with the uniform estimates
    $\lip(f_j)\le 1$ and $\|Df_j(x)-Df_j(y)\|\le L|x-y|^\alpha$ 
    for all $x,y\in T_{x_j}\Sigma_j$.
    We can assume by \eqref{angleconv} that $\vartheta_j\in (0,1/100)$ 
    for all $j\ge j_0$,
    so that Lemma \ref{lem:tilted-graph} applied to 
    the radius $r=\infty$ leads to functions
    $g_j\in C^{1,\alpha}(T,T^\perp)$ such that $\graph(f_j)=\graph(g_j)$, 
    $g_j(0)=0$ and
    \begin{equation}\label{Dgj-est}
        \|Dg_j(\xi)-Dg_j(\eta)\|\le L_j|\xi-\eta|^\alpha
        \quad\textnormal{for all $ \xi,\eta\in T$ and $j\ge j_0,$}
    \end{equation}
    where $L_j:=L_{g_j}(L,\vartheta_j,\alpha)\to L$ as $j\to\infty.$ Moreover,
    \begin{equation}\label{lipgj}
        \lip(g_j)\le\frac{1+2\vartheta_j}{1-2\vartheta_j}\to 1 
        \quad\textnormal{on $T$},
        \AND \|Dg_j(0)\|^2\le\frac{\vartheta_j^2}{1-\vartheta_j^2}\to
        0\quad\textnormal{as $j\to\infty$.}
    \end{equation}
    In addition, \eqref{graphcont} translates into
    \begin{equation}\label{graphcontgj}
        \Sigma_j\cap\Kugel(x_j,R)=(x_j+\graph(g_j))\cap\Kugel(x_j,R)\quad\Foa j\ge j_0.
    \end{equation}
    Because of the uniform estimates \eqref{Dgj-est} and \eqref{lipgj} 
    we can repeatedly apply Arzela-Ascoli's theorem to  
    successively choose subsequences 
    $(j_{i+1})_{i+1}\subset (j_{i})_{i}$ for $i\in\N$, such that 
    $g_{j_i}$ converges in $C^1$ to a function
    $G_i\in C^{1,\alpha}(T\cap \CKugel(0,i),T^\perp)$ with $G_i(0)=0$, such that
    \begin{equation}\label{DGi-est}
        \|DG_i(\xi)-DG_i(\eta)\|\le L|\xi-\eta|^\alpha
        \quad\textnormal{for all $ \xi,\eta\in T\cap \CKugel(0,i),$ }
    \end{equation}
    and with
    \begin{equation}\label{lipGi}
        \lip(G_i)\le 1 \quad\textnormal{on $T\cap \CKugel(0,i)$}\AND DG_i(0)=0.
    \end{equation}
    In addition, one has $G_{i+1}|_{\CKugel(0,i)}=G_i$ for all $i\in\N.$
    Then the diagonal sequence $g_{j_j}$ converges 
    in $C^1_\textnormal{loc}(T,T^\perp)$ to some
    limit function $G\in C^{1,\alpha}(T,T^\perp)$ satisfying 
    $G(0)=0$, $DG(0)=0$, $\lip(G)\le 1$ on $T$, and the estimate
    \begin{equation}\label{DG-est}
        \|DG(\xi)-DG(\eta)\|\le L|\xi-\eta|^\alpha
        \quad\textnormal{for all $ \xi,\eta\in T.$ }
    \end{equation}
    Applying \eqref{graphcontgj} to the diagonal sequence 
    $(g_{j_j})_{j_j}\subset (g_j)_j$ combined with
    \eqref{eq:xj-dist} one finds
    \begin{equation}\label{graphcontG}
        \Sigma_0\cap\Kugel(x_0,R)=(x_0+\graph(G))\cap\Kugel(x_0,R),
    \end{equation}
    which concludes Step 1 since $\Sigma_0$ is represented near the arbitrarily
    chosen point $x_0 \in \Sigma_0$ as a graph of the $C^{1,\alpha}$ function $G
    : T \to T^\perp$ satisfying all the requirements in
    Definition~\ref{class-definition} observing, in addition, that since $x
    \mapsto x_0 + x + G(x)$ parameterises $\Sigma_0$ locally near $x_0$ one has
    $T_{x_0}\Sigma_0=(\Id+DG(0))(T)=T$ s ince $DG(0)=0$, which a posteriori
    shows that the $m$-plane $T$ does not depend on the sequence $x_j\to x_0$.

        \emph{Step 2:  $\Sigma_j$ converges in $\Cmn$ to $\Sigma_0$.
        }\,\, 
        It suffices  to check  condition (iii) of Definition~\ref{def:C1a-conv}.
        Let $j_2 \in \N$, $j_2 \ge 200$ be such that
        \begin{gather}
            \label{eq:j2-choice}
            \HD(\Sigma_j,\Sigma_0) 
            < \min \left\{
                2^{-10} R^2 \,,\, L^{-2/\alpha} \,,\, 2^{-8} R \,,\, \big( 2^{-7} C_{\text{{\rm ang}}}(L,4)^{-1}\big)^{2/\alpha} \,,\, 2^{-8}
            \right\}
            \quad \text{for $j \ge j_2$} \,.
        \end{gather}
        Fix $x \in \Sigma_0$ and set $T = T_{x}\Sigma_0$. As before for each $j
        \in \N$ find $x_j \in \Sigma_j$ such that $|x - x_j| = \dist(x,\Sigma_j)
        \le \HD(\Sigma_0,\Sigma_j)$ and let $f_j : T_{x_j}\Sigma_j \to
        T_{x_j}\Sigma_j^{\perp}$  and $f_0:=f_x:T\to T^\perp$ be the functions
        whose existence is guaranteed by condition~\ref{cd:i:graph-patches} of
        Definition~\ref{class-definition}. 
        According to Lemma~\ref{lem:hd-angle} (generously for  $A=4$), 
        we get by \eqref{eq:j2-choice}
        \begin{gather}
            \label{eq:x0-xj-angle}
            \ang(T_{x}\Sigma_0, T_{x_j}\Sigma_j) \le C_{\text{{\rm ang}}}(L,4)
            \HD(\Sigma_0,\Sigma_j)^{\alpha/2} \overset{\eqref{eq:j2-choice}}{<}
            2^{-7} < \tfrac 1{100} 
            \quad \text{for $j \ge j_2$.}
        \end{gather}
        An  application of Lemma~\ref{lem:tilted-graph}  
        yields now functions $g_j \in
        C^{1,\alpha}(T,T^\perp)$ such that $\graph(g_j) = \graph(f_j)$ and
        $\lip(g_j) \le \frac{51}{49} < 2$ for each $j \ge j_2$, or $j=0$
        where $g_0=f_0$. Set
        $h_j(\eta) = g_j(\eta - T_\natural(x_j - x)) + T^{\perp}_\natural(x_j -
        x)$ for $\eta \in T$ and $j \ge j_2$, and for $j=0$ we have $h_0=g_0=f_0$,
        so that $x_j + \graph(g_j) =
        x + \graph(h_j)$ and consequently, recalling~\eqref{eq:j2-choice},
        \begin{gather}
            \label{eq:hj-graph}
            \Sigma_j \cap \Ball{x}{(1 - 2^{-8}) R} = 
            (x + \graph(h_j)) \cap \Ball{x}{(1 - 2^{-8}) R}
            \quad \text{for $j \ge j_2$ or $j=0$.}
        \end{gather}
        Set  $\rho := \min\{\frac 1{12} R,\frac 12 (2^{-7}/L)^{1/\alpha}\} $ 
        and note that
        \begin{gather}
            \label{eq:hj-graph-rho}
            \{ x + \eta + h_j(\eta) : \eta \in T \cap \Ball{0}{3\rho} \} \subseteq
            \Sigma_j \quad \text{for $j \ge j_2$ or $j=0$,}
        \end{gather}
        because $\lip(g_j) = \lip(h_j) < 2$. Let $\eta \in T \cap
        \Ball{0}{2\rho}$, $j \in \N$, $j \ge j_2$ and set $p := x + \eta +
        h_0(\eta) \in \Sigma_0$ and~$q:= x + \eta + h_j(\eta) \in
        \Sigma_j$. There exists $z \in \Sigma_j$ with $|z - p| \le \HD(\Sigma_0,
        \Sigma_j)$. By~\eqref{eq:j2-choice} we have $\HD(\Sigma_0, \Sigma_j) <
        \rho$, and if we write $z=x+\xi+h_j(\xi)$, then
        $\eta - \xi= T_\natural(p-x)-T_\natural(z-x)$ so that $|\eta - \xi| \le
        |z-p| \le \HD(\Sigma_0, \Sigma_j)$, and therefore $\xi \in T \cap
        \Ball{0}{3\rho}$.
        Since $\lip(h_j) < 2$
        we obtain 
        \begin{multline}
            \label{eq:hj-unif}
            |h_0(\eta) - h_j(\eta)| = |p-q| 
            \le |p - z| + |z - q|
            \\
            \le \HD(\Sigma_0, \Sigma_j) + |\eta - \xi| + |h_j(\eta) - h_j(\xi)|
            < 4 \HD(\Sigma_0, \Sigma_j) \,.
        \end{multline}
        We already know that $\Sigma_0 \in \Cmn(R,L,d)$ so employing
        Lemma~\ref{lem:hd-angle}  for $A=4$, we get
        \begin{gather}\label{TpTq}
            \ang(T_{p}\Sigma_0, T_{q}\Sigma_j) 
            \le C_{\text{{\rm ang}}}(L,4) \HD(\Sigma_0, \Sigma_j)^{\alpha/2}
            \quad \text{ for $j \ge j_2$ or $j=0$.}
        \end{gather}
        Apply~\cite[8.9(5)]{All72} with $\eta_1:= Dh_j(\eta)$, $\eta_2 := Dh_0(\eta)$, $S_1:=T_q\Sigma_j$,
        $S_2:=T_p\Sigma_0$,
        and $S := T$ to obtain, recalling~\eqref{eq:x0-xj-angle} and $\lip(h_0) =
        \lip(f_0) \le 1$,
        \begin{multline}
            \label{eq:Dhj-uni}
            \| Dh_j(\eta) - Dh_0(\eta) \|^2 
            \le \frac{\ang(T_{q}\Sigma_j, T_{p}\Sigma_0)^2}
            {1 - \ang(T_{q}\Sigma_j, T)^2}
            \left( 1 + \|Dh_0(\eta)\|^2 \right)
            \\
            \le \frac{2}{ 1 - \ang(T_{q}\Sigma_j, T)^2}
            C_{\text{{\rm ang}}}(L,4)^2 \HD(\Sigma_j,\Sigma_0)^{\alpha} \,.
        \end{multline}
        To analyze the term in the denominator we
        estimate using \eqref{TpTq} and \eqref{eq:j2-choice}
        $$
        \ang(T_q\Sigma_j,T)
        \le \ang(T_q\Sigma_j,T_p\Sigma_0) + \ang(T_p\Sigma_0,T)
        \overset{\eqref{TpTq},\eqref{eq:j2-choice}}{\le} 2^{-7} + \ang(T_p\Sigma_0,T).
        $$
        For the last summand we again use \cite[8.9(5)]{All72}, this time for
        $S:=T$, $\eta_1:=Dh_0(\eta)$, $S_1:=T_p\Sigma_0=(\Id+ \eta_1)(T)$,
        $\eta_2:=0$, and $S_2:=(\Id+\eta_2)(T)=T$, to deduce by virtue of
        $Dh_0(0)=Df_0(0)=0$ the angle estimate
        $$
        \ang(T_p\Sigma_0,T)
        \le \|Dh_0(\eta)\|\le L|\eta|^\alpha
        \overset{\eqref{eq:j2-choice}}{\le} 2^{-7}
        $$
        by our choice of $\rho$. Therefore we can insert the resulting estimate
        $\ang(T_q\Sigma_j, T)^2\le 2^{-12}$ into \eqref{eq:Dhj-uni} to obtain
        \begin{equation}
            \label{eq:Dhj-unif}
            \| Dh_j(\eta) - Dh_0(\eta) \|^2 
            \le 3C_{\text{{\rm ang}}}(L,4)^2 \HD(\Sigma_j,\Sigma_0)^{\alpha}.
        \end{equation}

        Since $\eta \in T \cap \Ball{0}{2\rho}$ and $j \ge j_2$ were chosen
        arbitrarily, the estimates~\eqref{eq:hj-unif} and~\eqref{eq:Dhj-unif}
        hold for any $\eta \in T \cap \Ball{0}{2\rho}$ and $j \ge j_2$. Fix a
        smooth cutoff function $\varphi : T \to \R$ such that $\varphi(\eta) =
        1$ for $\eta \in T \cap \Ball{0}{\rho}$ and $\varphi(\eta) = 0$ for
        $\eta \in T \without \Ball{0}{2\rho}$. For $j \ge j_2$ and for $j=0$
        define $f_{x,j} \in C^{1,\alpha}(T,T^\perp)$ by $f_{x,j}(\eta) :=
        h_j(\eta) \varphi(\eta)$ for $\eta \in T \cap \Ball{0}{2\rho}$ and
        $f_{x,j}(\eta) = 0$ for $\eta \in T \without \Ball{0}{2\rho}$.
        Estimates~\eqref{eq:hj-unif} and~\eqref{eq:Dhj-unif} show that the
        sequence $(f_{x,j})_{j \in \N}$ converges in $C^1(T,T^\perp)$ to the
        function~$f_{x,0}$. Since the limit function $f_{x,0}=\varphi h_0$ is of
        class $ C^{1,\alpha}$, it follows that $(f_{x,j})_{j \in \N}$ actually
        converges in $C^{1,\alpha'}$ for any $\alpha' \in (0,\alpha)$. Moreover,
        by~\eqref{eq:hj-graph} and~\eqref{eq:hj-graph-rho} one sees that
        \begin{gather*}
            \Sigma_j \cap \Ball{x}{ \rho} = 
            (x + \graph(f_{x,j})) \cap \Ball{x}{ \rho}
            \quad \text{for $j \ge j_2$ or $j=0$.}
        \end{gather*}
        Therefore, $(\Sigma_j)_{j \in \N}$ satisfies condition~\ref{C1a:i:param-conv} of
        Definition~\ref{def:C1a-conv} and the proof is complete.
    \end{proof}

    \section{Isotopies, tubular neighbourhoods and diffeomorphisms}\label{sec:4}

    To prove Theorem~\ref{thm:diff} we proceed as in~\cite[Chapter~4,
    Section~5]{Hir94}.

    We assign to each $V \in G(n,n-m)$  an orthogonal projection 
    $V_{\natural} \in \Hom(\R^n,\R^n)$ onto~$V$.  
    By~\cite[3.1.19(2)]{Fed69} the set
    \begin{gather*}
        \Gr = \{ P \in \Hom(\R^n,\R^n) : P \circ P = P,\, P^* = P,\, 
        \trace P = n-m  \} 
    \end{gather*}
    is a~$C^{\infty}$-submanifold of~$\R^{n^2}$, and the mapping $V \mapsto
    V_{\natural}$ is a~$C^{\infty}$-diffeomorphism \emph{and an
      isometry}. 

    \begin{defin}
        \label{def:e-normal}
        Let $\Sigma \subset \R^n$ be an $m$-dimensional
        $C^1$-submanifold of~$\R^n$ and $\varepsilon > 0$. 
        A map $\Phi : \Sigma \to \Gr$ is called an
        \emph{$\varepsilon$-normal map for $\Sigma$} if $\Phi$ is  $C^1$-smooth,
        $\lip(\Phi) < \infty$ and if
        \begin{gather*}
            \| \Phi(x) - (T_x\Sigma)^{\perp}_{\natural} \| \le \varepsilon \quad
            \forall x\in \Sigma \,.
        \end{gather*}
    \end{defin}

    \begin{lem}[\textbf{nearly normal spaces of class $\mathbf{C^1}$}]
        \label{lem:mollify}
        Let $L,R,d>0$, $\alpha\in (0,1],$ and $\Sigma \in \Cmn(R,L,d)$. 
        Then there exists a constant  $C = C(L,R,\alpha,m,n) \ge 1$ such that for~each
        $\varepsilon \in (0,1]$ there is an $\varepsilon$-normal map 
        $\Phi^{\varepsilon}[\Sigma] : \Sigma \to \Gr$ for $\Sigma$ satisfying,
        in addition, 
        $\lip(\Phi^{\varepsilon}[\Sigma]) \le C \varepsilon^{-1/\alpha}$.
    \end{lem}

    \begin{rem*}
     A similar statement for smooth manifolds (including the $C^1$-case)
        can be found in \cite[Thm. 10A, p.121]{whitney-book_1957}, 
        but for the convenience
        of the reader, and to emphasise how the constants depend quantitatively
        on the parameters determining the class $\Cmn(R,L,d)$ we provide
        the full argument here.
        We are going to 
        construct $\Phi^{\varepsilon}[\Sigma]$ simply by mollifying the map
        $x \mapsto (T_x\Sigma)_{\natural}$. 
        Note that since $\Sigma$ is embedded we do not
        need to use the \emph{center of mass} tool known from Riemannian geometry,
        which was used in~\cite{Bre12}.
    \end{rem*}

    \begin{proof}
        For $\Phi_0 : \Sigma \to \Gr$ given by $\Phi_0(x) :=
        (T_x\Sigma^{\perp})_{\natural}$ for $x\in\Sigma$, we first prove a simple 
        H\"older estimate as follows.
        For $x,y\in\Sigma $ we find
        \begin{gather}
            \label{eq:Phi0-holder}
            \| \Phi_0(x) - \Phi_0(y) \|
            =\|(T_x\Sigma)^\perp_\natural- (T_y\Sigma)^\perp_\natural\|
            \le \frac{2|x-y|^\alpha}{\min\{R^\alpha,(L\sqrt{2})^{-1}\}}
        \end{gather}
        if $|x-y|^\alpha\ge\min\{R^\alpha,(L\sqrt{2})^{-1}\}.$ If not, then
        $y\in \Ball{x}{R}$ so that we can use the local graph representation
        $$
        \Sigma\cap\Ball{x}{R}=(x+\graph(f))\cap\Ball{x}{R}
        $$
        to express the point $y$ as $y=x+\xi+f(\xi)$ for some $\xi\in T_x\Sigma$
        and the function $f:=f_x\in C^{1,\alpha}(T_x\Sigma,T_x\Sigma^\perp)$
        with $f(0)=0,$ $Df(0)=0$, $\lip(f)\le 1,$ and the H\"older estimate on
        $Df$ as in Definition~\ref{class-definition}. In other words, the
        mapping $F(\xi):=x+\xi + f(\xi)$ for $\xi\in T_x\Sigma$ parameterises
        $\Sigma$ over the tangent plane $T_x\Sigma$ locally near $x$, so that
        its differential $DF(\xi):T_x\Sigma \to T_y\Sigma $ can be used to
        estimate for an orthonormal basis $\{e_1,\ldots,e_m\}$ of $T_x\Sigma$
        \begin{multline*}
            \dist(e_i,T_y\Sigma)  \le  |e_i-DF(\xi)e_i|=|e_i-(\Id +Df(\xi))e_i|\\
            \le  \|Df(\xi)-Df(0)\|\le L|\xi|^\alpha\le L|(x+\xi+f(\xi))-x|^\alpha
            =L|y-x|^\alpha
            \quad\forall i=1,\ldots,m,
        \end{multline*}
        where we also used that $f(\xi)\perp \xi$ by definition of $f$.  Since
        $|y-x|^\alpha<\min\{R^\alpha,(L\sqrt{2})^{-1}\}\le
        (L\sqrt{2})^{-1}$ 
        in the present case, we can apply a quantitative linear 
        algebra estimate  \cite[Prop.~2.5]{KStvdM-GAFA} to find 
        a constant $C=C(m)$ such that
        $$
        \|\Phi_0(x)-\Phi_0(y)\|=\ang(T_x\Sigma,T_y\Sigma)\le C(m)L|x-y|^\alpha.
        $$
        Combining both cases leads to the desired H\"older estimate 
        for $\Phi_0$ with H\"older constant
        $$
        C_0=C_0(L,R,\alpha,m)
        :=\max\Big\{\frac{2}{\min\{(L\sqrt{2})^{-1},R^\alpha\}},C(m)L
        \Big\}.
        $$
        Notice that the constant $C_0$ does not depend on $R$ or 
        $\alpha$ if $R^\alpha\ge (L\sqrt{2})^{-1}.$

        Choosing an orthonormal
        coordinate system in~$\R^n$ we can represent $\Phi_0$ as an 
        $(n\times n)$-matrix
        of~functions $(\Phi_0^{ij})_{i,j = 1}^{n}$ and extend each
        $\Phi_0^{ij}$ to all of~$\R^n$ by setting
        \begin{gather*}
            \Phi_1^{ij}(x) 
            = \inf_{z\in\Sigma} \{ \Phi_0^{i,j}(z) + C_0|z-x|^{\alpha}\}
        \end{gather*}
        preserving the same H\"older exponent $\alpha$ and  H\"older constant
        $C_0$ for each $i,j=1,\ldots,n$ (cf. for the proof of
        \cite[Theorem 1, p.80]{EG-book_1992} which carries over to all 
        $\alpha\in (0,1]$).
        The matrix $(\Phi_1^{ij})_{i,j=1}^n$ represents the H\"older
        continuous mapping $\Phi_1:\R^n\to\Hom(\R^n,\R^n)$ with $\Phi_1|_\Sigma=
        \Phi_0$ and the estimate
        \begin{gather}
            \label{eq:Phi1-holder}
            \| \Phi_1(x) - \Phi_1(y) \| \le C_1 |x-y|^{\alpha} 
            \quad\forall x,y\in\R^n\,,
        \end{gather}
        where $C_1=C_1(L,R,\alpha,m,n) := n C_0(L,R,\alpha,m)$.  
        Now let $\phi \in C_0^\infty(\Ball{0}{1})$ with $\phi(x)=1$ for all
        $x\in\Ball{0}{1/2}),$ $0\le\phi(x)\le 1$ and $|\nabla\phi(x)|\le 4$ for
        all $x\in \Ball{0}{1}$, and $\int_{\R^n}\phi(x)\,dx=1,$ and consider for
        $r>0$ the usual scaling $\phi_r(x):=r^{-n}\phi(x/r)$ to define the
        convolution $\Phi_{2,r}: \R^n\to\Hom(\R^n,\R^n)$ as
        \begin{gather*}
            \Phi_{2,r}(x) = \phi_r * \Phi_1(x) 
            = \int_{\R^n} \phi_r(x-z)\Phi_1(z)\,dz \,.
        \end{gather*}
        Since 
        $$
        \|\Phi_1(z)\|\le
        \|\Phi_1(x)\|+C_1|x-z|^\alpha =
        \|\Phi_0(x)\|+C_1|x-z|^\alpha\le 1+C_1r^\alpha\le 1+2C_1
        $$
        for all $x\in\Sigma$, $z\in\Ball{x}{r},$ $r\in (0,2]$, we find 
        \begin{equation}\label{phi1bound} 
            \|\Phi_1(\cdot)\|\le 1+2C_1\quad\textnormal{on $\Sigma+\Ball{0}{2}$},
        \end{equation}
        where the constant on the right-hand side depends on $L,R,\alpha,m,$ and
        $n$.
        Therefore, we can estimate for $x,y\in\Sigma+\Ball{0}{1}$, $e\in\S^{n-1}$,
        $r\in (0,1)$,
        \begin{equation}\label{sup-Phi2r} 
            |\Phi_{2,r}(x)e|=\Big|\int_{\R^n}\phi_r(x-z)\Phi_1(z)e\,dz\Big|
            \le\int_{\Ball{x}{r}}\|\Phi_1(z)\|\phi_r(x-z)\,dz
            \le 1+2C_1,
        \end{equation} 
        because $\dist(z,\Sigma)\le |z-x|+1<2$ for all $z\in\Ball{x}{r},$
        whence
        \begin{equation}\label{osc-Phi2r-nonlocal}
            \|\Phi_{2,r}(x)-\Phi_{2,r}(y)\|\le 2(1+2C_1)\le 2(1+2C_1)\frac{|x-y|}{r}
        \end{equation}
        for all $x,y\in\Sigma+\Ball{0}{1}$ with $|x-y|\ge r,$ $r\in (0,1).$
        On the other hand, for $x,y\in\Ball{0}{1}+\Sigma$ with
        $|x-y|<r$ and for $e\in\S^{n-1}$ one estimates
        \begin{eqnarray}
            \big|\big(\Phi_{2,r}(x)-\Phi_{2,r}(y)\big)e\big| & = &
            \Big|\int_0^1\nabla\Phi_{2,r}(tx+(1-t)y)\cdot (x-y)e\,dt\Big|\notag\\
            &\le & \frac{1}{r^{n+1}}\int_0^1\int_{\Ball{tx+(1-t)y)}{r}}\big|
            \nabla\phi\big(\tfrac{tx+(1-t)y-z}{r}\big)\big|\|\Phi_1(z)\|\,dzdt|x-y|.
            \label{osc-Phi2r-local}
        \end{eqnarray}
        Since $\dist(tx+(1-t)y,\Sigma)\le\dist(x,\Sigma)+(1-t)|x-y|$ for all $t\in
        [1/2,1]$ and $\dist(tx+(1-t)y,\Sigma)\le\dist(y,\Sigma)+t|x-y|$ for all
        $t\in [0,1/2]$ we find $\dist(tx+(1-t)y,\Sigma)<1+r/2$ for all $t\in [0,1]$;
        hence $\dist(z,\Sigma)<1+ 3r/2 <2$ for all $z\in \Ball{tx+(1-t)y}{r}$, $r\in
        (0,2/3)$, which implies $\|\Phi_1(z)\|\le 1+2C_1$ for such $z$ by virtue
        of~\eqref{phi1bound},
        which inserted in
        \eqref{osc-Phi2r-local} gives
        \begin{equation}\label{eq:Phi2-lip}
            \begin{split}
                \|\Phi_{2,r}(x) - \Phi_{2,r}(y)\|
                \le 4(1+2C_1)\frac{\omega_n}{r}|x-y|
                &   \le 24(1+2C_1) \frac{|x-y|}{r}\\
                & =: C_2(L,R,\alpha,m,n)\frac{|x-y|}{r}
            \end{split}
        \end{equation}

        \noindent
        for all $x,y\in\Sigma+\Ball{0}{1}$ and $r\in (0,2/3)$. (We have
        used that the volume $\omega_n$ of the $n$-dimensional unit ball is at most $6$ for all $n=1,2,\ldots$)

        Furthermore, for $x\in\Sigma$,
        $$
        \Phi_0(x)-\Phi_{2,r}(x) 
        = \int_{\R^n} \big(\Phi_0(x)-\Phi_1(z)\big)\phi_r(x-z) \,dz
        = \int_{\Ball{x}{r}} \big(\Phi_1(x)-\Phi_1(z)\big)\phi_r(x-z)\,dz,
        $$
        since $\Phi_1|_\Sigma=\Phi_0$, so that by \eqref{eq:Phi1-holder}
        \begin{equation}\label{eq:Phi2-dist}
            \|\Phi_0(x) - \Phi_{2,r}(x)\|
            \le C_1\int_{\Ball{x}{r}}|x-z|^\alpha\phi_r(x-z)\,dz
            \le C_1r^\alpha<C_2r^\alpha\quad\forall x\in\Sigma,\, r>0.
        \end{equation}
        Since $\Gr$ is a~$C^{\infty}$-submanifold
        of~$\Hom(\R^n,\R^n)\simeq\R^{n^2}$, it has positive reach
        $r_{\Gr}=r_{\Gr}(m,n)>0$ in the sense of Federer \cite[Definition
        4.1]{federer_1959} such that the nearest point projection
        $P_\Gr:\Gr+\Ball{0}{r_\Gr}\to\Gr$ is $C^\infty$-smooth; see, e.g.,
        \cite[Lemma, p. 153]{Foote84}\footnote{Formally, Foote 
          \cite[Lemma, p.  153]{Foote84} mentions only \emph{a neighbourhood}
          of the manifold $M$. However, this neighbourhood is defined via
          an application of the inverse function theorem, which -- in light
          of Federer \cite[Theorem 4.8(13)]{federer_1959} -- is possible
          on the whole $\Gr+\Ball{0}{r_\Gr}$.}.
        In addition, for any $\delta\in
        (0,r_\Gr/2]$ it follows from \cite[Theorem 4.8(8)]{federer_1959} that
        $P_\Gr$ has Lipschitz constant $\lip(P_\Gr)\le 2$ on
        $\Gr+\Ball{0}{\delta}$.

        According to \eqref{eq:Phi2-dist} the map 
        $\Phi_{3,r}:=P_\Gr\circ\Phi_{2,r}|_\Sigma$
        maps $\Sigma$ into $\Gr$ if $C_2r^\alpha\le r_\Gr/2.$ 
        Now choose for a given
        $\varepsilon \in (0,1]$ first
        \begin{equation}
            \label{def:d0reach}
            \delta_0=\delta_0(L,R,\alpha,m,n)
            :=\min\{r_\Gr,(2/3)^\alpha C_2,1\},
        \end{equation} and
        then $r_\varepsilon:=(\delta_0/(2C_2))^{1/\alpha}\varepsilon^{1/\alpha}\in
        (0,2/3). $ Then $\Phi^\varepsilon
        [\Sigma]:=\Phi_{3,r_\varepsilon}:\Sigma\to\Gr$ as a composition of
        $C^1$-maps is also of class $C^1$, and according to \eqref{eq:Phi2-lip} with
        Lipschitz constant
        \begin{equation}
            \label{def:lipnorm}
            \lip(\Phi^\varepsilon [\Sigma])
            \le \lip(P_\Gr)C_2/r_\varepsilon
            \le
            (2C_2)^{1+(1/\alpha)}(\delta_0\varepsilon)^{-1/\alpha}
            =:C(L,R,\alpha,m,n)\varepsilon^{-1/\alpha}
        \end{equation}
        Finally, $\Phi^\varepsilon[\Sigma]$ is an 
        $\varepsilon$-normal map for $\Sigma$,
        since by \eqref{eq:Phi2-dist}
        $$
        \|\Phi^\varepsilon[\Sigma](x) 
        - (T_x\Sigma)^\perp_\natural\|=
        \|P_\Gr\circ\Phi_{2,r_\varepsilon}(x) -  
        P_\Gr\circ\Phi_0(x)\|\le\lip(P_\Gr)\|
        \Phi_{2,r_\varepsilon}(x) - \Phi_0(x)\| 
        \overset{\eqref{eq:Phi2-dist}}{<}
        2C_2r_\varepsilon^{\alpha}\le\varepsilon.
        $$
    \end{proof}

    \begin{rem} 
        \label{rem:lipnorm}	
        An inspection of the proof yields $C_0 \le C(m)L + 2R^{-\alpha} \le
        c(m,n,l,p)(L+1)$, $\alpha=1-p_0(\E)/p$, whenever $R,L$ are given by
        \eqref{RLfromEp} for an energy threshold $E\ge\E(\Sigma)$ for a
        particular energy $\E\in\{\E^l_p,\TP_p,\TP^G_p\}$. This gives $C_2\le
        c(m,n,l,p)(E^{1/p} +1)$. Assuming w.l.o.g. that $C_2 \ge \frac 32$,
        we obtain $\delta_0$ in \eqref{def:d0reach} unrelated to $C_2$, and
        finally, for a fixed $\eps\in (0,\frac{1}{100})$ and
        $\alpha=\alpha(p)=1-p_0(\E)/p$,
        \[
        \lip(\Phi^\varepsilon [\Sigma])
        \le c(m,n,l,p) (E^{1/p}+1)^{1+(1/\alpha)} \delta_0^{-1/\alpha}, 
        \qquad\mbox{where}\quad \delta_0=\min\{r_\Gr,1\}.
        \]
    \end{rem}

    \begin{defin}
        \label{def:Psi-N}
        Let $R,L,d>0$, and $\alpha\in (0,1]$, $\Sigma \in \Cmn(R,L,d)$, and for
        some $\varepsilon\in (0,1/100)$ let $\Phi:\Sigma\to\Gr$ 
        be an $\varepsilon$-normal map for $\Sigma$. 
        For $\delta >0 $ define the \emph{$\delta$-normal neighbourhood}
        \begin{gather*}
            N_{\delta}(\Sigma,\Phi) := 
            \{ (x,v) \in \Sigma \times \R^n :  \Phi(x)v = v,\, |v| < \delta \}
            \\
            \text{and the map}\quad
            \Psi_{\delta}[\Sigma,\Phi] : N_{\delta}(\Sigma,\Phi) \to \R^n \,,
            \qquad
            \Psi_{\delta}[\Sigma,\Phi](x,v) := x + v \,.
        \end{gather*}
    \end{defin}

    \begin{lem}
        [\textbf{tubular neighbourhoods for $\boldsymbol{C^{1,\alpha}}$ manifolds}]
        \label{lem:tubular}
        Assume $R,L,d >0$, $\alpha \in (0,1]$, and let $\Sigma \in \Cmn(R,L,d)$, and
        for
        $\varepsilon \in (0,1/100)$ let  $\Phi : \Sigma \to \Gr$ be an
        $\varepsilon$-normal map for $\Sigma$. Then there is a  
        constant 
        $\delta_{tub} = \delta_{tub}(R,L,\alpha,\varepsilon,\lip(\Phi)) > 0$ such
        that for all $\delta \in (0,\delta_{tub}]$
        \begin{enumerate}[label=(\roman*)]
        \item\label{i:Psi-embed}
            $\Psi = \Psi_{\delta}[\Sigma,\Phi]$ is a~$C^1$-embedding,
        \item\label{i:Psi-bilip}
            $(1/4) |(x-y,u-v)| \le |\Psi(x,u) - \Psi(y,v)| \le \sqrt{2}|(x-y,u-v)|$
            for all $(x,u), (y,v) \in N_{\delta}(\Sigma,\Phi)$,
        \item\label{i:Psi-dist}
            $\dist(\Psi(x,v), \Sigma) > \tfrac 14 |v|$
            for all $(x,v) \in N_{\delta}(\Sigma,\Phi)$, $v\not=0$,
        \item\label{i:Psi-onto}
            $\Sigma + \Ball{0}{\delta/2} \subset 
            \Psi_{\delta}[\Sigma,\Phi](N_{\delta}(\Sigma,\Phi))$.
        \end{enumerate}
    \end{lem}

    \begin{proof}
        For any $\delta > 0$ the mapping $\Psi=\Psi_{\delta}[\Sigma,\Phi]$ is
        the  restriction of the smooth function
        $$
        \R^n\times\R^n\ni (x,v) \mapsto x+v \in\R^n
        $$ 
        to the $C^1$-submanifold
        
        $$
        N:=N_{\delta}(\Sigma,\Phi))=\bigcup_{x\in\Sigma}
        \big[\{x\}\times \big(\textnormal{ker\,}(\Phi(x)-\Id)\cap\Ball{0}{\delta}
        \big)\big];
        $$ 

        \noindent
        hence $\Psi$ is of  class $C^1$. 
        To show that $\Psi$ is an embedding it suffices to prove that
        it is bilipschitz, i.e., that \ref{i:Psi-bilip} holds, for sufficiently
        small $\delta.$
        For any $(x,u),(y,v)\in N$ one has
        \begin{gather}
            \label{eq:Psi-upper}
            |\Psi(x,u) - \Psi(y,v)| 
            = |(x-y) + (u-v)| \le\sqrt{2}|(x-y,u-v)|\,,
        \end{gather}
        and therefore it is enough to prove the estimate from
        below in \ref{i:Psi-bilip}.
        Set
        \begin{gather}
            \label{eq:delta-tb}
            \delta_{tub}:=
            \min\left\{\frac R4 ,
                \frac 14  \Big(\frac{\varepsilon}{L}\Big)^{1/\alpha},
                \frac{\varepsilon}{4\lip(\Phi)},1
            \right\} \,.
        \end{gather}
        Assume $0 < \delta \le \delta_{tub}$. 
        For $(x,u),(y,v)\in N$ define the subspaces $U:=\im\Phi(x)$ and $V:=
        \im\Phi(y)$ and observe that if $|x-y|\ge 4\delta$ 
        then, on the one hand, 
        $$
        |\Psi(x,u) - \Psi(y,v)|  \ge |x-y| - |u| - |v| \ge |x-y|-2\delta\ge
        |x-y|/2,
        $$
        and, on the other hand, 
        $$
        |(x-y,u-v)|\le |x-y|+|u-v|\le |x-y|+2\delta\le \tfrac 32 |x-y|,
        $$
        so that
        \begin{gather}
            \label{eq:xy-far-away}
            |\Psi(x,u) - \Psi(y,v)|\ge \frac 13 |(x-y,u-v)|\,.
        \end{gather}
        Thus we have to treat the case $|x-y| < 4 \delta$. Since $\Sigma 
        \in \Cmn(R,L,d)$
        and $|x-y| < 4 \delta \le R$, we can use the local graph representation
        $$
        \Sigma\cap\Ball{x}{R}=(x+\graph(f_x))\cap\Ball{x}{R}
        $$
        for a function $f:=f_x\in C^{1,\alpha}(T_x\Sigma,T_x\Sigma^\perp)$
        satisfying $f(0)=0$, $Df(0)=0$, $\lip(f)\le 1,$ and the H\"older estimate
        for $Df$ as in Definition \ref{class-definition},
        to find for $y=x+\eta+f(\eta)\in x+\graph(f)$, $\eta\in T_x\Sigma$,
        by means of 
        \eqref{improved-f_x-estimate} 
        $$
        \dist(y,x+T_x\Sigma)=|f(\eta)|\overset{\eqref{improved-f_x-estimate}}{\le}
        L|\eta|^{1+\alpha}\le L|x-y|^{1+\alpha}\le L(4\delta)^\alpha|x-y|,
        $$
        so that we obtain by our choice of $\delta_{tub}$ in \eqref{eq:delta-tb}
        \begin{gather}\label{19}
            |(T_x\Sigma^{\perp})_{\natural}(y-x)| =  \dist(y, x + T_x\Sigma) 
            \overset{\eqref{eq:delta-tb}}{\le} \varepsilon |x-y| \,.
        \end{gather}
        Using this estimate together with the fact that $\Phi$ is 
        an~$\varepsilon$-normal
        map for $\Sigma$  we can write
        \begin{eqnarray}
            |U^{\perp}_{\natural}(x-y)| 
            & \ge &  |(T_x\Sigma)_{\natural}(x-y)| 
            - \|U^{\perp}_{\natural}-(T_x\Sigma)_\natural\||x-y|
            \notag\\
            & \ge & (1-\varepsilon) |x-y| -     
            \|\Phi(x)-(T_x\Sigma)^\perp_\natural\||x-y|
            \ge (1 - 2\varepsilon) |x-y| \,,
            \label{eq:U-perp-xy}
        \end{eqnarray}
        which implies by means of $|U_{\natural}(x-y)|^2 = |x-y|^2 -
        |U^{\perp}_{\natural}(x-y)|^2$  the inequality
        $
        |U_{\natural}(x-y)|^2 \le \big(1-(1-2\varepsilon)^2\big) |x-y|^2 \,;
        $
        hence, 
        \begin{equation}\label{eq:U-xy}
            |U_{\natural}(x-y)|\le 2\sqrt{\varepsilon}|x-y|.
        \end{equation}
        Recall our choice of $\delta_{tub}$ in \eqref{eq:delta-tb} to estimate for
        $u=\Phi(x)u\in U\cap\Ball{0}{\delta}$ and 
        $v=\Phi(y)v\in V\cap\Ball{0}{\delta}$
        \begin{gather}
            \label{eq:U-perp-uv}
            |U^{\perp}_{\natural}(u-v)| = |U^{\perp}_{\natural}v| 
            = |(U_{\natural} - V_{\natural}) v| 
            \le |v| \|\Phi(x) - \Phi(y)\|
            \le \delta \lip(\Phi) |x - y|
            \overset{\eqref{eq:delta-tb}}{\le} \varepsilon |x-y| \,,
        \end{gather}
        so that
        \begin{gather}
            \label{eq:U-uv}
            |U_{\natural}(u-v)| =|(\Id-U^\perp_\natural)(u-v)|
            \ge  |u-v| - \varepsilon |x-y| \,.
        \end{gather}
        Combining \eqref{eq:U-perp-xy}, \eqref{eq:U-xy}, \eqref{eq:U-perp-uv},
        \eqref{eq:U-uv} with  the  triangle inequality, we
        arrive at
        \begin{eqnarray}
            |\Psi(x,u) - \Psi(y,v)|  
            & = & |(U^\perp_\natural+U_\natural)
            \big((x-y)+(u-v)\big)|\notag\\
            & \ge & 
            |U^{\perp}_{\natural}(x-y) + U_{\natural}(u-v)| 
            - |U_{\natural}(x-y)| - |U^{\perp}_{\natural}(u-v)| \notag\\
            & \overset{\eqref{eq:U-xy},\eqref{eq:U-perp-uv}}{\ge} &
            \frac{1}{\sqrt{2}} 
            \Big(|U^{\perp}_{\natural}(x-y)| + |U_{\natural}(u-v)|\Big) 
            - (2 \sqrt{\varepsilon}+\varepsilon) |x-y| \notag\\
            & \overset{\eqref{eq:U-uv},\eqref{eq:U-perp-xy}}{\ge}& 
            \Big(\frac{1}{\sqrt{2}}(1-3\varepsilon)
            -2\sqrt{\varepsilon}-\varepsilon\Big)|x-y|+\frac{1}{\sqrt{2}}|u-v|
            \notag \\
            & \ge & \frac 14 |(x-y,u-v)|\,, \label{eq:Psi-lower}
        \end{eqnarray}
        since $\varepsilon\in (0,1/100)$.  
        So, part \ref{i:Psi-bilip}  
        of Lemma~\ref{lem:tubular} follows from \eqref{eq:Psi-upper}, 
        \eqref{eq:xy-far-away}, and \eqref{eq:Psi-lower}, which -- as observed
        above -- implies part \ref{i:Psi-embed} as well.
        
        We now turn to the proof of part \ref{i:Psi-dist}. For 
        $(x,v) \in
        N_{\delta}(\Sigma,\Phi)$ with $v\not= 0$, denote 
        $u:= (T_x\Sigma)^{\perp}_{\natural}v$
        and note that by definition of $N$ and the fact that $\Phi$ is an
        $\varepsilon$-normal map for $\Sigma$,
        \begin{equation}\label{24a}
            |u-v|=\big|\big(T_x\Sigma)^\perp_\natural -\Phi(x)\big)v\big|\le
            \|\big(T_x\Sigma)^\perp_\natural-\Phi(x)\||v|<\varepsilon|v|.
        \end{equation}
        Since $\Sigma
        \in \Cmn(R,L,d)$, we find for any  $y \in \Sigma$ with 
        $|y - x|<4\delta_{tub}\overset{\eqref{eq:delta-tb}}{\le}R$ 
        as in \eqref{19}
        \begin{gather}
            \label{eq:sigma-x-To}
            \dist(y, x + T_x\Sigma) \le L|y-x|^{1+\alpha} \le \varepsilon |y-x| \,.
        \end{gather}
        On the other hand, if $y\in\Ball{x+u}{\tfrac 12 |v|}$, then
        \begin{gather}
            \label{eq:x-o}
            |y-x| \le |y-(x+u)| + |u| \le \tfrac 12 |v|+\big|
            \big(T_x\Sigma\big)^\perp_\natural v\big|\le \tfrac 32 |v|,
        \end{gather}
        and by \eqref{24a}
        \begin{multline}
            \label{eq:x-To}
            \dist(y,x+T_x\Sigma) 
            = \big|\big(T_x\Sigma\big)^{\perp}_{\natural}(y-x)\big|
            \ge \big|\big(T_x\Sigma\big)^{\perp}_{\natural}u\big| - \big|
            \big(T_x\Sigma\big)^{\perp}_{\natural}(y-(x+u))| 
            \\
            =|u|-\big|\big(T_x\Sigma\big)^\perp_\natural (y-(x+u))\big|
            \overset{\eqref{24a}}{\ge} |v| - |u-v| - |y - (x+u)|
            \ge \big(\tfrac 12 - \varepsilon\big) |v| \,.
        \end{multline}
        Combining
        \eqref{eq:sigma-x-To}, \eqref{eq:x-o} and \eqref{eq:x-To} would yield for
        $y\in\Sigma\cap\Ball{x+u}{\tfrac 12 |v|}\cap \Ball{x}{4\delta_{tub}}$
        \begin{gather*}
            \big( \tfrac 12 - \varepsilon \big) |v| \overset{\eqref{eq:x-To}}{\le}
            |T_x\Sigma^{\perp}_{\natural}(y-x)|=\dist(y,x+T_x\Sigma)
            \overset{\eqref{eq:sigma-x-To}}{\le}\varepsilon|y-x|
            \overset{\eqref{eq:x-o}}{\le} \tfrac 32 \varepsilon |v| 
        \end{gather*}
        contradicting $\varepsilon\in (0,1/100)$ because $|v|\not=0.$
        Since
        $\Ball{x+u}{ \tfrac 12|v|} \subset \Ball{x}{4\delta_{tub}}$ because
        $|u|\le |v|<\delta<\delta_{tub}$, this can only mean that 
        $\Ball{x+u}{ \tfrac 12|v|} \cap \Sigma = \emptyset$, which implies by
        \eqref{24a} 
        \begin{gather*}
            \dist(\Psi(x,v), \Sigma) 
            \ge \dist(x+u, \Sigma) - |u-v| 
            \ge \tfrac 12 |v| - |u-v|
            \overset{\eqref{24a}}{\ge} \big( \tfrac 12 - \varepsilon \big) |v|
            > \tfrac 14|v| \,.
        \end{gather*}
        Finally we prove part  \ref{i:Psi-onto}. For $x \in \Sigma\in\Cmn(R,L,d)$
        there exists a function 
        $f=f_x\in C^{1,\alpha}(T_x\Sigma,T_x\Sigma^\perp)$ with
        $f(0)=0,$ $Df(0)=0$,  $\lip(f)\le 1,$ and the H\"older condition for 
        $Df$ in Definition \ref{class-definition},
        such that by \eqref{eq:delta-tb}
        $$
        \Ball{x}{4\delta_{tub}}\cap\Sigma=(x+\graph(f))\cap
        \Ball{x}{4\delta_{tub}}.
        $$
        For any $\xi\in T_x\Sigma\cap\Ball{0}{4\delta_{tub}}$ one can use 
        \eqref{improved-f_x-estimate} and \eqref{eq:delta-tb} to estimate
        \begin{gather}
            \label{eq:f-small}
            |f(\xi)|\overset{\eqref{improved-f_x-estimate}}{\le} 
            L|\xi|^{1+\alpha} \le L (4\delta_{tub})^{\alpha} |\xi| 
            \le \varepsilon |\xi| \,.
        \end{gather}
        Again by~\eqref{eq:delta-tb} in combination with 
        Definition~\ref{def:e-normal} we have for $\zeta \in
        \Ball{x}{4\delta_{tub}} \cap \Sigma$ 
        \begin{gather}
            \label{eq:Phi-Tan}
            \|\Phi(\zeta) - \big(T_x\Sigma\big)_{\natural}^{\perp}\| 
            \le \|\Phi(\zeta) - \Phi(x)\| + \|\Phi(x) - 
            \big(T_x\Sigma\big)_{\natural}^{\perp}\|
            < \lip(\Phi) 4\delta_{tub} + \varepsilon \le 2\varepsilon \,.
        \end{gather}
        For fixed $\delta\in (0,\delta_{tub}]$ consider
        the $C^1$-functions
        \begin{gather*}
            \psi : \CBall{0}{\delta} \to \Sigma
            \quad \text{given by} \quad
            \psi(w) := x+(T_x\Sigma)_\natural w
            +f\big((T_x\Sigma)_\natural w\big) \,,
            \\
            F : \CBall{0}{\delta}\to N_\delta(\Sigma,\Phi)
            \quad \text{defined by} \quad
            F(z) := 
            \big(\psi(z),\Phi(\psi(z))\big(T_x\Sigma\big)^\perp_\natural z\big)
            \,, 
            \\
            \text{and} \quad
            G = \big(\Psi_\delta[\Sigma,\Phi]\circ F\big)-x:
            \CBall{0}{\delta}\to\R^n
            \,.
        \end{gather*}
        Employing \eqref{eq:f-small} and \eqref{eq:Phi-Tan} we obtain for $z \in
        \Ball{0}{\delta}$
        \begin{eqnarray*}
            |G(z) - z| &=&|\psi(z)+\Phi(\psi(z))(T_x\Sigma)^\perp_\natural z-x-
            (T_x\Sigma)_\natural z-(T_x\Sigma)^\perp_\natural z|\\
            &   \le & |\psi(z)-(T_x\Sigma)_\natural z-x|
            +| \Phi(\psi(z))(T_x\Sigma)^\perp_\natural z 
            - (T_x\Sigma)^\perp_\natural z|\\
            & \le & |f\big((T_x\Sigma)_\natural z\big)|+
            \|\Phi(\psi(z))-(T_x\Sigma)^\perp_\natural\|
            |(T_x\Sigma)^\perp_\natural z|\\
            &\overset{\eqref{eq:f-small},\eqref{eq:Phi-Tan}}{\le} &
            \varepsilon |(T_x\Sigma)_\natural z|
            +2\varepsilon |(T_x\Sigma)^\perp_\natural z|
            \le 2\sqrt{2}\varepsilon |z|<2\sqrt{2}\varepsilon\delta,
        \end{eqnarray*}
        so that  we can apply
        Proposition~\ref{prop:deg-one} to get
        \begin{gather*}
            \CBall{0}{\delta/2}
            \subset \Ball{0}{(1 - 2 \sqrt 2 \varepsilon} \delta)
            \subset G(\CBall{0}{\delta})
            = \Psi_{\delta}[\Sigma,\Phi](F(\CBall{0}{\delta})) -x;
        \end{gather*}
        hence 
        \begin{gather*}
        x+\Ball{0}{\delta/2}\subset
        \Psi_\delta[\Sigma,\Phi](F(\CBall{0}{\delta})) 
        \subset \Psi_\delta[\Sigma,\Phi](N_\delta(\Sigma,\Phi)).
        \end{gather*}
    \end{proof}

    \begin{prop}
        \label{prop:tales2}
        Let $\theta \in [0,1]$, $\lambda,\gamma \in [0,1)$ and $k \in \{ 1, \ldots,
        n-1\}$ and suppose that  $W,T \in G(n,k)$ and $U,V \in G(n,n-k)$ satisfy
        $\ang(W,T) \le \theta$, $\ang(T,U^{\perp}) \le \lambda$,
        $\ang(T,V^{\perp}) \le \lambda$, and
        $\ang(U,V) \le \gamma$. Given any vectors $\mathbf{w} \in W$, $\mathbf{t}
        \in T$, $\mathbf{u} \in U$ and $\mathbf{v} \in V$ such that
        $\mathbf{u}+\mathbf{t} = \mathbf{w}+\mathbf{v}$ the following holds:
        \begin{gather}
            \left( |\mathbf{u}| - \tfrac{\theta}{1 - \lambda}|\mathbf{w}| \right)
            \left(1 - \tfrac{\gamma}{1 - \lambda} \right) \le
            |\mathbf{v}| \le
            \left( |\mathbf{u}| + \tfrac{\theta}{1 - \lambda}|\mathbf{w}| \right)
            \left(1 + \tfrac{\gamma}{1 - \lambda} \right)
        \end{gather}
        \begin{gather}
            \label{eq2-tales2}
            \text{and} \quad
            |\mathbf{u}-\mathbf{v}| 
            \le \tfrac{\gamma}{1-\lambda} \left(
                |\mathbf{u}| + \tfrac{\theta}{1-\lambda}|\mathbf{w}|
            \right)
            + \tfrac{\theta}{1-\lambda}|\mathbf{w}| \,.
        \end{gather}
    \end{prop}

    \begin{figure}[!ht]
        \begin{center}
            \includegraphics*[totalheight=7.5cm]{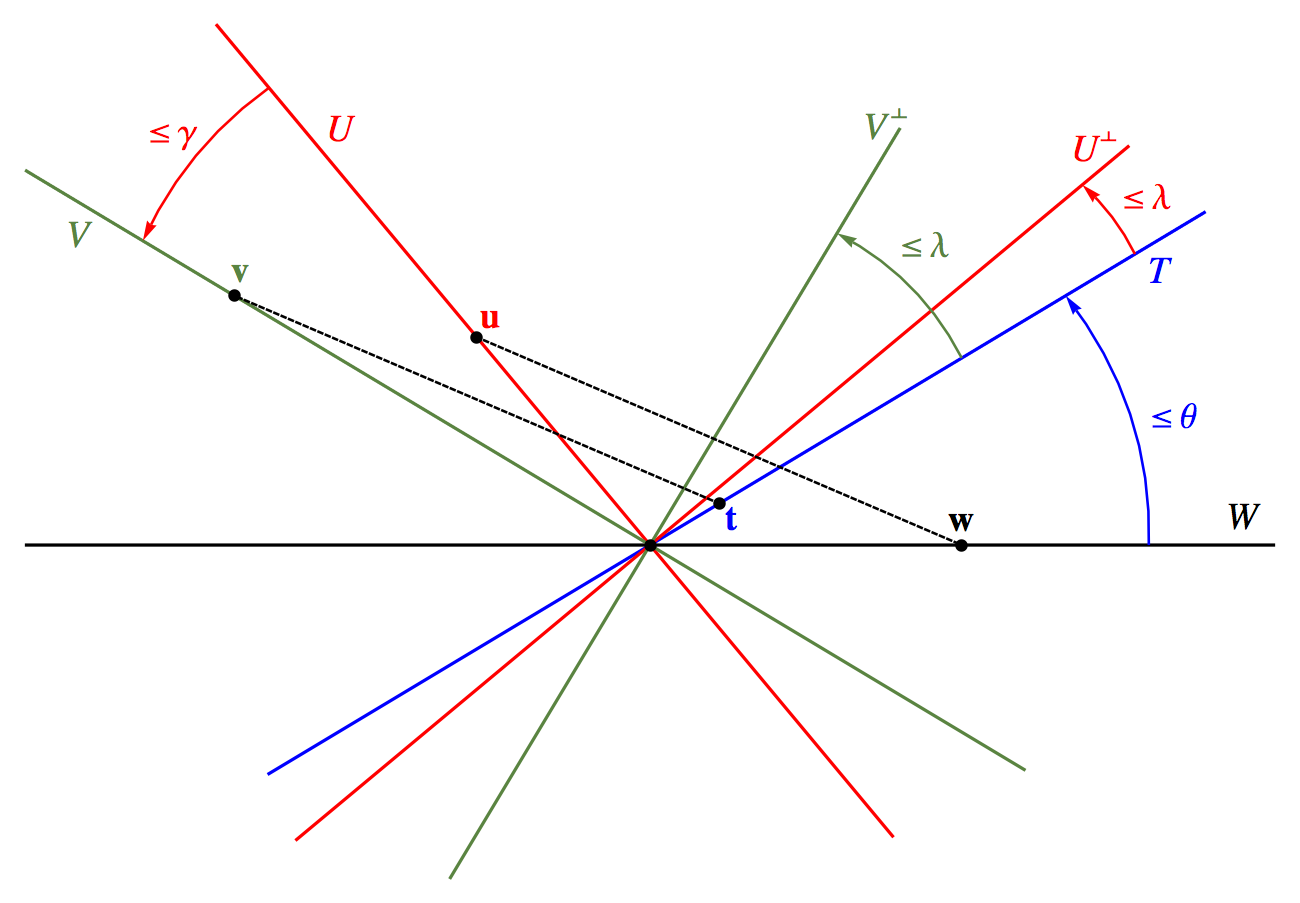}
        \end{center}
        \caption{\label{fig:prop4.6} The situation in Proposition~\ref{prop:tales2}: the
          vectors $\mathbf{u}-\mathbf{w}$ and $\mathbf{v}-\mathbf{t}$ are equal.}
    \end{figure}

    \begin{proof}
        Let $P : \R^n \to U$ be the oblique projection onto $U$
        with $\ker P = T$ and set
        $\bar{\mathbf{u}} = P(\mathbf{u}-\mathbf{w})\in U,$  so that
        $$
        \mathbf{w} + (\bar{\mathbf{u}} - \mathbf{u})
        = \mathbf{w} + (P(\mathbf{u}-\mathbf{w})-\mathbf{u})=
        \mathbf{w} - P\mathbf{w}+P\mathbf{u}-\mathbf{u}
        =\mathbf{w}- P\mathbf{w} \in \ker P  = T.
        $$
        Thus 
        we can apply
        Proposition~\ref{prop:tales1} to $\mathbf{z}:=\mathbf{u}-\bar{\mathbf{u}}
        \in U=:Z$  and $\mathbf{x}:=\mathbf{w}\in W=:X$ (with $Y:=T$ 
        implying $\ang(X,Y)\le\theta,$ $\ang(Y,Z^\perp)\le\lambda$, and
        $\mathbf{z}-\mathbf{x}\in Y$) to obtain $|\bar{\mathbf{u}} - \mathbf{u}| \le 
        \theta |\mathbf{w}|/(1-\lambda)$ which directly leads to
        \begin{gather}\label{intermed}
            |\mathbf{u}| - \frac{\theta}{1 - \lambda}|\mathbf{w}| 
            \le |\mathbf{u}|-|\bar{\mathbf{u}}-\mathbf{u}|
            \le |\bar{\mathbf{u}}| \le|\mathbf{u}|+|\bar{\mathbf{u}}-\mathbf{u}|
            \le |\mathbf{u}| + \frac{\theta}{1 - \lambda}|\mathbf{w}| \,.
        \end{gather}
        Applying Proposition~\ref{prop:tales1} now to $\mathbf{x}
        :=\bar{\mathbf{u}}\in U=:X$ and to
        $$
        \mathbf{z}:=\bar{\mathbf{u}}-\mathbf{v}=P(\mathbf{u}-\mathbf{w})-\mathbf{v}=
        \mathbf{u}-\mathbf{v}-P\mathbf{w}=\mathbf{w}-\mathbf{t}-P\mathbf{w}\in T=:Z
        $$
        (so that $\mathbf{z}-\mathbf{x}=-\mathbf{v}\in V=:Y,$ and hence $
        \ang(X,Y)=\ang(U,V)\le\gamma$ and  $\ang(Y,Z^\perp)=
        \ang(V,T^\perp)=\ang(V^\perp,T)\le\lambda$) to arrive at
        $|\bar{\mathbf{u}} - \mathbf{v}| 
        \le \gamma |\bar{\mathbf{u}}|/(1 - \lambda)$,
        and in consequence
        \begin{gather*}
            |\bar{\mathbf{u}}| \big(1 - \tfrac{\gamma}{1 - \lambda}\big) \le
            |\bar{\mathbf{u}}|-|\bar{\mathbf{u}}-\mathbf{v}|
            \le |\mathbf{v}| \le |\bar{\mathbf{u}}|+|\mathbf{v}-\bar{\mathbf{u}}|
            \le |\bar{\mathbf{u}}| \big(1 + \tfrac{\gamma}{1 - \lambda}\big) \,.
        \end{gather*}
        This together with \eqref{intermed} 
        gives the first part of the proposition. To get the second we
        use \eqref{intermed} to write
        \begin{gather*}
            |\mathbf{u} - \mathbf{v}| 
            \le  |\bar{\mathbf{u}} - \mathbf{v}| 
            + |\bar{\mathbf{u}} - \mathbf{u}|
            \le \tfrac{\theta}{1-\lambda}|\mathbf{w}| 
            + \tfrac{\gamma}{1-\lambda}|\bar{\mathbf{u}}|
            \overset{\eqref{intermed}}{\le}  
            \tfrac{\theta}{1-\lambda}|\mathbf{w}|
            + \tfrac{\gamma}{1-\lambda}\big(|\mathbf{u}| 
            + \tfrac{\theta}{1 - \lambda}|\mathbf{w}|
            \big)\,.
            \qedhere
        \end{gather*}
    \end{proof}

    \begin{defin}
        \label{def:mt}
        For $t\in\R$ we define the continuous map
        \begin{gather*}
            m_t : \R^n \times \R^n \to \R^n \times \R^n \,,
            \quad
            m_t(x,v) := (x,tv) \,.
        \end{gather*}
    \end{defin}

    \begin{lem}[\textbf{bilipschitz diffeomorphisms}]
        \label{lem:bilip-est}
        For $R,L,d \in (0,\infty)$, $\alpha \in (0,1]$, $ \varepsilon \in
        (0,10^{-2})$ let $\Sigma_1,\Sigma_2 \in \Cmn(R,L,d)$ with
        $\rho:=\HD(\Sigma_1, \Sigma_2)<\delta_{tub}/8$, and with
        $\varepsilon$-normal map  $\Phi_1:\Sigma_1\to\Gr$ for $\Sigma_1$,
        where $\delta_{tub}=\delta_{tub}(R,L,\alpha,\eps,\lip(\Phi_1))$
        is the radius of the tubular neighbourhood of $\Sigma_1$ established
        in Lemma \ref{lem:tubular}.
        Set $\Psi := \Psi_{\delta_{tub}}[\Sigma_1,\Phi_1]$ and define
        \begin{gather*}
            F : \Sigma_2 \to \R^n \quad\textnormal{by}
            \quad
            F := \Psi \circ m_0 \circ \Psi^{-1}|_{\Sigma_2}\,,
            \\
            G : \Sigma_2 \to \R^n \quad\textnormal{by}
            \quad
            G := F - \Id \,.
        \end{gather*}
        Then $\Sigma_1\subset\im(F)$, and
        there exist $C_l = C_l(L, \/ \lip(\Phi_1)) \ge 1$ 
        and $\rho_G =
        \rho_G(R,L,\alpha,\eps,\lip(\Phi_1))$ $ \in (0,\delta_{tub}/8]$ 
        such that for all 
        $ \rho=\HD(\Sigma_1, \Sigma_2)\in(0,\rho_G)$
        \begin{enumerate}[label=(\roman*)]
        \item \label{i:be:G-lip}
            $\lip(G) \le C_l \rho^{\alpha/2}$,
        \item \label{i:be:G-small}
            $|G(x)| \le 4 \dist(x,\Sigma_1)$ for all  $x \in \Sigma_2$,
        \item \label{i:be:F-diffeo}
            $F$ is a~bilipschitz diffeomorphism onto its image $\Sigma_1$ satisfying
            \begin{gather*}
                \big(1 - C_{l}\rho^{\alpha/2}\big)|x-y| \le |F(x) - F(y)| \le \big(1 
                + C_{l}\rho^{\alpha/2}\big)|x-y|\quad\forall x,y \in \Sigma_2 \,.
            \end{gather*}
        \end{enumerate}
    \end{lem}

    \begin{figure}[!ht]
        \begin{center}
            \includegraphics*[totalheight=6cm]{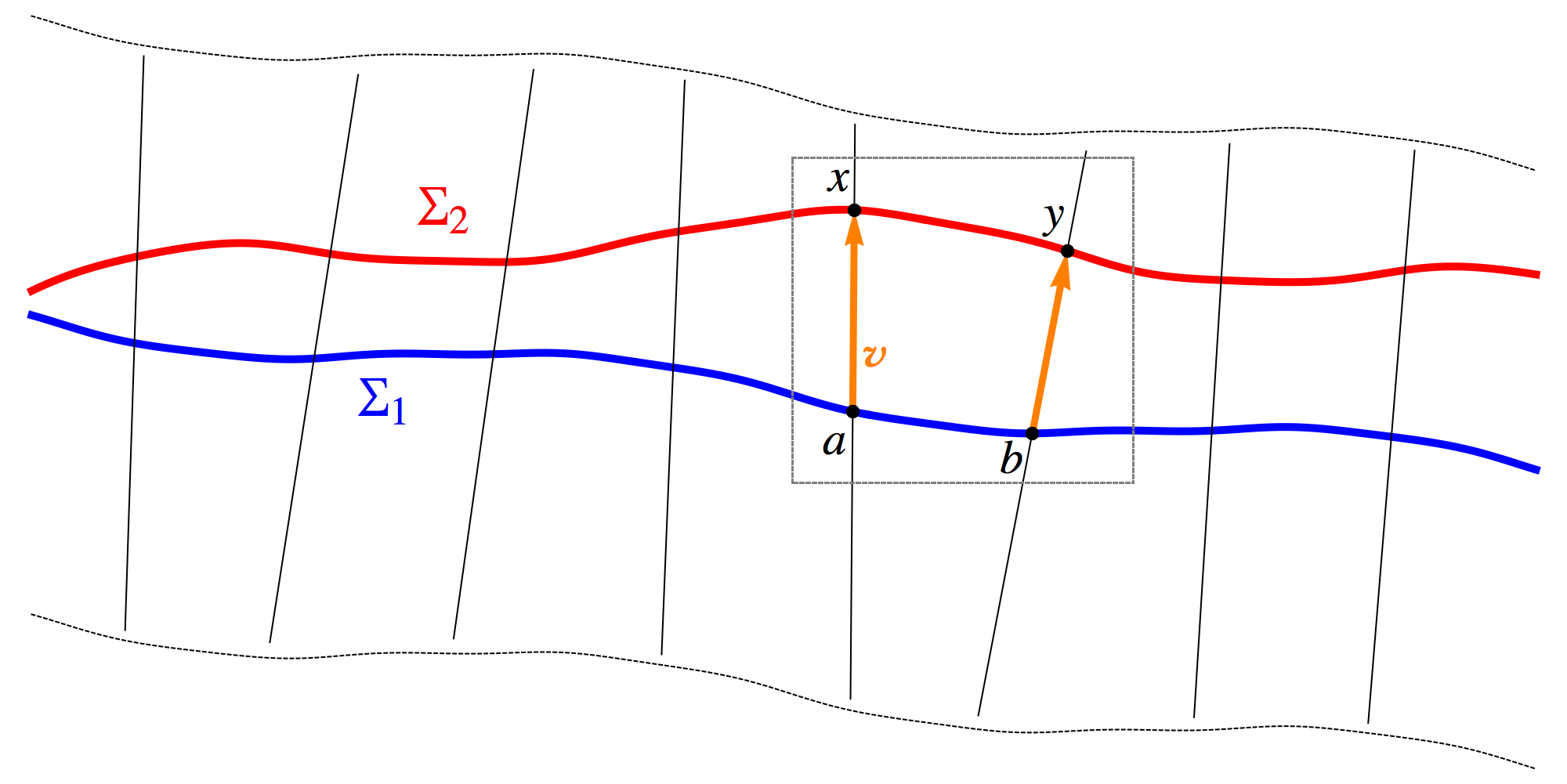}
        \end{center}

        \caption{\label{fig:mapF}
          The definition of $F\colon \Sigma_2\to\Sigma_1$. Thin nearly vertical lines
          represent $\varepsilon$-normal spaces to $\Sigma_1$.
          We have $x=\Psi(a,v)$, so that $\Psi^{-1}|_{\Sigma_2}(x)=(a,v)\in\R^{2n}$. 
          Next, $m_0(a,v)=(a,0)$, and $\Psi(a,0)=a+0=a$. This yields $F(x)=a$.}
    \end{figure}

    \begin{proof}
        Notice that $\Psi^{-1}$ is well-defined on a neighbourhood of $\Sigma_2$ by
        virtue of Lemma \ref{lem:tubular} \ref{i:Psi-onto}, since $\HD(\Sigma_1,
        \Sigma_2) <\delta_{tub}/8,$ so that for $x\in\Sigma_2$ we find a unique
        pair $(\xi,v)\in N_{\delta_{tub}}(\Sigma_1,\Phi_1)$ such that
        $x=\xi+v=\Psi(\xi,v).$ 

        By definition of the map $F$ it is clear that $\im(F)\subset\Sigma_1$,
        however, the converse $\Sigma_1\subset\im(F)$ is not so obvious.
        To establish that we 
        use a topological argument by means of the degree mod 2 as follows.
        For the $l$-plane $P\in G(n,l)$ denote the $l-1$-dimensional sphere
        $$
        \S^{l-1}
        (\xi,r,P):=\xi+\{v\in P:|v|=r\}\quad\textnormal{for $\xi\in\R^n$, $r>0$},
        $$
        and observe that  for $\xi\in\Sigma_1,$ $r\in (0,\delta_{tub}),$
        $$
        \S^{n-m-1}(\xi,r,\im(\Phi_1(\xi)))
        =\Psi(\xi,\im(\Phi_1(\xi))) \cap \partial\Ball{0}{r},
        $$
        so that by virtue of Lemma \ref{lem:tubular} \ref{i:Psi-dist}
        $\S^{n-m-1}(\xi,r,\im(\Phi_1(\xi)))\cap\Sigma_1=\emptyset$ 
        and, in addition,  $\S^{n-m-1}(\xi,r,\im(\Phi_1(\xi)))$ 
        and $\Sigma_1$ are nontrivially linked
        for all $r\in (0,\delta_{tub})$, that is, the map
        $$
        \Sigma_1\times\S^{n-m-1}
        (\xi,r,\im(\Phi_1(\xi)))\ni (w,z)\mapsto\frac{w-z}{|w-z|}
        \in\S^{n-1}
        $$
        has non-vanishing degree mod 2, for each $\xi\in\Sigma_1$ and 
        $r\in (0,\delta_{tub})$, since $\Sigma_1$ is a compact  $m$-dimensional
        $C^1$-submanifold without boundary. Since $\HD(\Sigma_1,\Sigma_2)<
        \delta_{tub}/8$ also $\Sigma_2$ and $\S^{n-m-1}(\xi,r,\im(\Phi_1(\xi)))$
        are non-trivially linked for each $\xi\in\Sigma_1$ and for
        all $r\in (\delta_{tub}/2,\delta_{tub})$, because $\HD(\Sigma_1,\Sigma_2)<
        \HD(\Sigma_1,\S^{n-m-1}(\xi,r,\im(\Phi_1(\xi))))$, again by virtue of
        Lemma \ref{lem:tubular} \ref{i:Psi-dist}.
        Therefore, each $n-m$-dimensional
        disk  
        \begin{equation}\label{disk-rep}
            \D^{n-m}(\xi,r,\im(\Phi_1(\xi)))
            :=\xi+\{v\in \im(\Phi_1(\xi)):|v|\le r\}
            = \im\big(\Psi(\xi,\im(\Phi_1(\xi))\cap \Ball{0}{r})\big)
        \end{equation} 
        for $\xi\in\Sigma_1$ and $r\in (\delta_{tub}/2,\delta_{tub})$
        contains at least one point of $\Sigma_2$; 
        see \cite[Lemma 3.5]{KStvdM-GAFA}.
        Take for fixed $\xi\in\Sigma_1$ and $r=3\delta_{tub}/4$
        one of those points
        $$
        z\in\Sigma_2\cap \D^{n-m}(\xi,3\delta_{tub}/4,\im(\Phi_1(\xi))),
        $$
        and use \eqref{disk-rep} to 
        express $z$ as $z=\xi+v=\Psi(\xi,v)$ for some $v\in\im(\Phi_1(\xi))$
        with $|v|<3\delta_{tub}/4$ to find
        $$
        F(z)=\Psi\circ m_0\circ\Psi^{-1}(z)=\Psi\circ m_0(\xi,v)=\Psi(\xi,0)=\xi,
        $$
        which establishes $\Sigma_1\subset\im(F).$

        \medskip\noindent
        \textbf{Remark.} One can also prove that 
        $\Sigma_1 \subseteq \im F$ later, right after proving that $F$
        is bilipschitz (i.e. after proving that $G$ is Lipschitz): once this is
        established, $F$ is a $C^{1,\alpha}$ diffeomorphism onto its image. Thus, the
        image of $F$ is a submanifold of $\R^n$ --- actually it is a submanifold of
        $\Sigma_1$, of the same dimension as $\Sigma_1$. Hence, it is open in
        $\Sigma_1$, which can be seen using a local graph representation; it is also
        closed in $\Sigma_1$ as a continuous image of a compact set. Therefore,
        $\Sigma_1 \without \im F$ is a connected component of $\Sigma_1$; assuming
        $\Sigma_1 \without \im F$ is not empty and using the definition of
        $\Cmn(R,L,d)$ one sees that $\Sigma_1 \without \im F$ is at least $R$ away 
        from $\im F$ which contradicts the assumption $\HD(\Sigma_1,\Sigma_2) \le
        \delta_{tub}$.

        \medskip\noindent
        \emph{Proof of \ref{i:be:G-lip}:}\,
        If $\rho=\HD(\Sigma_1,\Sigma_2)=0$ then $\Sigma_1=\Sigma_2$ and 
        $F=\Id$, and there is nothing to prove. Assume $\rho>0$ from now on.
        Let $x,y \in \Sigma_2$ and set 
        \begin{gather*}
            a := F(x) \,,
            \quad
            b := F(y) \,,
            \quad
            X = \Phi_1(a) \,,
            \quad
            Y = \Phi_1(b) \,.
        \end{gather*}
        Observe that, by~Lemma~\ref{lem:tubular}\ref{i:Psi-dist}, $|x-a| \le 4
        \dist(\Psi(a,x-a), \Sigma_1) \le 4 \rho$ for $\Psi(a,x-a)=x\in\Sigma_1$,
        and in the same way, $|y-b| \le 4
        \rho$, so that we infer immediately
        \begin{equation}
            \label{eq:far-bilip}
            \begin{split}
                |G(x) - G(y)|
                & \le\,  |a-x| + |b-y| \\
                & \le\, 8 \rho 
                \,\le\, 8 \sqrt{\rho}\, |x-y| \qquad
                \mbox{for all $x,y \in \Sigma_2$ with $|x-y| \ge \sqrt{\rho}$}\,.
            \end{split}
        \end{equation}
        Assume now that $x,y\in\Sigma_2$ satisfy $|x-y| < \sqrt{\rho}$. Note that
        by~Lemma~\ref{lem:tubular}\ref{i:Psi-bilip}
        \begin{gather}
            \label{eq:F-lip}
            \lip(F) \le \lip(\Psi)\lip(m_0)\lip(\Psi^{-1})\le 4 \sqrt 2 \,;\\
            \qquad  \qquad \text{hence,} \quad
            |a-b| = |F(x) - F(y)| \le 4 \sqrt 2 |x-y| < 4 \sqrt{2 \rho} \,.\notag
        \end{gather}
        Set
        \begin{gather*}
            \rho_0 := \min \left\{
                2^{-3}\delta_{tub} \,,\,
                (2 L)^{-2/\alpha} \,,\,
                2^{-6}R^2 \,,\,
                2^{-12/\alpha} C_{\text{{\rm ang}}}(L,4)^{-2/\alpha} \,,\,
                2^{-9} \lip(\Phi_1)^{-2}
            \right\} \,.
        \end{gather*}
        As $\delta_{tub}\le 1$, cf.~\eqref{eq:delta-tb} in the proof of 
        Lemma~\ref{lem:tubular}, we have $\rho_0\le \frac 18$.  If we require
        \begin{equation}\label{1st}
            \HD(\Sigma_1,\Sigma_2) = \rho < \rho_0 \,,
        \end{equation}
        then, as $|a-x|<4\rho$, we can use Lemma~\ref{lem:hd-angle} 
        with $A:=4$ to write
        \begin{equation}
            \label{eq:ang-TaS1-TxS2}
            \ang(T_a\Sigma_1,T_x\Sigma_2) \le C_{\text{{\rm ang}}}(L,4) \rho^{\alpha/2} \, .
        \end{equation}
        Moreover, by~\eqref{eq:F-lip} and the choice of $\rho_0$ above,
        \begin{equation}
            \label{eq:ang-YX}
            \begin{split}
                \ang(X,Y)=\|\Phi_1(a)-\Phi_1(b)\| 
                & \le \lip(\Phi_1) |a-b| \\
                & \le 4 \sqrt 2 \lip(\Phi_1) |x-y| 
                \le 4 \sqrt{2\rho} \lip(\Phi_1) < \tfrac 14\, .
            \end{split}
        \end{equation}
        Thus, by \eqref{eq:ang-YX}, \eqref{eq:ang-TaS1-TxS2}, 
        and our choice of $\rho_0$
        \begin{multline}
            \label{eq:ang-Y-TxS2}
            \ang(Y^{\perp},T_x\Sigma_2) 
            \le \ang(Y^{\perp},X^{\perp}) 
            + \ang(X^{\perp}, T_a\Sigma_1) + \ang(T_a\Sigma_1,T_x\Sigma_2)
            \\
            \overset{\eqref{eq:ang-YX}}{\le} 4\sqrt{2\rho}\lip(\Phi_1)
            + \|\Phi_1(a) - (T_a\Sigma_1)^\perp_\natural\| 
            + \ang(T_a\Sigma_1, T_x\Sigma_2)
            \\
            \overset{\eqref{eq:ang-TaS1-TxS2}}{\le} 4 \sqrt{2\rho} \lip(\Phi_1) + \varepsilon + C_{\text{{\rm ang}}}(L,4) \rho^{\alpha/2} < \tfrac 12\, .
        \end{multline}
        Similarly,
        \begin{equation}
            \label{eq:ang-Y-TaS1}
            \ang(Y^{\perp}, T_a\Sigma_1)\,\,
            \le \,\, \ang(Y^{\perp}, X^{\perp}) + \ang(X^{\perp}, T_a\Sigma_1)
            \,\, \overset{\eqref{eq:ang-YX}}{\le} \,\,
            4\sqrt{2\rho} \lip(\Phi_1) + \varepsilon \,\, < \,\, \tfrac 12 \,.
        \end{equation}

        \begin{figure}[!ht]
            \begin{center}
                \includegraphics*[totalheight=6.5cm]{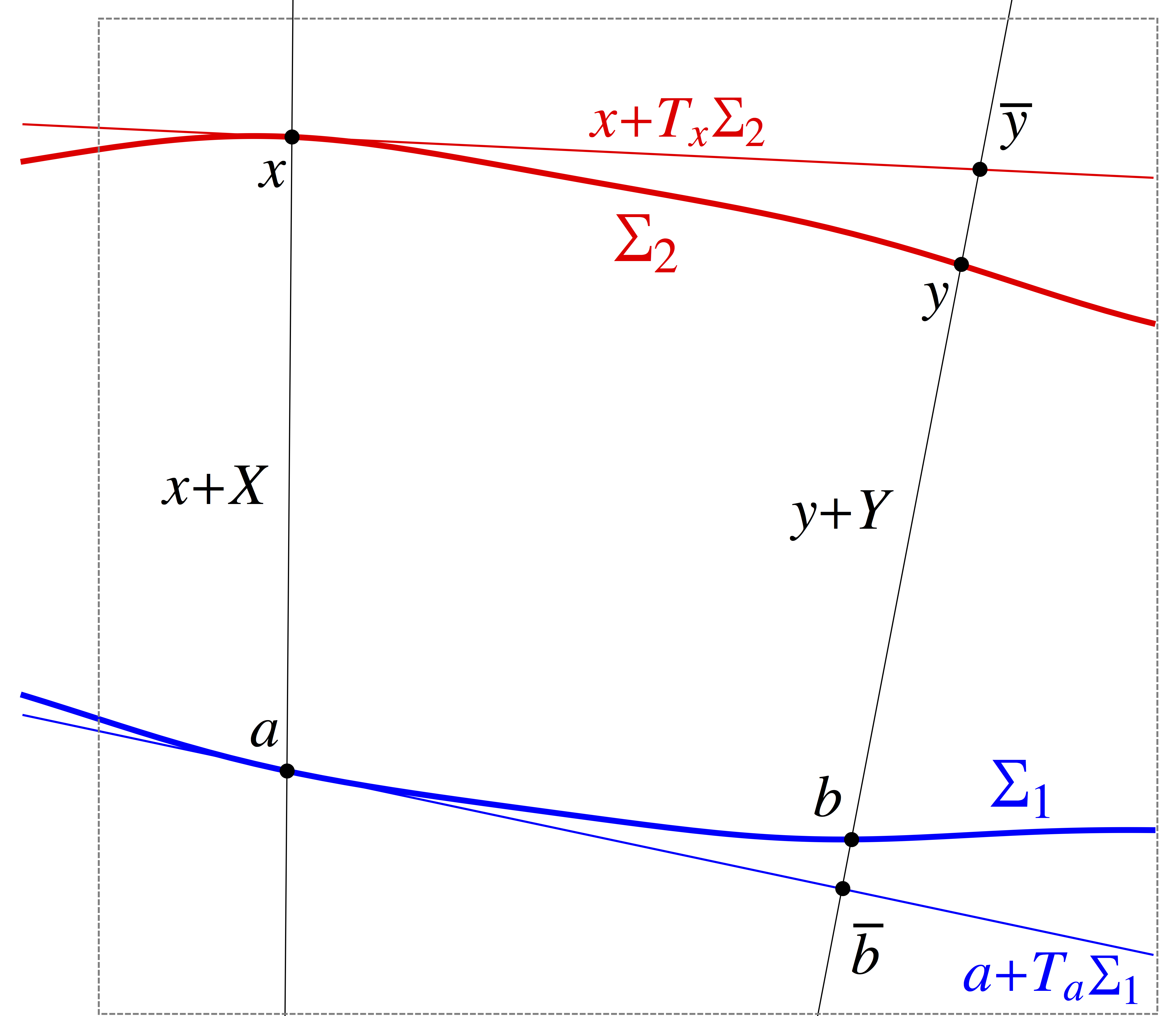}
            \end{center}

            \caption{\label{fig:lipG} An enlarged fragment 
              of \textsc{Figure~\ref{fig:mapF}}.
              We have $G(x)=a-x$, $G(y)=b-y$. 
              However, to prove that $G$ is Lipschitz, 
              we do not deal with $(a-x)-(b-y)$ directly. Instead, we use
              Proposition~\ref{prop:tales2}
              to estimate $|(a-x)-(\bar{b}-\bar{y})|$, 
              and add an error term $|b-\bar{b}|
              +|y-\bar{y}|$, which is small as both $\Sigma_1$ and $\Sigma_2$ are 
              of class $\Cmn(R,L,d)$. The final dependence of 
              $\lip(G)$ on a power of 
              $\rho$ is due to this error term. 
            }
        \end{figure}

        These angle conditions imply by Lemma \ref{lem:angle-isomorphism}
        that $Y \cap T_a\Sigma_1 = Y \cap T_x\Sigma_2 =
        \{0\}$, Therefore, there exist points $\bar{b}$, $\bar{y} \in \R^n$ (see Figure~\ref{fig:lipG}) such
        that
        \begin{gather}\label{oblproj}
            (y + Y) \cap (a + T_a\Sigma_1) = \{ \bar{b} \} \,,
            \quad \text{and} \quad
            (y + Y) \cap (x + T_x\Sigma_2) = \{ \bar{y} \} \,.
        \end{gather}
        Indeed, the characterisation $\{P(y-a)\}=((y-a)+Y)\cap T_a\Sigma_1$
        of the well-defined oblique projection $P:\R^n\to T_a\Sigma_1$ along $Y$
        (see Remark \ref{rem:oblique-proj})
        immediately gives $\bar{b}:=P(y-a)+a$, and similarly one finds $\bar{y}.$

        To prove that $G$ is Lipschitz we need to estimate $|G(x) - G(y)| = |(x-a) -
        (y-b)|$. To this end, we shall first estimate $|(x-a) - (\bar y - \bar b)|$
        treating $|(\bar y - \bar b) - (y-b)| \le |\bar b - b| + |\bar y - y|$ as
        a~small error term. Employing \eqref{eq2-tales2} of 
        Proposition~\ref{prop:tales2} with
        \begin{gather*}
            U:= Y \,,\,
            \quad
            V := X \,,\,
            \quad
            W := T_x\Sigma_2 \,,\,
            \quad
            T := T_a\Sigma_1 \,,\,
            \\
            \mathbf u := \bar y - \bar b \,,\,
            \quad
            \mathbf v := x - a  \,,\,
            \quad
            \mathbf w := \bar y - x \,,\,
            \quad
            \mathbf t := \bar b - a \,,\, 
        \end{gather*}
        in combination with \eqref{eq:ang-TaS1-TxS2}, \eqref{eq:ang-YX},
        \eqref{eq:ang-Y-TxS2}, \eqref{eq:ang-Y-TaS1} to estimate the angles
        by our choice of $\rho_0$
        \begin{gather*}
            \ang(W,T)\le\theta:=C_{\text{{\rm ang}}}(L,4) \rho^{\alpha/2}, 
            \qquad 
            \ang(U,V)\le\gamma:=4 \sqrt 2 \lip(\Phi_1) |x-y|,\\
            \qquad 
            \max\Big\{\!\ang(T,U^\perp),\ang(T,V^\perp)\Big\}\le \lambda 
            := \frac 12,
        \end{gather*}
        we obtain
        \begin{multline}
            \label{eq:initial-estimate}
            |(x-a) - (\bar y - \bar b)| 
            \\
            \overset{\eqref{eq2-tales2}}{\le} 8 \sqrt 2 \lip(\Phi_1) |x-y| \left( |\bar y - \bar b| + 2 C_{\text{{\rm ang}}}(L,4) \rho^{\alpha/2} |\bar  y - x| \right) 
            + 2 C_{\text{{\rm ang}}}(L,4) \rho^{\alpha/2} |\bar  y - x| \,.
        \end{multline}
        By \eqref{eq:F-lip} and the choice of $\rho_0$ we have 
        \[
        |a-b|<4\sqrt{(2\rho)} <4 (2^{-5}R^2)^{1/2}<R\, , 
        \qquad
        |a-b|^{1+\alpha} \le (4\sqrt 2)^2 |x-y|^{1+\alpha}= 2^5 |x-y|^{1+\alpha}\, .
        \]
        Thus, 
        since $\bar y - y \in Y$ and $\bar b - b= (\bar b -y)+(y-b) \in Y$
        (we use \eqref{oblproj} and note that $y=\Psi(b,y-b)$ 
        with $y-b\in\Phi_1(b)=Y$), it follows
        from~\eqref{eq:F-lip}, \eqref{eq:ang-Y-TxS2} and~~\eqref{eq:ang-Y-TaS1} that
        \begin{equation}
            \label{est:error}
            \begin{split}
                |\bar y - y| &\le 2 \dist(y, x + T_x\Sigma_2) \le 2L|x-y|^{1+\alpha} 
                \\ \text{and} \quad
                |\bar b - b| &\le 2 \dist(b, a + T_a\Sigma_1) 
                \le 2L|a-b|^{1+\alpha} \le 2^6 L |x-y|^{1+\alpha} \,,
            \end{split}
        \end{equation}
        where we estimated similarly as in \eqref{16}.
        Hence, since we observed $|y-b|\le 4\rho$ earlier, and   $|x-y|<\sqrt{\rho} \le \frac 1{\sqrt 8} <1$, we obtain 
        \begin{gather*}
            |\bar y - \bar b| \le |\bar y - y| +|y-b| + |\bar b - b| 
            \le 2^7 L |x-y|^{1+\alpha}
            +4\rho   < (2^7L+4)\rho^{\alpha/2} \, ,
            \\
            |\bar y - x| \le |\bar y - y| + |y - x| 
            \le |y - x| \left( 1 + 2 L \rho^{\alpha/2} \right) 
            \le 2 |x-y| < 2 \sqrt{\rho}  < 1 \,.
        \end{gather*}
        Therefore, plugging these 
        two estimates into~\eqref{eq:initial-estimate}, 
        and adding the error $ |\bar y - y| + |\bar b - b|$ which can be 
        estimated by \eqref{est:error},
        we compute 
        \begin{eqnarray*}
            \lefteqn{|G(x) - G(y)|}\\
            & = &
            |(x-a) - (y-b)| \le |(x-a) - (\bar y - \bar b)|
            + |\bar y - y| + |\bar b - b|
            \\
            & \overset{\eqref{eq:initial-estimate},\eqref{est:error}}{\le} & 
            |x-y|\rho^{\alpha/2}\Big\{8 \sqrt 2 \lip(\Phi_1)
            \big[2^7L+4  +2 C_{\text{{\rm ang}}}(L,4)\big]+
            4C_{\text{{\rm ang}}}(L,4)+2L+2^6L\Big\} \,.
        \end{eqnarray*}
        As $\rho<1$, taking into account  $C_{\text{{\rm ang}}}(L,4)=257L+8$ 
        (cf. Lemma~\ref{lem:hd-angle}), 
        we finally obtain an estimate of the Lipschitz constant of $G$,
        \begin{equation}
            \label{eq:G-lip}
            |G(x)-G(y)|\le C_l\rho^{\alpha/2} |x-y|, 
            \qquad C_l=C_l(L,\lip(\Phi_1)):=10^4(\lip(\Phi_1)+1)(L+1)\, .
        \end{equation}

        \medskip\noindent
        \emph{Proof of \ref{i:be:G-small}:}
        Directly from the definition of $\Psi$ we infer $\Psi(F(x),x-F(x)) =
        F(x)+x-F(x)=x$ for any $x\in\Sigma_2$, so that we obtain from Lemma
        \ref{lem:tubular}\ref{i:Psi-dist} 
        \begin{gather*}
            \dist(x,\Sigma_1) = \dist(\Psi(F(x),x-F(x)),\Sigma_1) 
            \overset{L.\ref{lem:tubular}\ref{i:Psi-dist}}{\ge} \tfrac 14 
            |x - F(x)|
            = \tfrac 14 |G(x)| \quad\forall x\in\Sigma_2\,.
        \end{gather*}

        \medskip\noindent
        \emph{Proof of \ref{i:be:F-diffeo}:} Since $F$ is a composition of
        $C^1$-smooth functions it is $C^1$-smooth. We can find $\rho_G
        =\rho_G(R,L,\alpha,\lip(\Phi_1))\in (0,\rho_0)$
        so small that   
        \begin{equation}
            \label{def:rhoG}
            C_l \rho^{\alpha/2} < 1\qquad\mbox{for all $\rho \in (0,
              \rho_G)$},
        \end{equation}
        and then
        \begin{multline*}
            |F(x)-F(y)|\le |x-y|+|F(x)-x-(F(y)-y)| \\ =|x-y|+|G(x)-G(y)| 
            \overset{\eqref{eq:G-lip}}{\le} \big(1+C_l\rho^{\alpha/2}\big)|x-y|.
        \end{multline*}
        The lower estimate in \ref{i:be:F-diffeo} follows in the same manner; hence
        $F$ is bilipschitz and, in~consequence, a diffeomorphism.
    \end{proof}

    As a corollary we can can establish a bound on the Hausdorff-distance
    $\HD(\Sigma_1,\Sigma_2)$ under which two submanifolds $\Sigma_1,\Sigma_2\in
    \Cmn(R,L,d)$ are actually ambient isotopic. Moreover, in Lemma \ref{lem:amb-diff}
    we construct a global diffeomorphism of the ambient space mapping
    $\Sigma_2$ onto $\Sigma_1$. Both results will be essential ingredients in the
    proof of  Theorem \ref{thm:diff}.

    \begin{cor}[\textbf{ambient isotopies}]
        \label{cor:isotopy}
        For $R,L,d\in (0,\infty)$, $\alpha\in (0,1]$, and $\varepsilon\in
        (0,1/100)$ let 
        $\Sigma_1, \Sigma_2\in\Cmn(R,L,d) $ with $\varepsilon$-normal map
        $\Phi_1$ for $\Sigma_1$, such that $\HD(\Sigma_1,\Sigma_2)\in (0,
        \rho_G)$, where $\rho_G=\rho_G(R,L,\alpha,\eps,\lip(\Phi_1))$ 
        is the 
        constant of  Lemma~\ref{lem:bilip-est}. 
        Then
        $\Sigma_1$ and $\Sigma_2$ are $C^1$-ambient isotopic.
    \end{cor}

    \begin{proof}
        According to \cite[Theorem 1.2]{Bl09a} it suffices to come up with
        a $C^1$-isotopy $h:\Sigma_2\times [0,1]\to\R^n$,  i.e., a family
        of $C^1$-embeddings $h_t(\cdot):=h(\cdot,t):\Sigma_2\to\R^n$, with
        \begin{equation}\label{iso}
            \Sigma_1 = h(\Sigma_2\times\{0\}) \quad\textnormal{and}\quad
            h(\Sigma_2\times\{1\})=\Sigma_2.
        \end{equation}
        Indeed, the map $h(x,t):=\Psi\circ m_t\circ\Psi^{-1}|_{\Sigma_2}(x)$
        for $(x,t)\in\Sigma_2\times [0,1]$, and with $m_t(y,v)=(y,tv)$ and
        $\Psi:=\Psi_{\delta_{tub}}[\Sigma_1,\Phi_1]$ for
        $y,v\in\R^n$ will do. Here $\delta_{tub}$ 
        is the constant from Lemma
        \ref{lem:tubular} defined in \eqref{eq:delta-tb}, and $\Phi_1$ is
        an $\varepsilon$-normal map for $\Sigma_1$.

        Observe that part \ref{i:Psi-onto} of Lemma \ref{lem:tubular}
        implies that $\Sigma_2\subset\Psi(N_{\delta_{tub}}(\Sigma_1,\Phi_1))$,
        since we have $\HD(\Sigma_1,\Sigma_2)<\rho_G<\delta_{tub}/8$ (see Lemma 
        \ref{lem:bilip-est}).
        Therefore, $\Psi^{-1}$ is a well-defined $C^1$-map 
        in an open neighbourhood
        of $\Sigma_2$, which implies that $h$ itself as a composition 
        of $C^1$-maps is 
        of class  $C^1.$ With Lemma \ref{lem:bilip-est} \ref{i:be:F-diffeo} 
        we obtain
        $h(\Sigma_2\times\{0\})=F(\Sigma_2)=\Sigma_1$, and $h(\cdot,0)$
        is a bilipschitz diffeomorphism from $\Sigma_2$ onto $\Sigma_1.$
        Moreover, one immediately sees that
        $h(x,1)=\Psi\circ m_1\circ\Psi^{-1}(x)=x$ for all $x\in\Sigma_2$
        by the very definition of $m_t$ for $t=1$, so that
        $h(\Sigma_2\times\{1\})=\Sigma_2$, which proves \eqref{iso}.

        So, it remains to be shown that $h(\cdot,t):\Sigma\to\R^n$ is an
        embedding for each $t\in (0,1)$. Note that $\Psi:N_\delta(\Sigma_1,\Phi_1)
        \to\R^n$ is bilipschitz  for all $\delta\in (0,\delta_{tub}]$
        by Lemma \ref{lem:tubular}  \ref{i:Psi-bilip}, and hence
        so is $\Psi^{-1}$ on $\Psi(N_{\delta_{tub}}(\Sigma_1,\Phi_1))$. 
        In addition,  $m_t$ is bilipschitz for $t \in (0,1)$, and
        $m_t(N_\delta(\Sigma_1,\Phi_1))\subset N_\delta(\Sigma_1,\Phi_1)$
        for all $t\in [0,1],  $ $\delta\in (0,\delta_{tub}].$
        Recall again from Lemma \ref{lem:tubular} \ref{i:Psi-onto} that
        $\HD(\Sigma_1,\Sigma_2)<\rho_G<\delta_{tub}/8$ implies
        that
        $\Sigma_2\subset\Psi(N_{2\HD(\Sigma_1,\Sigma_2)}(\Sigma_1,\Phi_1))$,
        so that $\Psi^{-1}|_{\Sigma_2}$ is just the restriction of a $C^1$-bilipschitz
        map, and in consequence $h(\cdot,t)$ is bilipschitz and $C^1$-smooth,
        and therefore a $C^1$-diffeomorphism onto its image $h(\Sigma_2,t)$
        for each $t\in [0,1]$. Consequently, $h(\cdot,t):\Sigma_2\to\R^n$
        is an embedding for each $t\in [0,1].$
    \end{proof}

    \begin{lem}[\textbf{diffeomorphisms of the ambient space}]
        \label{lem:amb-diff}
        For $R,L,d > 0$, $\alpha \in (0,1]$ there exist  constants
        $\rho_g:=\rho_g(R,L,\alpha,n,m)$ 
        and $C_J=C_J(R,L,\alpha,n,m)\ge 0$,
        such that for any two manifolds $\Sigma_1,\Sigma_2\in\Cmn(R,L,d)$ with
        $\rho:=\HD(\Sigma_1,\Sigma_2)\in (0,\rho_g]$ there exists a 
        bilipschitz $C^1$-diffeomorphism $J : \R^n \to \R^n$
        satisfying
        \begin{enumerate}
        \item $J(\Sigma_2) = \Sigma_1$,
        \item $J(x) = x$ for $x \in \R^n \without (\Sigma_2 + \Ball{0}{\rho_g})$,
        \item $(1-C_J\rho^{\alpha/2})|z_1-z_2| \le |J(z_1)-J(z_2)|
            \le (1 + C_{J} \rho^{\alpha/2})|z_1-z_2|$ for all $z_1,z_2\in\R^n$. 
        \end{enumerate}
    \end{lem}
    The constant $\rho_G$ was introduced in Lemma \ref{lem:bilip-est}.

    \begin{proof}
        Set $\varepsilon:=1/200.$ Lemma \ref{lem:mollify} guarantees the existence
        of  $\varepsilon$-normal maps $\Phi_i : \Sigma_i \to \Gr$ for
        $\Sigma_i$, $i=1,2$.  
        Define $\Psi_2 = \Psi_{\delta_{tub}}[\Sigma_2,\Phi_2]$ as
        in Definition~\ref{def:Psi-N}. 

        Choose $\rho_0 =
        \rho_0(R,L,\alpha,\lip(\Phi_1),\lip(\Phi_2))\in (0,\min\{\delta_{tub}/16,\rho_G/2\})$ so small that
        \begin{gather}
            \label{eq:rho-0}
            4 C_{l} \rho^{\alpha/2} < \varepsilon=\frac{1}{200} 
            \quad\textnormal{for all $\rho\in (0,\rho_0]$},
        \end{gather}
        where we denote by $\delta_{tub}=\delta_{tub}(R,L,\alpha,\lip(\Phi_2))$
        the tubular radius for $\Sigma_2$ established in Lemma \ref{lem:tubular} for our
        fixed $\eps=1/200.$ Moreover, $\rho_G=\rho_G(R,L,\alpha,\lip(\Phi_1))$ and 
        $C_{l}=C_{l}(L,\lip(\Phi_1))$ are the constants estimating
        the maps $F,G:\Sigma_2\to\R^n$ 
        in Lemma~\ref{lem:bilip-est} for $\eps=1/200$.
        Consider the projections $\pi_1,\pi_2 \colon 
        N_2:=N_{\delta_{tub}/2}(\Sigma_2,\Phi_2) \to
        \R^n$ via  $\pi_1(x,v):=x$ and $\pi_2(x,v):=v$ for $(x,v)\in N_2$, 
        define the map $\lambda:N_2\to\R^n$ by $\lambda(x,v):=F(x)+v,$ and finally,
        \begin{gather*}
            \tilde{J} : \Sigma_2 + \Ball{0}{\rho_0} \to \R^n \,,
            \quad
            \tilde{J} = \lambda \circ \Psi_2^{-1}
            \\
            \text{and the map} \quad
            I : \Sigma_2 + \Ball{0}{\rho_0} \to \R^n \,,
            \quad
            I(z) = \tilde{J}(z) - z \,.
        \end{gather*}
        measuring the deviation of $\tilde{J}$ from the identity.
        According to 
        Lemma~\ref{lem:bilip-est}\ref{i:be:G-small} one has for $x \in \Sigma_2$ 
        \begin{equation}
            \label{eq:sup-I}
            \begin{split}        
                |G(x)| = |F(x) - x| 
                & \le 4 \dist(x,\Sigma_1) \\
                & \le 4\HD(\Sigma_2,\Sigma_1)=
                4  \rho
                \le 4\rho_0<\delta_{tub}/4 \quad
                \textnormal{for all $\rho\in (0,\rho_0]$}\,,
            \end{split}
        \end{equation}
        and therefore, for $z\in\Sigma_2+\Ball{0}{\rho_0}$ with $\Psi_2^{-1}(z)=(x,v),$
        \begin{gather}\label{41A}
            I(z) = \tilde{J}(z) - z = (F(x) + v) - (x+v) = F(x) - x 
            = G(x) = G \circ \pi_1 \circ \Psi_2^{-1}(z) \,,
        \end{gather}
        so
        \begin{equation}\label{41B}
            |I(z)|\le 4\rho <\delta_{tub}/4,
        \end{equation}
        whence $\tilde{J}(\Sigma_2+\Ball{0}{\rho_0})\subset\Sigma_2+
        \Ball{0}{\delta_{tub}/2}.$ 
        The identity \eqref{41A} together with 
        Lemma~\ref{lem:tubular}\ref{i:Psi-bilip} applied to $\Sigma_2$ and $\Psi_2$
        and~Lemma~\ref{lem:bilip-est}\ref{i:be:G-lip} implies
        \begin{gather}
            \label{eq:lip-I}
            \lip(I) \le \lip(G) \lip(\Psi_2^{-1}) \le 4 C_{l} \rho^{\alpha/2} \,.
        \end{gather}
        Thus, we can estimate the difference $\tilde{J}(z_1)-\tilde{J}(z_2)=
        I(z_1)-I(z_2)+z_1-z_2$ using \eqref{41A} for $z_1,z_2\in\Sigma_2+
        \Ball{0}{\rho_0}$ as
        \begin{gather}
            \label{eq:bilip-J}
            (1 - 4 C_{l} \rho^{\alpha/2}) |z_1 - z_2|
            \le |\tilde J(z_1) - \tilde J(z_2)|
            \le (1 + 4 C_{l} \rho^{\alpha/2}) |z_1 - z_2|\,,
        \end{gather}
        so that
        by our choice of $\rho_0$ in
        \eqref{eq:rho-0},
        $\tilde{J}$ turns out to be bilipschitz, and since $\Psi_2^{-1}$ is $C^1$
        on $\Psi_2(N_2)$ and
        $$
        \Sigma_2+\Ball{0}{\rho_0}\subset\Sigma_2
        +\Ball{0}{\delta_{tub}/4}\overset{Lem. 
          \ref{lem:tubular}}{\subset}\Psi_2(N_2),
        $$
        and $\lambda$ is $C^1$ on $N_2$, the map $\tilde{J}$ is 
        a $C^1$-diffeomorphism
        from $\Sigma_2+\Ball{0}{\rho_0}$ onto its  image.
        Note that this image $\tilde{J}(\Sigma_2+\Ball{0}{\rho_0})$ 
        contains $\Sigma_1$,
        since by Lemma \ref{lem:bilip-est}, $F$ maps $\Sigma_2$ diffeomorphically
        onto $\Sigma_1$. In particular, for any $\xi\in\Sigma_1$ there is exactly
        one $x\in\Sigma_2$ such that $F(x)=\xi$, so that for
        $z=x+0$ one has $\tilde{J}(z)=\lambda(x,0)=F(x)+0=\xi.$ This also shows
        that $\tilde{J}(\Sigma_2)=\Sigma_1.$

        To construct the global diffeomorphism we smoothly extend $\tilde{J}$ to
        all of $\R^n$ by the identity in the following way. 
        Let $\phi\in  C^1(\R)$
        be a cut-off function satisfying
        $0\le\phi\le 1 $ on $\R$, $\phi(t)=0$ for $t\le \rho_0/8,$
        $\phi(t)=1$ for $t\ge\rho_0/4,$ and $|\phi'(t)|\le 16/\rho_0$ for all
        $t\in\R.$ Define 
        $\eta:\Sigma_2+\Ball{0}{\rho_0}\to\R$ as
        $\eta(z):=|\pi_2\circ\Psi_2^{-1}(z)|,$ and the transition term
        $T:\Sigma_2+\Ball{0}{\rho_0}\to\R^n$ as $T(z):=\phi(\eta(z))I(z),$
        which is of class $C^1$ since $\phi(\eta(z))$ vanishes for
        $0\le\eta(z)\le \rho_0/8.$
        In addition, we can estimate the Lipschitz constant of the transition term
        using Lemma \ref{lem:tubular}\ref{i:Psi-bilip} 
        for $\Sigma_2$ and $\Psi_2$, \eqref{41B}, 
        and \eqref{eq:lip-I} as
        \begin{gather}
            \label{eq:lip-E}
            \lip(T) 
            \le \lip(\phi \circ \eta) \|I\|_{\infty} + 
            \|\phi \circ \eta\|_{\infty} \lip(I) 
            \le \frac{16^2}{\rho_0} \rho + 4 C_{l} \rho^{\alpha/2}
            \le  C_T \rho^{\alpha/2},\quad\rho\in (0,\rho_0] \,
        \end{gather}
        for $C_T:=16^2/\rho_0^{\alpha/2}+4C_l.$
        The global
        diffeomorphism $J:\R^n\to\R^n$ can now be defined as
        \begin{gather*}
            J(z) := 
            \left\{
                \begin{array}{ll}
                    \tilde{J}(z) -  T(z) & \text{for } 
                    z \in \Sigma_2 + \Ball{0}{\rho_0} 
                    \\
                    z & \text{otherwise} \,,
                \end{array}
            \right.
        \end{gather*}
        which is of class $C^1$ since $T(z)=I(z)$ (and hence
        $\tilde{J}(z)-T(z)=z$) if $z$ is contained in the transition
        zone $\Sigma_2+\Ball{0}{\rho_0}\without\Ball{0}{\rho_0/2}$.
        Indeed, then $z=x+v$ for some $(x,v)\in N_2$ satisfying, by
        Lemma \ref{lem:tubular}\ref{i:Psi-bilip},
        $$
        \frac{\rho_0}{2}\le\dist(z,\Sigma_2)=\dist(\Psi_2(x,v),\Sigma_2)
        \le|\Psi_2(x,v)-\Psi_2(x,0)|
        \le\sqrt{2}|(x-x,v-0)|=\sqrt{2}|v|,
        $$
        so that $|v|\ge\rho_0/(2\sqrt{2})>\rho_0/4,$ from which
        $
        \eta(z)=|\pi_2\circ\Psi_2^{-1}(z)|=|\pi_2(x,v)|=|v|>\rho_0/4
        $
        follows, and thus $\phi(\eta(z))=1$ for such $z$ in the transition zone.
        Combining \eqref{eq:bilip-J} with \eqref{eq:lip-E} we arrive at the desired
        bilipschitz estimate
        \begin{gather*}
            \big{(}1 - (4 C_{l} + C_T) \rho^{\alpha/2}\big{)} |z_1 - z_2|
            \le |J(z_1) - J(z_2)|
            \le \big{(}1 + (4 C_{l} +  C_T) \rho^{\alpha/2}\big{)} |z_1 - z_2| \,,
        \end{gather*}
        which establishes Part (3) of Lemma \ref{lem:amb-diff} if we set
        $C_J:=4C_l+C_T,$ and if we choose
        $\rho_g=\rho_g(R,L,\alpha,\lip(\Phi_1),\lip(\Phi_2))\in (0,\rho_0)$ so small
        that $C_J\rho^{\alpha/2}<1$ for all $\rho\in (0,\rho_g].$ Recall that we
        have fixed $\varepsilon=1/200$ and that $\lip(\Phi_1)$ and $\lip(\Phi_2)$
        depend only on $\varepsilon,R,L,\alpha,m,n$ according to Lemma
        \ref{lem:mollify}, which means that $\rho_0$ and hence also $\rho_g$ and
        $C_J$ actually depend on $R,L,\alpha,m,n$ only.
    \end{proof}

    \begin{rem}
        Inspecting the above proof one can see that
        \begin{gather*}
            \lip(J - \Id) = \lip(I - T) 
            = \lip(I (1 - \phi \circ \eta)) 
            \le \frac{16^2C_l}{\rho_0} \rho^{\alpha/2} \,.
        \end{gather*}
    \end{rem}

    \begin{proof}[Proof of Theorem~\ref{thm:diff}]
        According to Definition \ref{def:C1a-conv} we have 
        $\HD(\Sigma_0,\Sigma_j) \to 0$ as $j \to
        \infty$, so that  we can choose $j_0$ such  that 
        $\HD(\Sigma_0,\Sigma_j) \le \rho_g<\rho_G/2$
        for all $j \ge j_0$, where $\rho_G$ is 
        the constant from Lemma \ref{lem:bilip-est} for fixed $\eps:=1/200$.
        Therefore, by Corollary~\ref{cor:isotopy}, $\Sigma_j$ 
        is  ambient isotopic to $\Sigma_0$ for all $j \ge j_0$. 
        Moreover, 
        by means of Lemma~\ref{lem:amb-diff} we can find  for each $j \ge j_0$ 
        a~$C^1$-diffeomorphism of the ambient space $J_j : \R^n \to \R^n$ such
        that $\bilip(J_j) \le 1 + C_J\HD(\Sigma_1,\Sigma_2)^{\alpha/2}$ and 
        $J_j(\Sigma_j) = \Sigma_0$.
    \end{proof}

    \section{Semicontinuity}   
    \label{sec:lsc}

    \subsection{Preliminaries}
    \label{sec:5.1}

    Before passing to the proof of Theorem~\ref{thm:lsc}, we set up some
    notation, and prove two technical lemmata which explain how our discrete
    curvatures change under small bilipschitz perturbations of the identity map.

    \begin{defin}
        \label{def:meas-deriv}
        Let $N \in \N$ and $\mu, \nu \in \RM(\R^{N})$ be Radon measures. We~set
        \begin{gather*}
            D(\nu,\mu,x) = \lim_{r \downarrow 0}
            \frac{\nu(\Kugel^{N} (x,r))}{\mu(\Kugel^{N} (x,r))} \,,
        \end{gather*}
        where we interpret $0/0 = 0$.
    \end{defin}

    \begin{defin}
        \label{def:cart-power}
        For any function $F : X \to Y$ and any $l \in \N$, 
        we~define $F^{\times l} : X^l \to Y^l$ by the formula 
        $F^{\times l}(x_1,\ldots,x_l) = (F(x_1),\ldots,F(x_l))$. 
    \end{defin}

    \begin{defin}
        \label{def:DC_l}
        Let $l \in \{1,2,\ldots,m+2\}$ and $\Sigma \in \Cmn$. 
        For $l\le m+1$,  we define
        \begin{gather*}
            \DC_l[\Sigma](y_0,\ldots,y_{l-1}) 
            = \sup_{y_l,\ldots,y_{m+1} \in \Sigma} \DC(y_0,\ldots,y_{m+1}) 
            \quad \text{for $y_0, \ldots, y_{l-1} \in \Sigma$}
        \end{gather*}
        for $\DC$  given in the introduction by formula \eqref{def:DC}. 
        We additionally set $\DC_{m+2}[\Sigma]\equiv \DC$.
    \end{defin}

    \begin{defin}
        Let $A \subset \R^n$ and $l \in \N$. 
        We define \emph{the $l$-diagonal of $A$}
        \begin{gather*}
            \Delta^l A = 
            \{ (x_0,\ldots,x_l) \in A^l : x_0 = x_1 = \cdots = x_l \} \,.
        \end{gather*}
    \end{defin}

    Formally, the integrands $\DC$ and $\Rtp^{-1}$ are defined
    only off the diagonal $\Delta^l\Sigma$. It does not matter
    how one defines them on the diagonal: it does not
    affect the integral, since $\HM^{ml}(\Delta^l \Sigma) = 0$.
    Below, we also freely use the equivalence of measures
    $\HM^{ml} \narrow \Sigma^l \simeq (\HM^m \narrow \Sigma)^l$ 
    which holds as long as $\Sigma$ is an embedded submanifold 
    due to~\cite[3.2.23]{Fed69}. (Actually, it holds even if $\Sigma$ 
    is just a~subset of the image of a \emph{single} 
    Lipschitz function.)

    \begin{rem}
        \label{rem:RN-theo}
        Let $N \in \N$ and $\mu$, $\nu$ be Radon measures on $\R^N$. The
        Radon-Nikodym theorem (cf.~\cite[Theorem~2.12]{Mat95}) implies that if
        $\nu$ is absolutely continuous with respect to~$\mu$, then for any $f
        \in L^1(\R^N,\nu)$
        \begin{gather*}
            \int f(x)\ d\nu(x) = \int f(x) D(\nu,\mu,x)\ d\mu(x) \,.
        \end{gather*}
    \end{rem}

    \begin{lem}
        \label{lem:density}
        Let $\Sigma_1, \Sigma_2 \in \Cmn$, $F : \R^n \to \R^n$ be a~bilipschitz
        homeomorphism such that $F(\Sigma_1) = \Sigma_2$. Set $\mu = \HM^{ml}
        \narrow \Sigma_2^{l}$ and $\nu = (F^{\times l})_*(\HM^{ml} \narrow
        \Sigma_1^{l})$. Then $\mu$ and $\nu$ are mutually absolutely continuous and
        \begin{gather*}
            D(\nu,\mu,x) \le \lip(F^{-1})^{ml}
            \quad \text{and} \quad
            D(\mu,\nu,x) \le \lip(F)^{ml}
        \end{gather*}
        for all $x \in \R^{nl}$.
    \end{lem}

    \begin{proof}
        If $x \in \R^{nl} \without \Sigma_2^l$, then $\dist(x,\Sigma_2^l) > 0$, so
        for $0 < r < \dist(x,\Sigma_2^l)$ we have $\mu(\Kugel^{ml} (x,r)) = 0 =
        \nu(\Kugel^{ml} (x,r))$ and, according to Definition~\ref{def:meas-deriv},
        $D(\mu,\nu,x) = D(\nu,\mu,x) = 0$.

        Note that $\lip(F^{\times l}) = \lip(F)$ and $(F^{\times l})^{-1} =
        (F^{-1})^{\times l}$. Furthermore, observe that for $x \in \Sigma_2^l$ and
        $0 < r < \infty$
        \begin{gather*}
            \Sigma_1^l \cap (F^{\times l})^{-1}(\Kugel^{ml} (x,r))
            = \Sigma_1^l \cap (F^{\times l})^{-1}(\Sigma_2^l \cap \Kugel^{ml} (x,r))
            = (F^{\times l})^{-1}(\Sigma_2^l \cap \Kugel^{ml} (x,r)) \,;
        \end{gather*}
        hence
        \begin{gather*}
            \frac{\nu(\Kugel^{ml} (x,r))}{\mu(\Kugel^{ml}(x,r))}
            = \frac{\HM^{ml}((F^{\times l})^{-1}(\Sigma_2^l \cap \Kugel^{ml} (x,r)))}
            {\HM^{ml}(\Sigma_2^l \cap \Kugel^{ml} (x,r))}
            \le \lip(F^{-1})^{ml}
        \end{gather*}
        and consequently $D(\nu,\mu,x) \le \lip(F^{-1})^{ml}$. The estimate for
        $D(\mu,\nu,x)$ is obtained by writing
        \begin{gather*}
            \Sigma_2^l \cap \Kugel^{ml}(x,r)  
            = F^{\times l} 
            ( (F^{\times l})^{-1}(\Sigma_2^l \cap \Kugel^{ml}(x,r))) \,.
            \qedhere
        \end{gather*}
    \end{proof}

    \begin{lem}
        \label{lem:DC-conv}
        Let $\Sigma_1, \Sigma_2 \in  \Lmn$, $0 < \varepsilon < 1/2$. 
        Assume $F :
        \Sigma_1 \to \R^n$ is bilipschitz and satisfies 
        $F(\Sigma_1) = \Sigma_2$,
        $F(z) = z + G(z)$ for $z \in \Sigma_1$ and some 
        $G : \Sigma_1 \to \R^n$
        having $\lip(G) \le \varepsilon$. 
        Then for any $T \in \Sigma_1^l \without
        \Delta^l \Sigma_1$ and $l \in \{2,\ldots,m+2\}$ 
        \begin{gather*}
            \left| \DC_l[\Sigma_2](F^{\times l}(T)) - \DC_l[\Sigma_1](T) \right|
            \le \varepsilon C_{\ref{lem:DC-conv}} \diam(\simp T)^{-1} \,,
        \end{gather*}
        where $C_{\ref{lem:DC-conv}} = C_{\ref{lem:DC-conv}}(m) > 0$.
    \end{lem}

    \begin{proof}
        First we treat the case $l = m+2$. Let $T = (x_0,\ldots,x_{m+1}) \in
        \Sigma_1^l \without \Delta^l \Sigma_1$.  Set $u_i = x_i - x_0$, $v_i =
        F(x_i) - F(x_0)$ and $e_i = G(x_i) - G(x_0)$ for $i = 1, 2, \ldots,m+1$.
        Observe that $v_i = u_i + e_i$ and that $|e_i| \le \varepsilon |u_i|$.
        We compute
        \begin{multline*}
            \left| \HM^{m+1}( \simp F^{\times l}(T) ) - \HM^{m+1}(\simp T)
            \right| = \tfrac 1{(m+1)!} \Big| |v_1 \wedge \cdots \wedge v_{m+1}|
                - |u_1 \wedge \cdots \wedge u_{m+1}| \Big|
            \\
            \le \tfrac 1{(m+1)!} \diam(\simp T)^{m+1} \sum_{i=1}^{m+1}
            \tbinom{m+1}{i} \varepsilon^i \le \tfrac{2^{m+1}}{(m+1)!}
            \diam(\simp T)^{m+1} \varepsilon \,.
        \end{multline*}
        Since $(1-\varepsilon) \diam(\simp T) \le \diam(\simp F^{\times l}(T))
        \le (1+\varepsilon) \diam(\simp T)$ and recalling $\varepsilon < 1/2$,
        we obtain
        \begin{gather}
            \label{eq:l=m+2}
            \left| \DC( F^{\times l}(T) ) - \DC(T) \right|
            \le \tfrac{\varepsilon}{1-\varepsilon} \left(
                \DC(T) + 
                \tfrac{2^{m+1} \varepsilon} {(m+1)!(1-\varepsilon)} 
                \tfrac 1{\diam(\simp T)} 
            \right)
            \le C \varepsilon \diam(\simp T)^{-1} \,,
        \end{gather}
        where $C = C(m) > 0$.

        In case $ 2\le  l < m+2$ for 
        $T = (x_0,\ldots,x_{l-1}) \in \Sigma_1^l \without
        \Delta^l \Sigma_1$ we employ the assumption that 
        $F$ is bilipschitz, so that
        we can write
        \begin{align*}
            (\DC_l[\Sigma_2] \circ F^{\times l})(T) 
            &= \sup_{y_l,\ldots,y_{m+1} \in \Sigma_2} 
            \DC( F(x_0),\ldots,F(x_{l-1}), y_l,\ldots,y_{m+1})
            \\
            &= \sup_{x_l,\ldots,x_{m+1} \in \Sigma_1} 
            \DC( F(x_0),\ldots,F(x_{m+1}))
        \end{align*}
        and using \eqref{eq:l=m+2} 
        \begin{multline*}
            \sup_{x_l,\ldots,x_{m+1} \in \Sigma_1} \DC( F(x_0),\ldots,F(x_{m+1}))
            \\
            \le \sup_{x_l,\ldots,x_{m+1} \in \Sigma_1} \left(
                \DC(x_0,\ldots,x_{m+1})
                + C \varepsilon \diam(\{x_0,\ldots,x_{m+1}\})^{-1}
            \right)
            \\
            \le \DC_l[\Sigma_1](T) + C \varepsilon \diam(\simp T)^{-1} \,.
        \end{multline*}
        In the same way we obtain the lower bound
        \begin{gather*}
            \sup_{x_l,\ldots,x_{m+1} \in \Sigma_1} \DC( F(x_0),\ldots,F(x_{m+1}))
            \ge \DC_l[\Sigma_1](T) - C \varepsilon \diam(\simp T)^{-1} \,.
            \qedhere
        \end{gather*}
    \end{proof}

    A similar lemma does hold for the $\Rtp$ function.

    \begin{lem}\label{lem:Rtp-conv}
        Let $\Sigma_1, \Sigma_2 \in  \Lmn$, 
        $0 < \varepsilon < 1/2$. Assume $F :
        \Sigma_1 \to \R^n$ is bilipschitz and satisfies $F(\Sigma_1) = \Sigma_2$,
        $F(z) = z + G(z)$ for $z \in \Sigma_1$ and some $G : \Sigma_1 \to \R^n$
        having $\lip(G) \le \varepsilon$. Then for $\HM^m$-almost all 
        $x_1,y_1 \in \Sigma_1$, $x_1\not=y_1$, we have
        \begin{gather}
            \label{ineq5.8}
            \left| \frac{1}{\Rtp[\Sigma_1](x_1,y_1)} - 
                \frac{1}{\Rtp[\Sigma_2](F(x_1),F(y_1))}\right|
            \le \frac{C_{\ref{lem:Rtp-conv}} \, \varepsilon}{|x_1-y_1|} \,,
        \end{gather}
        where $C_{\ref{lem:Rtp-conv}} = C_{\ref{lem:Rtp-conv}}(m) > 0$.
    \end{lem}

    \begin{proof} Set $x_2=F(x_1)$, $y_2=F(y_1)$. Without loss of generality,  
        by the classic Rademacher theorem, 
        assume that $G$ is differentiable at $x_1$ and the tangent spaces to both
        manifolds, $U_i:=T_{x_i}\Sigma_i$, are well defined for $i=1,2$. Then,
        \[
        \frac{1}{\Rtp[\Sigma_i](x_i,y_i)} = \frac{2d_i}{|x_i-y_i|^2}\, , 
        \qquad i=1,2,
        \]
        where $d_i=\dist(y_i-x_i, U_i)$. By the triangle inequality,
        \begin{eqnarray}
            \left| \frac{1}{\Rtp[\Sigma_1](x_1,y_1)} 
                - \frac{1}{\Rtp[\Sigma_2](F(x_1),F(y_1))}\right| 
            & \le & \frac{2|d_1-d_2|}{|x_1-y_1|^2}\label{2termsRtp} \\
            & & {} + 2d_2 \left| \frac{1}{|x_1-y_1|^2} 
                - \frac{1}{|x_2-y_2|^2}\right|\, . \nonumber	
        \end{eqnarray}
        We shall show that each of these two terms is controlled by a constant 
        multiple of $\varepsilon |x_1-y_1|^{-1}$. Indeed, since $d_i\le |x_i-y_i|$ 
        and
        \[
        (1-\varepsilon)|x_1-y_1| \le |x_2-y_2| 
        = |F(x_1)-F(y_1)|\le (1+\varepsilon)|x_1-y_1|\, ,
        \]
        we easily estimate the second term on the right hand side of \eqref{2termsRtp},
        \begin{align}
            2d_2 \left| \frac{1}{|x_1-y_1|^2} - \frac{1}{|x_2-y_2|^2}\right|
            & \le \frac{2|x_2-y_2|}{|x_1-y_1|^2|x_2-y_2|^2} 
            \left| {|x_1-y_1|^2} - {|x_2-y_2|^2}\right| \nonumber\\
            & \le \frac{2}{(1-\varepsilon) |x_1-y_1|^3}
            \, \cdot\, \varepsilon (2+\varepsilon)|x_1-y_1|^2
            \label{1-Rtp}\\
            & < \frac{10\varepsilon}{|x_1-y_1|}
            \qquad\mbox{as $\varepsilon\in (0,\frac 12)$.}\nonumber
        \end{align}
        To estimate the first term on the right hand side of \eqref{2termsRtp}, 
        it is enough to check that $|d_1-d_2|\le C\varepsilon |x_1-y_1|$. 
        Note that $d_i=\dist(x_i-y_i,U_i)=|(x_i-y_i)-(U_i)_\natural(x_i-y_i)|$, 
        so that
        \begin{align}
            |d_1-d_2| 
            &\le |(x_1-y_1)-(x_2-y_2)| 
            + |(U_1)_\natural(x_1-y_1)-(U_2)_\natural(x_2-y_2)|
            \nonumber \\
            & \le \varepsilon|x_1-y_1| 
            + \big|(U_1)_\natural\big( (x_1-y_1)-(x_2-y_2)\big)\big| 
            + |((U_1)_\natural-(U_2)_\natural)(x_2-y_2)|
            \label{2-Rtp}\\
            & \le 2\varepsilon|x_1-y_1| 
            + \|(U_1)_\natural-(U_2)_\natural\|\, \cdot\,  (1+\eps) |x_1-y_1| \, .
            \nonumber
        \end{align}
        By the assumption on $F$ and $x_1$, we have
        \[
        U_2=DF(x_1)(U_1)=(\mathrm{Id}+DG(x_1))(U_1)\, , 
        \qquad \|DG(x_1)\|\le \varepsilon\, 
        \]
        The estimate of the angle between $m$-planes, 
        see \cite[Prop.~2.5]{KStvdM-GAFA}, yields 
        $\|(U_1)_\natural-(U_2)_\natural\|\le C \varepsilon$ 
        for some constant $C=C(m)$, and the lemma follows.
    \end{proof}

    \subsection{Semicontinuity, compactness and existence of minimisers}
    \label{sec:5.2}

    We are now ready to give the proof of Theorem~\ref{thm:lsc} and
    Corollary~\ref{thm:minimizer}. We begin with lower semicontinuity which is
    crucial for the compactness of sublevel sets of geometric curvature 
    and the existence of energy minimisers in isotopy classes.

   \begin{rem*}  There are several more or less equivalent ways to phrase 
	the argument which yields semicontinuity. If the curvature
	integrand contains no supremum, then the aim can be achieved by 
	(a) localization, (b) parametrization of all integrals by the same 
	domain, (c) an application of Fatou's lemma which is enabled by 
	the $C^1$ convergence of the parameterisations. Such an argument is 
	presented in \cite[pp. 2297--2298]{SvdM11a}. However, even in the 
	case of curves the integrands involving a supremum are by no means 
	continuous (see e.g. \cite[Section~3]{SvdM07}) with respect 
	to $C^1$ convergence, which requires more 
	care, cf. e.g. the proof of Thm.~3 in \cite{SvdM07}. In order to avoid
	a case by case study of the appropriately understood lower 
	semicontinuity of all the integrands considered, we present 
	here a general argument which is streamlined so that all the cases can be 
	included into the same scheme.\footnote{The proof is flexible 
	enough to show that all the energies are in fact lower 
	semicontinous w.r.t.~the bilipschitz 
	convergence of submanifolds also in the case $p\le p_0$, when
	we have no $C^{1,\alpha}$--regularization effect.}

   \end{rem*}

    \begin{proof}[Proof of part (i) of Theorem~\ref{thm:lsc}]
        Fix $l \in \{1,2,\ldots, m+2\}$. For $j \in \N \cup \{0\}$ let $\Sigma_j
        \in \Amn(E,d)$. Assume that $\Sigma_j$ converges in Hausdorff distance
        to $\Sigma$ and, without loss of generality, such that
        \begin{equation}\label{limeqliminf}
            \lim_{j\to\infty}\E(\Sigma_j)=\liminf_{j\to\infty}\E(\Sigma_j).
        \end{equation}
        Hence for some fixed $x\in\Sigma$ we find a sequence of points
        $x_j\in\Sigma_j$ such that $x_j\to x$ as $j\to\infty$, so that the
        shifted submanifolds $\tilde{\Sigma}_j:= \Sigma_j-x_j$ converge in
        Hausdorff-distance to $\tilde{\Sigma}:= \Sigma-x.$ Hence
        $0\in\tilde{\Sigma}$ and $0\in\tilde{\Sigma}_j$, and by translation
        invariance of the geometric curvature energies,
        $\E(\tilde{\Sigma}_j)=\E(\Sigma_j)\le E$ for all $j\in\N.$ Thus, by the
        Regularity Theorem, $\tilde{\Sigma}_j \in \Cmn(R,L,d)$, for
        appropriate $R,L$ given by \eqref{RLfromEp} depending only on the fixed
        $p$ and on the uniform energy threshold $E$. Hence, by the compactness
        result in $\Cmn(R,L,d)$, Theorem \ref{thm:compactness}, we find a
        subsequence still denoted by $\tilde{\Sigma}_j$ and some submanifold
        $\tilde{\Sigma}_0\in\Cmn(R,L,d)$ such that
        $\tilde{\Sigma}_j\to\tilde{\Sigma}_0$ in $\Cmn$ (see Definition
        \ref{def:C1a-conv}), which immediately implies that
        $\tilde{\Sigma}=\tilde{\Sigma}_0$ is contained in $\Cmn(R,L,d)$, so that
        we can apply all results of Section 4 to $\tilde{\Sigma}_j$ and
        $\tilde{\Sigma}$, and, in addition, we may evaluate the energy at
        $\Sigma=\tilde{\Sigma}+x$ to obtain $\E(\Sigma)= \E(\tilde{\Sigma})$, so
        it is enough to establish
        $\E(\tilde{\Sigma})\le\lim_{j\to\infty}\E(\tilde{\Sigma}_j).$ To
        simplify notation, we identify $\tilde{\Sigma}$ with $\Sigma$, and
        $\tilde{\Sigma}_j$ with $\Sigma_j$ from now on.

        We may also assume that for each $j$
        \[
        \HD (\Sigma_j,\Sigma) < \rho_g, \qquad\mbox{where $\rho_g$ is 
          given by Lemma~\ref{lem:amb-diff}.}
        \]
        Now, for $j \in \N$ set $\mu_j := \HM^{ml} 
        \narrow
        \Sigma_j^{l}$, $\mu:= \HM^{ml}\narrow\Sigma^l,$ and let $J_j : 
        \R^n \to \R^n$ be the diffeomorphism
        constructed in Lemma~\ref{lem:amb-diff} such that $J_j(\Sigma) =
        \Sigma_j$. 

        Observe that, by Lemma~\ref{lem:amb-diff}, $\bilip(J_j) \le 1 +
        C_J \HD(\Sigma,\Sigma_j)^{\alpha/2}$. 
        Moreover, the restriction $F_j\colon=J_j\mid_{\Sigma}$ satisfies
        \begin{equation}
            \label{Lip-Gj} F_j(x) = x+G_j(x) 
            \quad\mbox{on $\Sigma$,}\qquad \lip(G_j)=: 
            \eps_j, \qquad \eps_j\to 0 
            \quad\mbox{as $j\to\infty$,}
        \end{equation}
        since by Lemma~\ref{lem:bilip-est} $\eps_j\le C_l 
        \HD(\Sigma_j,\Sigma)^{\alpha/2}$.        

		 Now, the reader who is not overly keen to follow 
		all the technical details can skim the lines from here 
		to \eqref{pre-lsc}, thinking that, roughly, Step~1 below 
		involves the reparameterisations and the control of Jacobians,
		Step~2 is the place where `Fatou does the job', and Step~3 is the 
		price to pay for the flexibility of the argument (in the case where
		the integrand contains too many suprema).

        \medskip\noindent\emph{Step 1 (fixing the domain of integration).} 
        We shall
        first check that if $\F \in \{ \E_p^l, \TP_p, \TP_p^G \}$ is one of the
        energies considered, then -- in order to check 
        that $\F(\Sigma)\le \liminf \F(\Sigma_j)$ -- we can consider 
        the limes inferior of a sequence of
        integrals
        over a \emph{fixed} domain $\Sigma$, with appropriately 
        perturbed integrands.
        (In our application we could write actual limits because of \eqref{limeqliminf},
        but that does not  affect any of the following arguments.)

        Indeed, in each of the cases considered we have
        \begin{equation}
            \label{general-F} 
            \F(\Sigma)=\int_{\Sigma^l} (K_\Sigma)^p\, d\HM^{ml}\,  
            \qquad\mbox{for an appropriate integrand}
            \quad K_\Sigma\colon \Sigma^l\to [0,\infty ]\, .
        \end{equation}
        (To fix the ideas, we assume in all cases 
        $K_\Sigma\equiv\infty$ on $\Delta^l\Sigma$; 
        this does not affect the value of $\F$ as 
        $\HM^{ml}(\Delta^l\Sigma) = 0.$) 
        Thus, changing the variables and using Lemma~\ref{lem:density}, 
        for a sequence of
        $\Sigma_j$'s with $\sup_j\F(\Sigma_j)$ finite we obtain
        \begin{eqnarray*}
            \left| \F(\Sigma_j) 
            - \int \big(K_{\Sigma_j} \circ J_j^{\times l}\big)^p\ d\mu \right|
            & = & \left| 
                \int\big( K_{\Sigma_j}\big)^p\ d\mu_j 
                - \int \big(K_{\Sigma_j} \circ J_j^{\times l}\big)^p\ d\mu 
            \right| \\ 
            & = & \left| \int \big( K_{\Sigma_j}\big)^p\ d\mu_j 
                - \int \big( K_{\Sigma_j}\big)^p\ d((J_{j}^{\times l})_*\mu) 
            \right| \\
            & = & \left|  
                \int \big( K_{\Sigma_j}\big)^p\ 
                \left(1 - 
                    D\big((J_{j}^{\times l})_{*}\mu,\mu_j,\cdot)
                \right)\ d\mu_j  \right|\\
            & \le & C \left(\sup_j\F(\Sigma_j) \right)
            \HD(\Sigma_j,\Sigma)^{\alpha/2} 
            \xrightarrow{j \to \infty} 0 \,.
        \end{eqnarray*}
        The last inequality follows from the fact that by
        Lemma~\ref{lem:density} applied to $\nu=(J_{j}^{\times l})_*\mu$ and
        $\mu=\mu_j$ we have the density estimate
        \[
        \frac 1{\big(\lip J_j\big)^{ml}}\le 
        D\big((J_{j}^{\times l})_{*}\mu,\mu_j,\cdot)\le 
        \big(\lip J_j^{-1}\big)^{ml}
        \]
        Therefore, as $\bilip(J_j) \le 1 +
        C_J \HD(\Sigma,\Sigma_j)^{\alpha/2}\to 1$, we have
        \begin{equation}
            \left|1 - 
                D\big((J_{j}^{\times l})_{*}\mu,\mu_j,\cdot)
            \right|\le C 
            \HD(\Sigma_j,\Sigma)^{\alpha/2}\, .
        \end{equation}
        All this yields
        \begin{gather}
            \label{eq:push-energy}
            \liminf_{j \to \infty}  \int
            \big(K_{\Sigma_j} \circ J_j^{\times l}\big)^p\ d\mu
            = \liminf_{j \to \infty} \F(\Sigma_j) \,.
        \end{gather}
        Below, we work with the left hand side of \eqref{eq:push-energy}.

        \medskip\noindent\emph{Step 2 (energies with at least two integrals 
          over the manifold).} Suppose now that $l\ge 2$. If $\F=\E^l_p$, $l\ge 2$, 
        with the integrand
        \[
        K_{\Sigma} = \DC_l[\Sigma]\colon \Sigma^l\to [0,\infty)
        \]
        being the discrete curvature from Definition~\ref{def:DC_l}, then, 
        by Lemma~\ref{lem:DC-conv}, we have
        \begin{equation}
            \label{KK:estimate}
            K_{\Sigma}(T) \le K_{\Sigma_j}\big(F_j^{\times l} (T)\big) 
            +\frac{C(m)\, \eps_j}{\diam T}, 
            \qquad T\in \Sigma^l\without \Delta^l \Sigma\, .
        \end{equation}
        If on the other hand $\F=\TP_p$ (so that $l=2$), and
        \[
        K_{\Sigma}= \frac{1}{\Rtp[\Sigma]}\colon \Sigma\times \Sigma\to [0,\infty], 
        \]
        with the tangent--point radius $\Rtp$ defined  
        on $\Sigma^2\without\Delta^2\Sigma$ by \eqref{def:Rtp},
        then inequality \eqref{KK:estimate} holds by Lemma~\ref{lem:Rtp-conv}. 
        Thus, in both cases we can use \eqref{KK:estimate} to write, for 
        a fixed $T\in \Sigma^l\without \Delta^l \Sigma$,
        \begin{eqnarray*}
            K_{\Sigma}(T)^p & \le & \liminf_{j\to\infty}\left(
                K_{\Sigma_j}\big(F_j^{\times l} (T)\big) +
                \frac{C(m)\, \eps_j}{\diam T}\right)^p \\
            & = & \liminf_{j\to\infty}
            K_{\Sigma_j}\big(J_{j}^{\times l} (T)\big)^p\, 
            \qquad\mbox{(as $J_j=F_j$ on $\Sigma $).}
        \end{eqnarray*}
        Since $\HM^{ml}(\Delta^l \Sigma) = 0$, we can now ingrate both sides 
        w.r.t.\ $\mu$ and invoke Fatou's lemma (cf.~\cite[2.4.6]{Fed69}) to obtain
        \begin{eqnarray*}
            \F(\Sigma) =\int K_{\Sigma}(T)^p\, d\mu 
            & \le & \int \liminf_{j\to\infty}\, 
            K_{\Sigma_j} \big(J_j^{\times l} (T)\big)^p\, \, d\mu\\
            & \le &  \liminf_{j\to\infty}\,
            \int K_{\Sigma_j}\big(J_j^{\times l} (T)\big)^p\, \, d\mu\\
            & = & \liminf_{j \to \infty} \F(\Sigma_j)
        \end{eqnarray*}
        by \eqref{eq:push-energy}. This concludes the proof of Theorem~1 
        for $\F=\E^l_p$ with $l\ge 2$ and for $\F=\TP_p$.

        \medskip\noindent\emph{Step 3 (energies with a single integral).} The case
        $l=1$, i.e. when $\F=\E^1_p$, resp. $\F=\TP^G_p$, needs a separate treatment.
        We shall now work with the auxiliary integrands
        \[
        K_\Sigma\colon \Sigma\times\Sigma\to [0,\infty)\, ,
        \]
        using $K_\Sigma=\DC_2[\Sigma]$ for $\F=\E^1_p$, resp. $K_\Sigma = 
        1/\Rtp[\Sigma]$ for $\F=\TP^G_p$.

        The argument from Step~2 does not work here, as for $l=1$ we deal with
        simplices $T$ that degenerate to one point, and \eqref{KK:estimate} would yield
        nothing. To avoid this problem, we remove a~small neighbourhood of the
        diagonal, and pass to the limit twice. Here are the details.

        For a fixed $s\in \N$, set
        \begin{equation}
            \label{def:K-sup-hole}
            K_{\Sigma,s} (x) = \sup_{\substack{y\in \Sigma\\|y-x|\ge 1/s}}
            K_\Sigma(x,y)\, , \qquad x\in \Sigma\, .
        \end{equation}
        Define
        \[
        \F_s(\Sigma)=\int_{\Sigma} 	K_{\Sigma,s}^p\, d\HM^m\, .
        \]
        Note that $0\le K_{\Sigma,1}^p\le K_{\Sigma,2}^p\le\ldots$, $K_{\Sigma,s}^p\nearrow K_\Sigma^{p}$ as $s\to\infty$,
        so that  by the monotone convergence
        theorem, we have in each of the two cases ($\F=\E^1_p$ or $\F=\TP_p^G$) 
        that are being considered 

        \begin{equation}
            \label{mono-Fs}
            \F(\Sigma) =\sup_{s\in\N} \F_s(\Sigma) = \lim_{s\to\infty}\F_s(\Sigma)\, .
        \end{equation}
        Repeating Step~1 for each of the $\F_s$, we obtain
        \begin{equation}
            \label{energy-push-Fs}
            \liminf_{j \to \infty}  \int
            \big(K_{\Sigma_j, s} \circ J_j\big)^p\ d\mu
            = \liminf_{j \to \infty} \F_s (\Sigma_j) \,.
        \end{equation}
        Rewriting~\eqref{KK:estimate} for $l=2$, $T=(x,y)$, $x\not=y\in \Sigma$, 
        for the auxiliary integrands $K_\Sigma$, we obtain
        \begin{equation}
            \label{KK'}
            K_{\Sigma}(x,y) \le K_{\Sigma_j}\big(F_j(x), F_j(y) \big) 
            +\frac{C(m)\, \eps_j}{|x-y|}\, .
        \end{equation}
        We shall use this estimate for $s$ fixed and $j>1$ so large that 
        $\eps_j< \frac 1{s+1}$ (keep in mind that $\eps_j\to 0$ as $j\to\infty$). 
        Then, for points $x,y\in \Sigma$ with $|x-y|\ge \frac 1s$, we have
        \[
        |F_j(x)-F_j(y)|\ge (1-\eps_j) |x-y| \ge 
        \left(1-\frac 1{s+1}\right)\frac 1s = \frac 1{s+1},
        \]
        and upon taking the suprema of both sides of \eqref{KK'} with respect 
        to $y\in \Sigma$, $|x-y|\ge \frac 1s$, we obtain
        \begin{equation}
            \label{KKsingle}
            K_{\Sigma,s}(x)\le K_{\Sigma_j,s+1} (F_j(x)) + C(m)s\cdot \eps_j\, .
        \end{equation}
        Thus, for each $x\in \Sigma$,
        \[
        K_{\Sigma,s}(x)^p\le \liminf_{j\to\infty}  K_{\Sigma_j,s+1} (F_j(x))^p\, .
        \]
        Integration and Fatou's lemma yield now
        \begin{gather}
            \begin{aligned}
                \F_s(\Sigma) 
                &\le \liminf_{j\to\infty}\int K_{\Sigma_j,s+1} (F_j(x))^p\, d\mu \\
                &\stackrel{\eqref{energy-push-Fs}}=
                \liminf_{j \to \infty} \F_s (\Sigma_j)  \label{pre-lsc}
                 \le \liminf_{j \to \infty} \F (\Sigma_j) \,, 
            \end{aligned}
        \end{gather}
        as $\F_s\le \F$ for all $s\in \N$. Upon taking the supremum 
        of the left-hand sides with respect to
        $s\in \N$, in light of \eqref{mono-Fs}, we conclude 
        the proof for $\F=\E^1_p$ and for $\F=\TP_p^G$.
    \end{proof}

    \begin{proof}[Proof of part (ii) of Theorem~\ref{thm:lsc}]
         By the Regularity Theorem, 
        $\Sigma_j\in\Cmn(R,L,d)$ for all $j\in\N$, where the parameters
        $R,L$ are given by \eqref{RLfromEp} and do not depend on $j$. Thus,
        Theorem \ref{thm:compactness} implies that there is a subsequence (still
        denoted by $\Sigma_j$) and a submanifold $\Sigma\in\Cmn(R,L,d)$, such
        that $\Sigma_j\to\Sigma$ in $\Cmn$, i.e., in the sense of Definition
        \ref{def:C1a-conv}, which implies in particular that $\Sigma_j\to
        \Sigma$ in Hausdorff-distance, that $\diam\Sigma\le d$, and that $\Sigma\in\Lmn$
        Therefore,  we
        may evaluate the energy $\E$ on $\Sigma$. Part (i) implies that
        $\Sigma\in\Amn(E,d)$.
    \end{proof}

    \begin{proof}[Proof of Corollary \ref{thm:minimizer}]
        Notice that the class
        $$
        \mathcal{C}:=\{\Sigma\in\Amn(E,d):
        \textnormal{$\Sigma$ is ambient isotopic to $M_0$}\}
        $$
        contains the reference manifold $M_0$, so that we can find a minimising
        sequence $(\Sigma_j)_j\subset\mathcal{C}$ with
        $\E(\Sigma_j)\to\inf_\mathcal{C}\E$ as $j\to\infty.$ The uniform energy
        bound $E$ implies by the Regularity Theorem that
        $\Sigma_j\in\Cmn(R,L,d)$ for all $j\in\N$, where the parameters $R,L$
        depend only on the energy bound and on the integrability parameter $p$,
        so that we can apply the improved compactness result, Theorem
        \ref{thm:compactness}, to deduce the existence of a subsequence (still
        denoted by $\Sigma_j$) that converges to a limit submanifold
        $\Sigma_0\in\Cmn(R,L,d)$ in $\Cmn$. Then the isotopy result, Theorem
        \ref{thm:diff}, implies that $\Sigma_j$ is ambient isotopic to
        $\Sigma_0$ for $j$ sufficiently large, which implies that
        $\Sigma_0\in\mathcal{C}$.  Part (i) of Theorem \ref{thm:lsc} finally
        leads to
        $$
        \inf_\mathcal{C}\E\le\E(\Sigma_0)\le\liminf_{j\to\infty}\E(\Sigma_j)=\inf_\mathcal{C}\E,
        $$
        which concludes the proof.
    \end{proof}

    \section{Bounds on the number of diffeomorphism  and isotopy types}

    \label{sec:finiteness}

    \begin{proof}[Proof of Theorem~\ref{thm:finiteness}]

        Fix an energy $\E\in\{\E^l_p,\TP_p,\TP_p^G\}$    
        and a $p>p_0(\E)$. Let $\Sigma\in
        \Amn(E,d)$ be a manifold with controlled energy and diameter, cf.~\eqref{Amn}.
        Translating $\Sigma$ if necessary, we have $\Sigma \in \Cmn(R,L,d)$ by the
        Regularity Theorem, where the parameters $R$ and $L$ depend only 
        on $E$ and $p$. Thus, $\Sigma \subset [-d,d]^n$.

        Now fix $\varepsilon:=1/200$ and let 
        $\mathfrak{F} = \{ Q_1, Q_2,\ldots Q_N\}$ be a minimal collection
        of~closed cubes of edge $e:=\rho_G/(2\sqrt n)$ covering $[-d,d]^n$;
        here $\rho_G>0$ is the constant of Lemma~\ref{lem:bilip-est} and
        Corollary~\ref{cor:isotopy} for $\varepsilon=1/200$. 
        Notice that the dependence of
        $\rho_G$ on the Lipschitz constant of an $\varepsilon$-normal 
        map for $\Sigma$ boils down to
        $\rho_G=\rho_G(R,L,\alpha,m,n)$ since we have fixed $\varepsilon$; 
        see Lemma \ref{lem:mollify}.
        Clearly, the cardinality of $\mathfrak{F}$
        satisfies
        \begin{equation}
            \label{eq:card-F}
            \HM^0(\mathfrak{F}) = N\le  k^n, \qquad\mbox{with}\qquad 
            k= \Big\lceil\frac{2d\sqrt{n}}{{\rho_G}}\Big\rceil\, .
        \end{equation}

        Following Durumeric \cite[Section~5]{Dur02} (see also the remarks in Peters
        \cite[Section~5]{Pet84}), to each $\Sigma \in \Cmn(R,L,d)$ 
        we assign the subset
        $P(\Sigma)\subset \mathfrak{F}$ which consists of those cubes $Q \in
        \mathfrak{F}$ that intersect $\Sigma$,
        i.e.
        \begin{gather*}
            Q \in P(\Sigma) \subset \mathfrak{F} 
            \iff          Q \cap \Sigma \ne \emptyset \,.
        \end{gather*}

        If~$P(\Sigma_1) = P(\Sigma_2)$, then obviously $\HD(\Sigma_1,\Sigma_2)$
        does not exceed the diameter of all the $Q_i$ which equals
        $e\sqrt{n}= \rho_G/2$. Hence, by Corollary~\ref{cor:isotopy},
        $\Sigma_1$ and $\Sigma_2$ are ambient isotopic. Therefore, the number of
        distinct isotopy classes of manifolds $\Sigma \in \Amn(E,d)$ is not larger
        than $K=2^N$, the number of all subsets of $\mathfrak{F}$.
        
        Finally, since $\rho_G = \rho_G(R,L,\alpha,m,n)$ depends only on
        $R$ and $L$ which are given, for a particular energy $\E$ and an upper
        energy bound $E$, by \eqref{RLfromEp} in the Regularity Theorem, and on $\alpha
        = 1- p_0(\E)/p$,  it is clear
        that $K=K(E,d,m,n,p)$.
    \end{proof}

    \begin{rem} 
        \label{rem:connected}	
        The estimate $K\le 2^N$ is obviously not optimal for \emph{connected}
        manifolds. If $\Sigma$ is connected, then the union of all cubes in $P(\Sigma)$
        is connected, too; thus, one only needs to count those subsets of
        $\mathfrak{F}$ which have connected unions. (For $n=1$ there are $2^k$ subsets
        of the family of intervals and only $O(k^2)$ connected subsets!) One can prove
        \cite{PiliPili} that the number $K_{\text{con}}$ of such subsets of
        $\mathfrak{F}$ satisfies
        \begin{equation}
            \label{eq:pilipczuk}
            (2-a(n))^N\le K_{\text{con}}\le (2-b(n))^N\, ,\qquad N\le k^n\, ,
        \end{equation}
        where $a,b\colon \N\to (0,\infty)$ are positive (but go to zero as 
        the dimension $n\to
        \infty$). 

        Here is the gist of the argument. Assume for the sake of simplicity that the
        closed cubes in $\mathfrak{F}$ have disjoint interiors, and that 
        $k=2d/\eps$ is
        divisible by $3$.

        To obtain the upper bound, divide $[-d,d]^n$ into larger cubes $\widetilde Q$,
        each of them consisting of $3^n$ of the initial $Q_j$'s. In each $\widetilde
        Q_0$, one $Q_{j_0}$ --- call it \emph{central} --- contains the center of
        $\widetilde Q_0$ and is separated from other $\widetilde Q_i$'s by a layer of
        small $Q_j$'s. Now, if for a connected $\Sigma$ the subset $P(\Sigma)$ contains
        one of the small \emph{central} $Q_j$'s, then it must contain at least one small
        cube from the layer around this $Q_j$ unless  the whole
        $P(\Sigma)=\{Q_j\}$. This limits the number of possible choices of
        $P(\Sigma)$ and yields the upper bound in \eqref{eq:pilipczuk}.

        To obtain the lower bound, one constructs a specific family of subsets of
        $\mathfrak{F}$ with connected unions, e.g. as follows. Let $X\subset
        \mathfrak{F}$ consist of $k^{n-1}$ little cubes adjacent to a fixed
        $(n-1)$-dimensional face of $[-d,d]^n$ (think of it as the bottom face) and
        of $(k/3)^{n-1}$ thin vertical,  symmetrically placed  `towers'
        standing on the bottom face, each of these towers consisting of $k-1$ little
        cubes and reaching to the top of the whole box $[-d,d]^n$. Thus, 
        \[ \mbox{the number of cubes in $X$} = k^{n-1}+
        \frac{1}{3^{n-1}}k^{n-1}(k-1)\, . \] 
        Note that adding to $X$ \emph{any}
        subset of $\mathfrak{F}\without X$, we obtain a family of cubes with
        connected union (because each of the cubes in $\mathfrak{F}\without X$
        touches one of the towers in $X$). From this, one obtains the lower bound
        for $K_{\text{con}}$.

        It is however clear that \eqref{eq:pilipczuk} does not take into account any
        \emph{global\/} information on $\Sigma$ (e.g., it does not exclude those
        subsets of $\mathfrak{F}$ that are too small or too flat to cover a $\Sigma$
        with $\E(\Sigma)\le E$). \end{rem}

    \begin{rem}[\textbf{Explicit bounds}] 
        \label{rem:explicit}
        One can track an estimate of $N$ (the number of little cubes in
        $\mathfrak{F}$) in terms of the energy bounds $\E(\Sigma)\le E$ etc. as
        follows.

        \begin{enumerate}
            \renewcommand{\labelenumi}{(\roman{enumi})} 
        \item Note that $R$ and $L$ given
            by \eqref{RLfromEp} in the Regularity Theorem 
            satisfy $R^\alpha L=c(m,n,l,p)$.
        \item Lemma~\ref{lem:hd-angle}
            with $A=4$ yields $C_{\text{{\rm ang}}}(L,4)=257L+8$. 
        \item
            Lemma~\ref{lem:mollify} gives the Lipschitz constant of 
            the $\eps$-normal map,
            cf. \eqref{def:lipnorm} and Remark~\ref{rem:lipnorm}. 
        \item The number $\rho_G$ for fixed $\eps=1/200$ emerges
            in~Lemma~\ref{lem:bilip-est} and Corollary~\ref{cor:isotopy}, via
            the constant $C_l$; a combination of
            \eqref{eq:G-lip}--\eqref{def:rhoG} with (i) above shows that we can
            have, e.g.,
            \begin{equation*} 
                \label{est:1overrho}
                \begin{split} 
                    \frac{1}{\rho_G} = C_l^{2/\alpha} 
                    & \le c(m,n,l,p) (L+1)^{2(2\alpha+1)/\alpha^2}\\ 
                    & \le \tilde c(m,n,l,p) \big(E^{1/p}+1\big)^{2(2\alpha+1)/\alpha^2},
                    \qquad \alpha = 1-\frac{p_0(\E)}p\, . 
                \end{split} 
            \end{equation*} 
        \end{enumerate} 
        Plugging the last estimate into $\log\log K\le \log N\le n\log
        \big(1+2d\sqrt{n}/\rho_G\big)$, one obtains the bound
        \eqref{est:K-finiteness} stated in the Introduction.
    \end{rem}

  \section*{Acknowledgments}

  Parts of this work were completed between 2012~and~2015, while the authors, in
  various combinations, were visiting the Albert Einstein Institute in Golm and
  RWTH Aachen University. They are all grateful for the hospitality and for
  excellent working conditions.

  The first author was partially supported by the Foundation for Polish Science.
  The first and second author were partially supported by NCN Grant no.
  2013/10/M/ST1/00416 \emph{Geometric curvature energies for subsets of the
    Euclidean space.} The third author's work is partially funded by the
  Excellence Initiative of the German federal and state governments.

  \bibliography{refs}{}
  \bibliographystyle{acm}
\end{document}